\documentclass{amsart}
\pdfoutput=1 
\usepackage[cmtip,all]{xy}
\newdir{ >}{{}*!/-5pt/@{>}} 
\usepackage{array}
\usepackage{longtable}
\usepackage{lscape} 
\usepackage{graphicx} 
\usepackage[alphabetic,nobysame]{amsrefs}

\title[Local cohomology for topological modular forms]
    {The local cohomology spectral sequence \\ for topological modular forms}
\author{Robert Bruner}
\address{Department of Mathematics, Wayne State University, USA}
\email{robert.bruner@wayne.edu}
\author{John Greenlees}
\address{Mathematics Institute, University of Warwick, UK}
\email{John.Greenlees@warwick.ac.uk}
\author{John Rognes}
\address{Department of Mathematics, University of Oslo, Norway }
\email{rognes@math.uio.no}

\thanks{The first and second authors are grateful to the organisers of
the 2017 Ecuador Topology Conference, where initial discussions on this
project began in front of the Gal{\'a}pagos finches.}

\newtheorem{theorem}{Theorem}[section]
\newtheorem{proposition}[theorem]{Proposition}
\newtheorem{lemma}[theorem]{Lemma}
\newtheorem{corollary}[theorem]{Corollary}
\theoremstyle{definition}
\newtheorem{definition}[theorem]{Definition}
\newtheorem{notation}[theorem]{Notation}
\theoremstyle{remark}
\newtheorem{remark}[theorem]{Remark}
\newtheorem{example}[theorem]{Example}

\numberwithin{equation}{section}
\numberwithin{figure}{section}
\numberwithin{table}{section}

\DeclareMathOperator{\Cell}{Cell}
\DeclareMathOperator{\Ext}{Ext}
\DeclareMathOperator{\Hom}{Hom}
\DeclareMathOperator*{\hocolim}{hocolim}
\newcommand{\cE}{\mathcal{E}}
\newcommand{\bF}{\mathbb{F}}
\newcommand{\bQ}{\mathbb{Q}}
\newcommand{\bZ}{\mathbb{Z}}
\newcommand{\fn}{\mathfrak{n}}
\newcommand{\tmf}{tm\!f}
\newcommand{\Tmf}{Tm\!f}
\newcommand{\longto}{\longrightarrow}

\renewcommand{\:}{\colon}
\newcommand{\<}{\langle}
\renewcommand{\>}{\rangle}

\begin{document}

\begin{abstract}
We discuss proofs of local cohomology theorems for topological modular
forms, based on Mahowald--Rezk duality and on Gorenstein duality, and then
make the associated local cohomology spectral sequences explicit,
including their differential patterns and hidden extensions.
\end{abstract}


\maketitle

\section{Introduction}
Several interesting ring spectra satisfy duality theorems relating local
cohomology to Anderson or Brown--Comenetz duals.  The algebraic precursor
of these results is due to Grothendieck~\cite{Har67}, and is a local
analogue of Serre's projective duality theorem.  In each case there is a
covariant local cohomology spectral sequence converging to the homotopy
of the local cohomology spectrum, and a contravariant $\Ext$ spectral
sequence computing the homotopy of a functionally dual spectrum.  As a
consequence of self-dualities intrinsic to the ring spectra in question,
the results of the two calculations agree up to a shift in grading, in
spite of their opposite variances.  It is the purpose of this paper to
make these self-dualities explicit for the connective topological modular
forms spectrum.  Figures~\ref{fig:A-duality} and~\ref{fig:Ap3-duality}
depict the $2$- and $3$-complete dualities, respectively.  A reader
wondering if this paper is of interest might glance at these figures;
they do require explanation (given below), but the structural patterns
are immediately and strikingly apparent in the pictures.  A reader new
to $\tmf$ might prefer to start with the simpler charts for $p=3$, as
preparation for the case of~$p=2$.  We also treat the much simpler case
of the connective real $K$-theory spectrum.

We work at one prime~$p$ at a time, write $ko = ko^\wedge_p$ and $ku
= ku^\wedge_p$ for the $p$-completed real and complex connective
topological $K$-theory spectra, and write $\tmf = \tmf^\wedge_p$
for the $p$-completed connective topological modular forms spectrum.
We also consider a $2$-complete spectrum $\tmf_1(3) = \tmf_1(3)^\wedge_2$
and a $3$-complete spectrum $\tmf_0(2) = \tmf_0(2)^\wedge_3$ related to
elliptic curves with $\Gamma_1(3)$ and $\Gamma_0(2)$ level structures,
respectively.  These are all commutative $S_p$-algebras, where $S_p =
S^\wedge_p$ denotes the $p$-completed sphere spectrum.  See~\cite{DFHH14}
and~\cite{BR21} for theoretical and computational background regarding
topological modular forms.
For $p=2$ there are Bott and Mahowald classes $B \in \pi_8(\tmf)$ and $M
\in \pi_{192}(\tmf)$ detected by the modular forms~$c_4$ and $\Delta^8$,
respectively.  The homotopy groups~$\pi_*(\tmf)$ for $0 \le * \le 192$
are shown in Figure~\ref{fig:pitmf}.  The red dots indicate $B$-power
torsion classes, and the entire picture repeats $M$-periodically.

For any commutative $S_p$-algebra $R$ and a choice of finitely generated
ideal $J = (x_1, \dots, x_d) \subset \pi_*(R)$, the local cohomology
spectrum $\Gamma_J R$ encapsulates the $J$-power torsion of $\pi_*(R)$,
together with its right derived functors.  There is a local cohomology
spectral sequence
$$
E_2^{s,t} = H^s_J(\pi_*(R))_t
	\Longrightarrow_s \pi_{t-s}(\Gamma_J R)
$$
(in Adams grading), which can be used to compute its homotopy.  For $R =
\tmf$ at $p=2$ and $J = (B, M)$ the spectral sequence collapses to a short
exact sequence
$$
0 \to H^2_{(B, M)}(\pi_*(\tmf))_{n+2}
	\longto \pi_n(\Gamma_{(B, M)} \tmf)
	\longto H^1_{(B, M)}(\pi_*(\tmf))_{n+1} \to 0
$$
in each topological degree~$n$, cf.~Figures~\ref{fig:GammaBN-ab}
and~\ref{fig:GammaBN-cd}, while for $J = (2, B, M)$ its $E_2$-term
is concentrated in filtration degrees $1 \le s \le 3$ and contains
nonzero $d_2$-differentials, cf.~Figures~\ref{fig:Gamma2BN-ab}
through~\ref{fig:Gamma2BN-gh}.

The $S_p$-module Anderson and Brown--Comenetz duals $I_{\bZ_p} R$ and $IR$
are defined as function spectra $F_{S_p}(R, I_{\bZ_p})$ and $F_{S_p}(R,
I)$, where $I_{\bZ_p}$ and $I$ are so designed that the associated
homotopy spectral sequences collapse to a short exact sequence
$$
0 \to \Ext_{\bZ_p}(\pi_{m-1}(R), \bZ_p)
	\longto \pi_{-m}(I_{\bZ_p} R)
	\longto \Hom_{\bZ_p}(\pi_m(R), \bZ_p) \to 0
$$
and an isomorphism $\pi_{-m}(IR) \cong \Hom_{\bZ_p}(\pi_m(R),
\bQ_p/\bZ_p)$, respectively.  The local cohomology duality theorems for
$\tmf$ at $p=2$ establish equivalences
\begin{gather}
\Gamma_{(B, M)} \tmf \simeq \Sigma^{-22} I_{\bZ_2}(\tmf)
	\label{eq:A-duality-tmf} \\
\Gamma_{(2, B, M)} \tmf \simeq \Sigma^{-23} I(\tmf) \,,
\end{gather}
which in particular imply that the covariantly defined $\pi_n(\Gamma_{(B,
M)} \tmf)$ and the contravariantly defined $\pi_{-m}(I_{\bZ_2} \tmf)$
are isomorphic for $n+m = -22$, and similarly that $\pi_n(\Gamma_{(2, B,
M)} \tmf)$ and $\Hom_{\bZ_2}(\pi_m(\tmf), \bQ_2/\bZ_2)$ are isomorphic
for $n+m = -23$.

Figure~\ref{fig:A-duality} illustrates $\pi_*(\tmf)$, $\pi_*(\Gamma_{(B,
M)} \tmf)$, $\pi_*(I_{\bZ_2}(\tmf))$ and the duality isomorphism,
up to a degree shift, between the latter two graded abelian groups.
More precisely, $\pi_*(\tmf)$ is isomorphic to the `basic block'
$\pi_*(N)$ shown in the first part of the figure, tensored with $\bZ[M]$.
The local cohomology $\pi_*(\Gamma_{(B, M)} \tmf)$ and the Anderson
dual $\pi_*(I_{\bZ_2} \tmf)$ are isomorphic to $\pi_*(\Gamma_B N)$ and
$\pi_*(I_{\bZ_2} N)$ tensored with $\bZ[M]/M^\infty$, respectively, up
to appropriate degree shifts.  The second part of the figure shows the
covariantly defined $\pi_*(\Gamma_B N)$, while the third part shows the
contravariantly defined $\pi_*(\Sigma^{171} I_{\bZ_2} N)$.  The nearly
mirror symmetric isomorphism between the latter two graded abelian groups
thus exhibits the duality isomorphism~\eqref{eq:A-duality-tmf}, in its
`basic block' form $\Gamma_B N \simeq \Sigma^{171} I_{\bZ_2} N$.  The same
structure is presented in greater detail in Figures~\ref{fig:pitmf},
\ref{fig:GammaBN-ab} and~\ref{fig:GammaBN-cd}, where the additive
generators are named and the module action by $\eta$, $\nu$ (and $B$) is
shown by lines increasing the topological degree by~$1$, $3$ (and~$8$),
respectively, but the distinctive symmetry implied by the duality theorem
is most easily seen in the first figure.

Several different approaches lead to proofs of such local cohomology
duality theorems.  For fp-spectra~$X$, i.e., bounded below and
$p$-complete spectra whose mod~$p$ cohomology is finitely presented
as a module over the Steenrod algebra, Mahowald and Rezk~\cite{MR99}
determined the cohomology of the Brown--Comenetz dual of the finite
$E(n)$-acyclisation $C^f_n X$.  In many cases $C^f_n R$ is a local
cohomology spectrum, and we show in Theorem~\ref{thm:MRduality} how this
leads to duality theorems for $R = ko$ at all primes, and for $R = \tmf$
at $p=2$ and~$p=3$.  This strategy ties nicely in with chromatic homotopy
theory.

Next, Dwyer, Greenlees and Iyengar~\cite{DGI06} showed that for augmented
ring spectra $R \to k$ such that $\pi_*(R) \to k$ is algebraically
Gorenstein, the $k$-cellularisation $\Cell_k R$ is often equivalent to
a suspension of $IR$ or $I_{\bZ_p} R$, for $k = \bF_p$ or $k = \bZ_p$,
respectively.  We use descent methods to extend this to ring spectra
with a good map to an augmented ring spectrum $T \to k$ satisfying the
algebraic Gorenstein property, e.g., with $\pi_*(T) = k[x_1, \dots,
x_d]$ polynomial over~$k$.  Moreover, $\Cell_k R$ is in many cases a
local cohomology spectrum, and we show in Theorem~\ref{thm:Gorduality}
how this leads to duality theorems for $R = ko$ and $R = \tmf$, at all
primes~$p$.  This strategy emphasises commutative algebra inspired by
algebraic geometry.

There is a growing list~\cite{Gre93}, \cite{Gre95}, \cite{BG97a},
\cite{BG97b}, \cite{BG03}, \cite{DGI06}, \cite{BG08}, \cite{BG10},
\cite{Gre16}, \cite{GM17}, \cite{GS18} of examples known to enjoy
Gorenstein duality. Many of them are of equivariant origin, or have $R
= C^*(X)$ for a manifold~$X$, or arise from Serre duality in derived
algebraic geometry.  For instance, Stojanoska~\cite{Sto12}, \cite{Sto14}
used Galois descent and homotopy fixed point spectral sequences to deduce
Anderson self-duality for $\Tmf$ from its covers $\Tmf(2)$ (at $p=3$)
and~$\Tmf(3)$ (at $p=2$).  More recently, Bruner and Rognes~\cite{BR21}
used a variant of the descent arguments above to directly deduce
local cohomology duality theorems for~$\tmf$ at $p=2$ and~$p=3$ from
similar theorems for $\tmf_1(3)$ and~$\tmf_0(2)$, respectively.
We summarise these results in Theorems~\ref{thm:TheradualityA}
and~\ref{thm:TheradualityBC}.

The main goal of this paper is to draw on the Hopkins--Mahowald
calculation of $\pi_*(\tmf)$, as presented in~\cite{BR21}, to make
the local cohomology spectral sequences for $R = \tmf$ at $p=2$
and at $p=3$, with $J = (B, M)$ and $J = (p, B, M)$, completely
explicit.  In order to determine the differential patterns and some
of the hidden (filtration-shifting) multiplicative extensions in
these spectral sequences, we rely on the local cohomology duality
theorems to identify the abutments with shifts of the Anderson
and Brown--Comenetz duals of $\tmf$.  This is carried out in
Subsections~\ref{subsec:BMloccohtmf} and~\ref{subsec:2BMloccohtmf}
for $p=2$, and in Subsections~\ref{subsec:BHloccohtmf}
and~\ref{subsec:3BHloccohtmf} for $p=3$.  See also the explanations in
Subsection~\ref{subsec:charts} of the graphical conventions used in
the charts.  As a warm-up we first go through the corresponding, but
far simpler, calculations for $R = ko$ at $p=2$ in Section~\ref{sec:ko}.

\section{Colocalisations}

\subsection{Small and proxy-small}
The stable homotopy category of spectra and, more generally, the homotopy
category of $R$-modules for any fixed $S$-algebra~$R$, are prototypical
triangulated categories.  We keep the terminology from~\cite{DGI06}*{3.15,
4.6}: A full subcategory of a triangulated category is \emph{thick}
if it is closed under equivalences, integer suspensions, cofibres and
retracts, and it is \emph{localising} if it is furthermore closed under
coproducts.  An object~$A$ \emph{finitely builds} an object~$X$ if $X$
lies in the thick subcategory generated by~$A$, and more generally $A$
\emph{builds}~$X$ if $X$ lies in the localising subcategory generated
by~$A$.  An $R$-module $A$ is \emph{small} if it is finitely built
from~$R$, and more generally it is \emph{proxy-small} if there is a
small $R$-module~$K$ that both builds~$A$ and is finitely built by~$A$.

\subsection{Acyclisation}
We recall three related colocalisations.
First,
for any spectrum~$X$ and integer $n\ge0$ let $C^f_n X \to X$ denote
its \emph{finite $E(n)$-acycl(ic)isation}, as defined by Miller
\cite{Mil92}*{\S2}.  Here $E(n)$ denotes the $n$-th $p$-local
Johnson--Wilson spectrum, with coefficient ring $\pi_* E(n) =
\bZ_{(p)}[v_1, \dots, v_{n-1}, v_n^{\pm1}]$.  The map $F(A, C^f_n X)
\to F(A, X)$ is an equivalence for each finite $E(n)$-acyclic~$A$,
and $C^f_n X$ is built from finite $E(n)$-acyclic spectra.  There is
a natural equivalence $C^f_n X \simeq X \wedge C^f_n S$, so for any
$R$-module $M$ the spectrum $C^f_n M$ admits a natural $R$-module
structure.  A $p$-local finite spectrum has \emph{type~$n+1$} if it
is $E(n)$-acyclic but not $E(n+1)$-acyclic.  If $X$ is $p$-local,
then by Hopkins--Smith~\cite{HS98}*{Thm.~7} any one choice of a finite
spectrum~$A$ of type~$n+1$ suffices to build $C^f_n X$.  Inductively for
each $n\ge0$, Hovey and Strickland~\cite{HS99}*{Prop.~4.22} build a
cofinal tower of \emph{generalised Moore spectra} $S/I$ of type~$n+1$,
for suitable ideals $I = (p^{a_0}, v_1^{a_1}, \dots, v_{n-1}^{a_{n-1}},
v_n^{a_n})$, such that there are homotopy cofibre sequences
$$
\Sigma^{2(p^n-1)a_n} S/I' \overset{v_n^{a_n}}\longto S/I' \longto S/I
$$
with $I' = (p^{a_0}, v_1^{a_1}, \dots, v_{n-1}^{a_{n-1}})$.  Here $S/()
= S$ and $v_0 = p$.  By~\cite{HS99}*{Prop.~7.10(a)} there is a natural
equivalence $\hocolim_I F(S/I, X) \simeq C^f_n X$, where $S/I$ ranges
over this tower.

\subsection{Cellularisation} \label{subsec:cellularisation}
Second,
let $k$ and $M$ be $R$-modules.  The \emph{$k$-cellularisation} of~$M$
is the $R$-module map $\Cell_k M \to M$ such that $F_R(k, \Cell_k M) \to
F_R(k, M)$ is an equivalence, and such that $\Cell_k M$ is built from~$k$
in $R$-modules.  It can be realised as the cofibrant replacement in a
right Bousfield localisation of the stable model structure on $R$-modules
in symmetric spectra, cf.~\cite{Hir03}*{\S5.1, \S4.1}, hence always
exists.

\begin{lemma} \label{lem:mutuallybuild}
If two $R$-modules $k$ and~$\ell$ mutually build one another, then
$\Cell_k M \simeq \Cell_{\ell} M$ for all $R$-modules~$M$.
Conversely, if $\Cell_k M \simeq \Cell_{\ell} M$ for all~$M$, then $k$
and~$\ell$ mutually build one another.
\end{lemma}

\begin{proof}
If $k$ builds~$\ell$, then $F_R(\ell, \Cell_k M) \to F_R(\ell, M)$
is an equivalence.  If $\ell$ builds~$k$, then $\Cell_k M$ is built
from~$\ell$.  If both conditions hold, then $M \to \Cell_k M$ is the
$\ell$-cellularisation of~$M$.

If $\Cell_k \ell \simeq \Cell_{\ell} \ell = \ell$, then $k$ builds~$\ell$,
and if $k = \Cell_k k \simeq \Cell_\ell k$ then $\ell$ builds~$k$, so if
both hold then $k$ and~$\ell$ build one another.
\end{proof}

\begin{lemma} \label{lem:CfnCell}
If $A$ is a $p$-local finite spectrum of type~$n+1$, and $M$ is a
$p$-local $R$-module, then $C^f_n M \simeq \Cell_{R \wedge A} M$
as $R$-modules.
\end{lemma}

\begin{proof}
We know that $R$ builds~$M$ in $R$-modules, and $A$ builds $C^f_n S_{(p)}$
in $S$-modules, so $R \wedge A$ builds $M \wedge C^f_n S_{(p)} \simeq
C^f_n M$ in $R$-modules.  Moreover, $F_R(R \wedge A, C^f_n M) \to F_R(R
\wedge A, M)$ is an equivalence, since this the same map as $F(A, C^f_n M)
\to F(A, M)$.  Hence $C^f_n M$ is the $R \wedge A$-cellularisation of~$M$.
\end{proof}

Let $\cE = F_R(k, k)$ be the endomorphism $S$-algebra of the
$R$-module~$k$.  An $R$-module~$M$ is \emph{effectively constructible}
from~$k$ if the natural map
$$
F_R(k, M) \wedge_{\cE} k \longto M
$$
is an equivalence.  It is proved in~\cite{DGI06}*{Thm.~4.10} that, if $k$
is proxy-small, then this map always realises the $k$-cellularisation
of~$M$.  Hence $\Cell_k M$ is determined by the right $\cE$-module
structure on $F_R(k, M)$, for proxy-small~$k$.

\subsection{Local cohomology} \label{subsec:localcohom}
Third,
suppose that $R$ is a commutative $S$-algebra, and let $J = (x_1, \dots,
x_d)$ be a finitely generated ideal in the graded ring $\pi_*(R)$.  For
each $x \in \pi_*(R)$ define the \emph{$x$-power torsion spectrum}
$\Gamma_x R$ by the homotopy (co-)fibre sequence
$$
\Sigma^{-1} R[\frac{1}{x}] \overset{\alpha}\longto \Gamma_x R
	\overset{\beta}\longto R
	\overset{\gamma}\longto R[\frac{1}{x}] \,.
$$
For any $R$-module $M$ let
$$
\Gamma_J M = \Gamma_{x_1} R \wedge_R \dots \wedge_R \Gamma_{x_d} R \wedge_R M
$$
be the \emph{local cohomology spectrum}.  By~\cite{GM95}*{\S1, \S3},
this $R$-module only depends on the radical~$\sqrt{J}$ of the ideal~$J$.
The convolution product of the short filtrations $\alpha \: \Sigma^{-1}
R[1/x_i] \to \Gamma_{x_i} R$ for $1 \le i \le d$ leads to a length~$d$
decreasing filtration of $\Gamma_J M$, with subquotients
$$
F^s/F^{s+1} \simeq \bigvee_{1 \le i_1 < \dots < i_s \le d}
	\Sigma^{-1} R[\frac{1}{x_{i_1}}] \wedge_R \dots
		\wedge_R \Sigma^{-1} R[\frac{1}{x_{i_s}}] \wedge_R M \,.
$$
In Adams indexing, the associated spectral sequence has $E_1$-term
$$
E_1^{s,t} = \pi_{t-s}(F^s/F^{s+1})
	\cong \bigoplus_{1 \le i_1 < \dots < i_s \le d}
	\pi_t(M[\frac{1}{x_{i_1}}, \dots, \frac{1}{x_{i_s}}])
$$
for $0 \le s \le d$, with differentials $d_r \: E_r^{s,t} \to
E^{s+r,t+r-1}$.  The $d_1$-differentials are induced by the various
localisation maps $\gamma \: R \to R[1/x_i]$, and the cohomology of $(E_1,
d_1)$ defines the \emph{local cohomology groups} of the $\pi_*(R)$-module
$\pi_*(M)$, in the sense of Grothendieck~\cite{Har67}.  This construction
defines the \emph{local cohomology spectral sequence}
\begin{equation} \label{eq:lcss}
E_2^{s,t} = H^s_J(\pi_*(M))_t
	\Longrightarrow_s \pi_{t-s}(\Gamma_J M) \,,
\end{equation}
which is a strongly convergent $\pi_*(R)$-module spectral sequence,
cf.~\cite{GM95}*{(3.2)}.  As in the topological case, the local
cohomology groups $H^*_J(\pi_*(M))$ only depend on~$J$ through its
radical in $\pi_*(R)$, not on the explicit generators~$x_1, \dots, x_d$.
We emphasise that the local cohomology spectrum $\Gamma_J M$, and the
associated spectral sequence, are covariantly functorial in~$M$.

\begin{definition}
Given a finite sequence $J = (x_1, \dots, x_d)$ of elements in $\pi_*(R)$,
we let
$$
R/J = R/x_1 \wedge_R \dots \wedge_R R/x_d \,,
$$
where each $R/x$ is defined by a homotopy cofibre sequence
$$
\Sigma^{|x|} R \overset{x}\longto R \longto R/x \longto \Sigma^{|x|+1} R \,.
$$
We shall also write $R/J$ for this $R$-module in contexts where $J$ is
interpreted as the ideal in~$\pi_*(R)$ generated by the given sequence
of elements.  This is, however, an abuse of notation, since $R/J$ depends
upon the chosen generators for the ideal, not just on the ideal~$J$ itself.
We may refer to $R/J$ and $\Gamma_J R$ as the \emph{Koszul complex}
and the \emph{stable Koszul complex}, respectively.
\end{definition}

\begin{lemma} \label{lem:CellGammaJ}
If $J = (x_1, \dots, x_d)$ is a finitely generated ideal in $\pi_*(R)$,
and $M$ is any $R$-module, then $\Cell_{R/J} M \simeq \Gamma_J M$
as $R$-modules.
\end{lemma}

\begin{proof}
We show that $\Gamma_J M$ is the $R/J$-cellularisation of~$M$.
An inductive argument, as in the proof of~\cite{DGI06}*{Prop.~9.3}, shows
that $R/J$ finitely builds $R/x_1^m \wedge_R \dots \wedge_R R/x_d^m$
for each $m\ge1$.  Passing to the colimit over~$m$, it follows that
$R/J$ builds $\Gamma_J R$.  Since $R$ builds~$M$, it also follows that
$R/J$ builds $\Gamma_J M$.  Finally, $F_R(R/J, \Gamma_J M) \to F_R(R/J,
M)$ is an equivalence, because $F_R(R/J, N[1/x_i]) \simeq *$ for each
$R$-module~$N$ and any $1 \le i \le d$.
\end{proof}

With notation as above, if $R \wedge A$ and $R/J$ mutually build one
another, then $C^f_n M \simeq \Cell_{R \wedge A} M \simeq \Cell_{R/J} M
\simeq \Gamma_J M$ by Lemmas~\ref{lem:CfnCell} and~\ref{lem:CellGammaJ}.
Under slightly different hypotheses we can close the cycle and obtain
this conclusion directly.

\begin{lemma} \label{lem:CfnGammaJ}
Let $I = (p^{a_0}, \dots, v_n^{a_n})$ and $J = (x_1, \dots, x_d)$.  If
\begin{enumerate}
\item
each $x_i$ acts nilpotently on each $\pi_* F(S/I, R)$, and
\item
$v_s^{a_s}$ acts nilpotently on $\pi_* F(S/(p^{a_0}, \dots,
v_{s-1}^{a_{s-1}}), R/J)$ for each $0 \le s \le n$,
\end{enumerate}
then $C^f_n M \simeq \Gamma_J M$ as $R$-modules.
\end{lemma}

\begin{proof}
Item~(1) ensures that $F(S/I, \Gamma_J R) \simeq \Gamma_J
F(S/I, R)$ is equivalent to $F(S/I, R)$ for each~$I$ in the cofinal
system, and passage to homotopy colimits implies that
$$
C^f_n \Gamma_J R \overset{\simeq}\longto C^f_n R
$$
is an equivalence.
Item~(2) ensures that $C^f_n R/J \simeq
\hocolim_I F(S/I, R/J)$ is equivalent to $R/J$, which implies that
$$
C^f_n \Gamma_J R \overset{\simeq}\longto \Gamma_J R
$$
is an equivalence, since $R/J$ builds~$\Gamma_J R$.  Hence $C^f_n R
\simeq \Gamma_J R$, and more generally $C^f_n M = C^f_n R \wedge_R M
\simeq \Gamma_J R \wedge_R M = \Gamma_J M$.
\end{proof}

\subsection{A composite functor spectral sequence}
\label{subsec:composite}

Let $I, J \subset R_*$ be finitely generated ideals in a graded
commutative ring, and let $M_*$ be an $R_*$-module.  If $I = (x)$ we write
$\Gamma_x M_* = H^0_I(M_*)$ and $M_*/x^\infty = H^1_I(M_*)$
for the kernel and the cokernel of the localisation homomorphism~$\gamma$ below.
$$
0 \to \Gamma_x M_* \longto M_* \overset{\gamma}\longto M_*[1/x]
	\longto M_*/x^\infty \to 0 \,.
$$
More generally, let $\Gamma_I M_* = H^0_I(M_*)$ denote the $I$-power
torsion submodule of~$M_*$.  The identity $\Gamma_I(\Gamma_J M_*) =
\Gamma_{I+J} M_*$ leads to a composite functor spectral sequence
$$
E_2^{i,j} = H^i_I(H^j_J(M_*))
	\Longrightarrow_i H^{i+j}_{I+J}(M_*) \,.
$$
This is a case of the double complex spectral sequence of~\cite{CE56}*{\S
XV.6}.  When $I = (x)$ and $J = (y)$, it arises by horizontally filtering
the condensation of the central commutative square below, leading to
an $E_1$-term given by the inner modules in the upper and lower rows,
and an $E_2$-term given by the modules at the four corners.
$$
\xymatrix{
\Gamma_x(M_*/y^\infty) \ar@{ >->}[r]
	& M_*/y^\infty \ar[r] & M_*[1/x]/y^\infty \ar@{->>}[r]
	& (M_*/y^\infty)/x^\infty \\
	& M_*[1/y] \ar[r]^-{\gamma} \ar@{->>}[u] & M_*[1/x, 1/y] \ar@{->>}[u] \\
	& M_* \ar[r]^-{\gamma} \ar[u]^-{\gamma} & M_*[1/x] \ar[u]^-{\gamma} \\
\Gamma_x(\Gamma_y M_*) \ar@{ >->}[r]
	& \Gamma_y M_* \ar[r] \ar@{ >->}[u] & \Gamma_y M_*[1/x] \ar@{ >->}[u]
		\ar@{->>}[r]
	& (\Gamma_y M_*)/x^\infty
}
$$
For bidegree reasons, the spectral sequence collapses at this stage,
so that
$$
E_2^{i,j} = E_\infty^{i,j} = \begin{cases}
\Gamma_x(\Gamma_y M_*) & \text{for $(i,j) = (0,0)$,} \\
(\Gamma_y M_*)/x^\infty & \text{for $(i,j) = (1,0)$,} \\
\Gamma_x (M_*/y^\infty) & \text{for $(i,j) = (0,1)$,} \\
(M_*/y^\infty)/x^\infty & \text{for $(i,j) = (1,1)$.}
\end{cases}
$$
It follows that we have identities
\begin{align*}
\Gamma_x(\Gamma_y M_*) &= \Gamma_{(x,y)} M_* = H^0_{(x,y)}(M_*) \\
(M_*/y^\infty)/x^\infty &= M_*/(x^\infty, y^\infty) = H^2_{(x,y)}(M_*) \,,
\end{align*}
and a natural short exact sequence
$$
0 \to (\Gamma_y M_*)/x^\infty \longto H^1_{(x,y)}(M_*)
	\longto \Gamma_x(M_*/y^\infty) \to 0 \,.
$$
If $\Gamma_y M_* \subset \Gamma_x M_*$, so that the $y$-power torsion
is entirely $x$-power torsion, then $(\Gamma_y M_*)/x^\infty = 0$
and $H^1_{(x,y)}(M_*) \cong \Gamma_x(M_*/y^\infty)$.

\section{Dualities}

\subsection{Artinian and Noetherian $S_p$-modules}

Let $S_p$ denote the $p$-completed sphere spectrum.  The category of
$S_p$-modules contains a subcategory of \emph{$p$-power torsion modules}
satisfying $\Gamma_p M \simeq M$, and a subcategory of \emph{$p$-complete
modules} for which $M \simeq M^\wedge_p$.  The (covariant) functors
$\Gamma_p$ and $(-)^\wedge_p$ give mutually inverse equivalences
between these full subcategories, cf.~\cite{HPS97}*{Thm.~3.3.5}.

We say that a $p$-power torsion module~$M$ is \emph{Artinian} if each
homotopy group $\pi_t(M)$ is an Artinian $\bZ_p$-module, i.e., a finite
direct sum of modules of the form $\bQ_p/\bZ_p$ or $\bZ/p^a$ for $a\ge1$.
Dually, we say that a $p$-complete module~$M$ is \emph{Noetherian} if
each homotopy group $\pi_t(M)$ is a Noetherian $\bZ_p$-module, i.e., a
finite direct sum of modules of the form $\bZ_p$ or $\bZ/p^a$ for $a\ge1$.
The latter are the same as the finitely generated $\bZ_p$-modules.
The simultaneously Artinian and Noetherian $S_p$-modules~$M$ are those
for which each $\pi_t(M)$ is finite.

\subsection{Brown--Comenetz duality}
We recall two related dualities.  First, working in $S_p$-modules, the
\emph{Brown--Comenetz duality} spectrum $I$ represents the cohomology
theory
$$
I^t(M) = \Hom_{\bZ_p}(\pi_t(M), \bQ_p/\bZ_p) \,,
$$
cf.~\cite{BC76}.  This makes sense because $\bQ_p/\bZ_p$ is an injective
$\bZ_p$-module.  Letting $IM = F_{S_p}(M, I)$, we obtain a contravariant
endofunctor~$I$ of $S_p$-modules, with $\pi_{-t}(IM) = I^t(M)$.  It maps
$p$-power torsion modules to $p$-complete modules, and restricts to
a functor from Artinian $S_p$-modules to Noetherian $S_p$-modules, since
$$
\Hom_{\bZ_p}(\bQ_p/\bZ_p, \bQ_p/\bZ_p) \cong \bZ_p
\quad\text{and}\quad
\Hom_{\bZ_p}(\bZ/p^a, \bQ_p/\bZ_p) \cong \bZ/p^a \,.
$$
In general, it does not map $p$-complete modules to $p$-power torsion
modules, but it does restrict to a functor from Noetherian $S_p$-modules
to Artinian $S_p$-modules.  Moreover, the natural map
$$
\rho \: M \longto I(IM)
$$
is an equivalence for $M$ that are Artinian or Noetherian.  Hence the
two restrictions of $I$ are mutually inverse contravariant equivalences
between Artinian $S_p$-modules and Noetherian $S_p$-modules.  If the
$S_p$-module action on~$M$ extends to a (left or right) $R$-module
structure, then $IM$ is naturally a (right or left) $R$-module.

\subsection{Anderson duality}
Second, the Eilenberg--MacLane spectrum $I_{\bQ_p} = H\bQ_p$
represents the ordinary rational cohomology theory
$$
I_{\bQ_p}^t(M) = \Hom_{\bZ_p}(\pi_t(M), \bQ_p)
$$
in $S_p$-modules.  The canonical surjection $\bQ_p \to \bQ_p/\bZ_p$
induces a map of cohomology theories, and a map $I_{\bQ_p} \to I$ of
representing spectra, whose homotopy fibre defines the \emph{Anderson
duality} spectrum~$I_{\bZ_p}$, cf.~\cite{And69} and~\cite{Kai71}.
Letting $I_{\bZ_p}M = F_{S_p}(M, I_{\bZ_p})$ and $I_{\bQ_p}M = F_{S_p}(M,
I_{\bQ_p})$, we obtain a natural homotopy fibre sequence
$$
\Sigma^{-1} IM \longto I_{\bZ_p}M \longto I_{\bQ_p}M \longto IM \,,
$$
which lifts the injective resolution $0 \to \bZ_p \to \bQ_p \to
\bQ_p/\bZ_p \to 0$.  The associated long exact sequence in homotopy
splits into short exact sequences
\begin{equation} \label{eq:ExtIZpHom}
0 \to \Ext_{\bZ_p}(\pi_{t-1}(M), \bZ_p)
	\longto \pi_{-t}(I_{\bZ_p}M)
	\longto \Hom_{\bZ_p}(\pi_t(M), \bZ_p) \to 0 \,.
\end{equation}
If the $S_p$-module action on~$M$ extends to a (left or right) $R$-module
structure, then $I_{\bZ_p} M$ is naturally a (right or left) $R$-module,
and the short exact sequence above is one of $\pi_*(R)$-modules.

The contravariant endofunctor $I_{\bZ_p}$ on $S_p$-modules is equivalent
to $\Sigma^{-1} I$ on the subcategory of $p$-power torsion modules, since
$I_{\bQ_p}$ is trivial on these objects.  More relevant to us is the fact
that it maps Noetherian $S_p$-modules to Noetherian $S_p$-modules, since
\begin{alignat*}{2}
&\Ext_{\bZ_p}(\bZ_p, \bZ_p) = 0
&
&\Hom_{\bZ_p}(\bZ_p, \bZ_p) \cong \bZ_p \\
&\Ext_{\bZ_p}(\bZ/p^a, \bZ_p) \cong \bZ/p^a
&\qquad
&\Hom_{\bZ_p}(\bZ/p^a, \bZ_p) = 0 \,.
\end{alignat*}
Moreover, the natural map
\begin{equation} \label{rho:Anderson}
\rho \: M \longto I_{\bZ_p}(I_{\bZ_p} M)
\end{equation}
is an equivalence for Noetherian~$M$, cf.~\cite{Yos75}*{Thm.~2}
and~\cite{Kna99}*{Cor.~2.8}.  Hence $I_{\bZ_p}$ restricts to a
contravariant self-equivalence of Noetherian $S_p$-modules, being its
own inverse equivalence.

We emphasise that the Brown--Comenetz and Anderson dual spectra, $IM$
and $I_{\bZ_p}M$, and the algebraic expressions for their homotopy groups,
are contravariantly functorial in the $S_p$- or $R$-module $M$.

\begin{lemma} \label{lem:BCdualofGammap}
There are natural equivalences
$$
I(\Gamma_p M) \simeq (IM)^\wedge_p
	\simeq \Sigma (I_{\bZ_p}M)^\wedge_p
$$
\end{lemma}

\begin{proof}
For $S_p$-modules $M$ and~$N$ we have $\Gamma_p S \wedge_{S_p} M
= \Gamma_p M$ and $F_{S_p}(\Gamma_p S, N) = N^\wedge_p$.
The first equivalence then follows from the adjunction $F_{S_p}(\Gamma_p S
\wedge_{S_p} M, I) \simeq F_{S_p}(\Gamma_p S, F_{S_p}(M, I))$.  The second
equivalence follows from the homotopy fibre sequence defining the Anderson
dual, since $(I_{\bQ_p} M)^\wedge_p$ is trivial.
If the $S_p$-module structure on~$M$ extends to an $R$-module structure,
then this is respected by all of these equivalences.
\end{proof}

\section{Mahowald--Rezk duality}

\subsection{Spectra with finitely presented cohomology}
Let $A$ denote the mod~$p$ Steenrod algebra, where $p$ is a prime.
We write $H^*(X)$ for the mod~$p$ cohomology of a spectrum~$X$, with
its natural left $A$-module structure.  For $n\ge0$ let $A(n)$ be the
finite sub (Hopf) algebra of~$A$ that is generated by $Sq^1, Sq^2, \dots,
Sq^{2^n}$ for $p=2$, and by $\beta, P^1, \dots, P^{p^{n-1}}$ for $p$ odd.
Also let $E(n)$ be the exterior sub (Hopf) algebra of $A(n)$ generated
by $Q_0, Q_1, \dots, Q_n$, where $Q_0 = Sq^1$ and $Q_i = [Sq^{2^i},
Q_{i-1}]$ for $i\ge1$ and $p=2$, and $Q_0 = \beta$ and $Q_{i+1} =
[P^{p^i}, Q_i]$ for $i\ge0$ and $p$ odd.

Let $X$ be a spectrum that is $p$-complete and bounded below.  Following
Mahowald and~Rezk~\cite{MR99}*{\S3} we say that $X$ is an \emph{fp-spectrum}
if $H^*(X)$ is finitely presented as an $A$-module.  This is equivalent
to asking that $H^*(X) \cong A \otimes_{A(n)} M$ is induced up from a
finite $A(n)$-module~$M$, for some~$n$.  We say that a graded abelian
group $\pi_*$ is \emph{finite} if the direct sum $\bigoplus_t
\pi_t$ is finite.  The class of $p$-local finite spectra~$V$ such that
$\pi_*(V \wedge X)$ is finite generates a thick subcategory of the
stable homotopy category, and is therefore equal to the class of $p$-local
finite spectra of type~$\ge m+1$ for some well-defined integer $m\ge0$.
We then say that $X$ has \emph{fp-type} equal to~$m$.  In each case $n
\ge m$, sometimes with strict inequality, cf.~\cite{BR:imj}*{Prop.~3.9}.

\begin{theorem}[\cite{MR99}*{Prop.~4.10, Thm.~8.2}]
Let $X$ be $p$-complete and bounded below, with $H^*(X) \cong A
\otimes_{A(n)} M$ for some finite $A(n)$-module~$M$.  Then $I C^f_n
X$ is $p$-complete and bounded below, with $H^*(I C^f_n X) \cong A
\otimes_{A(n)} \Hom_{\bF_p}(M, \Sigma^{a(n)} \bF_p)$, where $a(n)$
is the top degree of a nonzero class in~$A(n)$.
\end{theorem}

Recall that we write $ko = ko^\wedge_p$ and $\tmf = \tmf^\wedge_p$ for
the $p$-completed connective real $K$-theory and topological modular
forms spectra, respectively, and that these are commutative $S$-algebras.

\begin{proposition}[\cite{MR99}*{Cor.~9.3}]
\label{prop:MRCor93}
There is an equivalence of $ko$-modules
$$
\Sigma^6 ko \simeq I C^f_1 ko
$$
(at all primes~$p$), and an equivalence of $\tmf$-modules
$$
\Sigma^{23} \tmf \simeq I C^f_2 \tmf
$$
(at $p=2$ and at $p=3$).  The underlying $S_p$-modules are Noetherian
and bounded below.
\end{proposition}

\begin{proof}
For $X = ko$ completed at $p=2$, we have $H^*(ko) \cong A/\!/A(1) = A
\otimes_{A(1)} \bF_2$ by Stong~\cite{Sto63}, so
$$
H^*(IC^f_1 ko) \cong A \otimes_{A(1)} \Sigma^6 \bF_2
	= \Sigma^6 A/\!/A(1) \,,
$$
since $a(1) = 6$.  Choosing a map $S^6 \to IC^f_1 ko$ generating the
lowest homotopy (and homology) group, and using the natural $ko$-module
structure on the target, we obtain a $ko$-module map $\phi \: \Sigma^6
ko \to IC^f_1 ko$.  The induced $A$-module homomorphism $\phi^* \:
H^*(IC^f_1 ko) \to H^*(\Sigma^6 ko)$ has the form $\Sigma^6 A/\!/A(1)
\to \Sigma^6 A/\!/A(1)$, and is an isomorphism in degree~$6$, hence is
an isomorphism in all degrees.  It follows that $\phi$ is an equivalence
of $2$-complete $ko$-modules.

For $X = \ell = BP\<1\>$ completed at any prime~$p$ we have $H^*(\ell)
\cong A/\!/E(1)$, essentially by~\cite{Sin68}, so $H^*(IC^f_1 \ell) \cong
\Sigma^{2p} A/\!/E(1)$ and $IC^f_1 \ell \simeq \Sigma^{2p} \ell$.  For $p$
odd this implies the claim for $ko \simeq \bigvee_{i=0}^{(p-3)/2} \ell$.


For $X = \tmf$ completed at $p=2$, we have $H^*(\tmf) \cong
A/\!/A(2) = A \otimes_{A(2)} \bF_2$ by Hopkins--Mahowald and
Mathew~\cite{Mat16}*{Thm.~1.1}.  Hence
$$
H^*(IC^f_2 \tmf) \cong A \otimes_{A(2)} \Sigma^{23} \bF_2
	= \Sigma^{23} A/\!/A(2) \,,
$$
since $a(2) = 23$.  Choosing a map $S^{23} \to IC^f_2 \tmf$ generating the
lowest homotopy group, and using the natural $\tmf$-module structure on
the target, we obtain a $\tmf$-module map $\phi \: \Sigma^{23} \tmf \to
IC^f_2 \tmf$.  The induced $A$-module homomorphism $\phi^* \: H^*(IC^f_2
\tmf) \to H^*(\Sigma^{23} \tmf)$ has the form $\Sigma^{23} A/\!/A(2)
\to \Sigma^{23} A/\!/A(2)$, and is an isomorphism in degree~$23$,
hence is an isomorphism in all degrees.  It follows that $\phi$ is an
equivalence of $2$-complete $\tmf$-modules.


For $X = \tmf$ completed at $p=3$, we have $H^*(\tmf) \cong A
\otimes_{A(2)} M$ for a finite $A(2)$-module $M$ with $\Hom_{\bF_3}(M,
\bF_3) \cong \Sigma^{-64} M$, by Proposition~\ref{prop:Htmfp3} below.
It follows that
$$
H^*(I C^f_2 \tmf) \cong A \otimes_{A(2)} \Sigma^{23} M
	\cong \Sigma^{23} H^*(\tmf)
$$
as $A$-modules, since $a(2) = 87$ and $87-64=23$.  Choosing a map
$S^{23} \to IC^f_2 \tmf$ generating the lowest homotopy group, and
using the natural $\tmf$-module structure on the target, we obtain
a $\tmf$-module map $\phi \: \Sigma^{23} \tmf \to IC^f_2 \tmf$.
The induced $A$-module homomorphism $\phi^* \: H^*(IC^f_2 \tmf) \to
H^*(\Sigma^{23} \tmf)$ has the form $\Sigma^{23} A \otimes_{A(2)} M \to
\Sigma^{23} A \otimes_{A(2)} M$, and is an isomorphism in degree~$23$.
It follows from the relation $P^3 g_0 = P^1 g_8$ that $\phi^*$ is also
an isomorphism in degree~$23+8=31$, hence in all degrees, and that $\phi$
is an equivalence of $3$-complete $\tmf$-modules.
\end{proof}

\begin{proposition} \label{prop:Htmfp3}
At $p=3$ there is an isomorphism $H^*(\tmf) \cong A \otimes_{A(2)} M$
of $A$-modules, where
$$
M = \frac{A(2)/\!/E(2)\{g_0, g_8\}}{(P^1 g_0, P^3 g_0 = P^1 g_8)}
$$
is a finite $A(2)$-module of dimension~$18$ satisfying
$\Hom_{\bF_3}(M, \bF_3) \cong \Sigma^{-64} M$.
\end{proposition}

\begin{proof}
Let all spectra be implicitly completed at~$p=3$.
Let $\Psi = S \cup_\nu e^4 \cup_\nu e^8$, where $\nu = \alpha_1$
is detected by~$P^1$.  According to~\cite{Mat16}*{Thm.~4.16},
there is an equivalence $\tmf \wedge \Psi \simeq \tmf_0(2)$ of
$\tmf$-module spectra, where $\pi_*(\tmf_0(2)) = \bZ_p[a_2, a_4]$
with $|a_2| = 4$ and $|a_4| = 8$.  We take as known that
$H^*(\tmf_0(2)) \cong A/\!/E(2)\{g_0, g_8\}$, where $|g_0| = 0$
and $|g_8| = 8$.  The homotopy cofibre sequences
\begin{align*}
S &\longto \Psi \longto \Sigma^4 C\nu \\
\Sigma^4 C\nu &\longto \Sigma^4 \Psi \longto S^{12}
\end{align*}
induce homotopy cofibre sequences
\begin{align*}
\tmf &\longto \tmf_0(2) \longto \Sigma^4 \tmf \wedge C\nu \\
\Sigma^4 \tmf \wedge C\nu &\longto \Sigma^4 \tmf_0(2) \longto \Sigma^{12} \tmf
\end{align*}
of $\tmf$-modules.  Passing to cohomology, we get two short exact
sequences, which we splice together to an exact complex
$$
0 \to \Sigma^{12} H^*(\tmf)
	\to \Sigma^4 A/\!/E(2)\{g_0, g_8\}
	\overset{\partial}\longto A/\!/E(2)\{g_0, g_8\}
	\to H^*(\tmf) \to 0
$$
of $A$-modules.  Here $\partial(\Sigma^4 g_0) \in \bF_3\{P^1 g_0\}$
and $\partial(\Sigma^4 g_8) \in \bF_3\{P^3 g_0, P^1 g_8\}$.  Hence this
complex is induced up from an exact complex
$$
0 \to \Sigma^{12} M
	\to \Sigma^4 A(2)/\!/E(2)\{g_0, g_8\}
	\overset{\partial}\longto A(2)/\!/E(2)\{g_0, g_8\}
	\to M \to 0
$$
of $A(2)$-modules, where the rank of~$\partial$ is twice the dimension
of~$M$.  The dimension of $A(2)/\!/E(2)$ is~$27$, so the dimension of~$M$
is~$18$.  By exactness, $\partial(\Sigma^4 g_0)$ and $\partial(\Sigma^4
g_8)$ are nonzero.  From~\cite{Cul21}*{Cor.~6.7} it follows that we can
choose the signs of the generators so that $\partial(\Sigma^4 g_0) =
P^1 g_0$ and $\partial(\Sigma^4 g_8) = P^3 g_0 - P^1 g_8$.  This gives
the stated presentation of~$M$.

Applying $D = \Hom_{\bF_3}(-, \bF_3)$ we find that $D(A(2)/\!/E(2))
\cong \Sigma^{-64} A(2)/\!/E(2)$, and can calculate that $D\partial$
has the same form as~$\partial$, so that the dual of the exact
$A(2)$-module complex above presents $D(\Sigma^{12} M) = \Sigma^{-12}
DM$ as $\Sigma^{-76} M$, which implies that $M$ is concentrated in
degrees $0 \le * \le 64$, and is self-dual.
\end{proof}

\subsection{Local cohomology theorems by Mahowald--Rezk duality}

\begin{notation}
The graded ring structure of $\pi_*(ko)$ is well known~\cite{Bot59}.
We use the notation
$$
\pi_*(ko) = \bZ_p[\eta, A, B]/(2 \eta, \eta^3, \eta A, A^2 = 4 B)
$$
where $|\eta| = 1$, $|A| = 4$ and $|B| = 8$, cf.~\cite{BR21}*{Ex.~2.30}.
If $p$ is odd this simplifies to $\pi_*(ko) = \bZ_p[A]$.  We call~$B$
the \emph{Bott element}.
\end{notation}

\begin{table}
\caption{Algebra generators~$x_k$ for $\pi_*(\tmf)$ at $p=2$
	\label{tab:alggenpitmfp2}}
\begin{tabular}{>{$}r<{$}|>{$}l<{$}|>{$}c<{$}}
x & k & |x| \\
\hline
\eta & 0, 1, 4 & 1 \\
\nu & 0, 1, 2, 4, 5, 6 & 3 \\
\epsilon & 0, 1, 4, 5 & 8 \\
\kappa & 0, 4 & 14 \\
\bar\kappa & 0 & 20 \\
B & 0, 1, 2, 3, 4, 5, 6, 7 & 8 \\
C & 0, 1, 2, 3, 4, 5, 6, 7 & 12 \\
D & 1, 2, 3, 4, 5, 6, 7 & 0 \\
M & 0 & 192 \\
\end{tabular}
\end{table}

\begin{table}
\caption{Algebra generators~$x_k$ for $\pi_*(\tmf)$ at $p=3$
	\label{tab:alggenpitmfp3}}
\begin{tabular}{>{$}r<{$}|>{$}l<{$}|>{$}c<{$}}
x & k & |x| \\
\hline
\nu & 0, 1 & 3 \\
\beta & 0 & 10 \\
B & 0, 1, 2 & 8 \\
C & 0, 1, 2 & 12 \\
D & 1, 2 & 0 \\
H & 0 & 72 \\
\end{tabular}
\end{table}

\begin{notation}
The graded ring structure of $\pi_*(\tmf)$ is also
known~\cite{DFHH14}*{Ch.~13}, \cite{BR21}*{Ch.~9, Ch.~13}, up to a couple
of finer points.

For $p=2$, the graded commutative $\bZ_2$-algebra $\pi_*(\tmf)$ is
generated by forty classes $x_k$, where $x \in \{\eta, \nu, \epsilon,
\kappa, \bar\kappa, B, C, D, M\}$ and $0 \le k \le 7$.  The indices~$k$
that occur are shown in Table~\ref{tab:alggenpitmfp2}.  We abbreviate
$x_0$ to $x$, and note that $|x_k| = |x| + 24 k$ is positive in each case.
We call $B \in \pi_8(\tmf)$ and $M \in \pi_{192}(\tmf)$ the \emph{Bott
element} and the \emph{Mahowald element}, respectively.
See Figure~\ref{fig:pitmf} for the mod~$2$ Adams $E_\infty$-term for
$\tmf$ in the range $0 \le t-s \le 192$.

For $p=3$, the graded commutative $\bZ_3$-algebra $\pi_*(\tmf)$ is
generated by twelve classes $x_k$, where $x \in \{\nu, \beta, B, C, D,
H\}$ and $0 \le k \le 2$.  The values of~$k$ that occur are shown in
Table~\ref{tab:alggenpitmfp3}.  We abbreviate $x_0$ to $x$, and again
note that $|x_k| = |x| + 24 k$ is positive in each case.  We call $B
\in \pi_8(\tmf)$ and $H \in \pi_{72}(\tmf)$ the \emph{Bott element}
and the \emph{Hopkins--Miller element}, respectively.
See Figure~\ref{fig:pitmfp3} for the mod~$3$ ($\tmf$-module) Adams
$E_\infty$-term for $\tmf$ in the range $0 \le t-s \le 72$.

To avoid repetitive case distinctions we will sometimes write
$\bZ_p[B, M]$, $(p, B, M)$ or~$(B, M)$, both for $p=2$ and for $p=3$,
in spite of the fact that the correct notations for $p=3$ would be
$\bZ_3[B, H]$, $(3, B, H)$ or~$(B, H)$.  In effect, the element `$M$'
should be read as `$H$' for $p=3$.
\end{notation}

\begin{definition}
For any $p$-complete connective $S$-algebra~$R$ with $\pi_0(R) =
\bZ_p$ let
$$
\fn_0 = \ker(\pi_*(R) \to \bZ_p)
$$
denote the ideal in~$\pi_*(R)$ given by the classes in positive
degrees, and let
$$
\fn_p = \ker(\pi_*(R) \to \bF_p)
$$
denote the maximal ideal generated by $\fn_0$ and~$p$.
\end{definition}

We shall review the precise structure of $\pi_*(\tmf)$ as a $\bZ_p[B,
M]$-module in Section~\ref{sec:tmf}, but for now we will only need
the following, more qualitative, properties.  Their analogues for $ko$
are straightforward.

\begin{lemma} \label{lem:tmffinite}
The following hold for $p=2$ and for $p=3$.
\begin{enumerate}
\item \label{item1}
The graded group $\pi_*(\tmf)$ is finitely generated as a $\bZ_p[B,
M]$-module.
\item \label{item2}
The radical of $(B, M) \subset \pi_*(\tmf)$ is $\fn_0$, and
the radical of $(p, B, M)$ is $\fn_p$.
\item \label{item3}
The graded group $\pi_*(\tmf/(p, B, M))$ is finite.
\item \label{item4}
Each $\pi_t(\Gamma_{(B, M)} \tmf)$ is a finitely generated $\bZ_p$-module.
\end{enumerate}
\end{lemma}

\begin{proof}
We use the notation and results of Subsections~\ref{subsec:BMloccohtmf}
and~\ref{subsec:BHloccohtmf}.  In particular, $N = \tmf/M$, with $N_*
\cong \pi_*(N)$.

(1) Since $\pi_*(\tmf) \cong N_* \otimes \bZ[M]$, it suffices to check
that $N_*$ is finitely generated as a $\bZ_p[B]$-module, which is clear
from the explicit expressions given in Theorems~\ref{thm:ZBmoduleN}
and~\ref{thm:ZBmoduleNp3}.

(2) Because $\pi_*(\tmf)/(B, M)$ is finitely generated over~$\bZ_p$,
each positive degree class is nilpotent, and therefore each class in
$\fn_0$ lies in the radical of~$(B, M)$.  The other conclusions follow.

(3) Since $\Gamma_B N_*$ is finite, each $ko[k]$ is $B$-torsion free,
and each $ko[k]/B$ is a finitely generated $\bZ_p$-module, it follows that
$\pi_*(N/B)$ is a finitely generated $\bZ_p$-module.  Hence $\pi_*(N/(p,
B)) \cong \pi_*(\tmf/(p, B, M))$ is finite.

(4) Likewise, since $\Gamma_B N_*$ is finite and each $ko[k]/B^\infty$
is bounded above and finitely generated over~$\bZ_p$ in each degree,
it follows that $\pi_*(\Gamma_B N)$ is bounded above and finitely
generated over~$\bZ_p$ in each degree.  Hence $\pi_*(\Gamma_{(B, M)}
\tmf) \cong \pi_*(\Gamma_B N) \otimes \bZ[M]/M^\infty$ is also bounded
above and finitely generated over~$\bZ_p$ in each degree.
\end{proof}

\begin{theorem} \label{thm:MRduality}
There are equivalences of $ko$-modules
$$
\Gamma_{(p,B)} ko = \Gamma_{\fn_p} ko
	\simeq C^f_1 ko \simeq \Sigma^{-6} I(ko)
$$
(at all primes~$p$), and equivalences of $\tmf$-modules
$$
\Gamma_{(p,B,M)} \tmf = \Gamma_{\fn_p} \tmf
	\simeq C^f_2 \tmf \simeq \Sigma^{-23} I(\tmf)
$$
(at $p=2$ and at $p=3$).
The underlying $S_p$-modules are Artinian and bounded above.
\end{theorem}

\begin{proof}
For any $S_p$-module $M$ the natural homomorphism
$$
\rho \: \pi_t(M) \longto
	\Hom_{\bZ_p}(\Hom_{\bZ_p}(\pi_t(M), \bQ_p/\bZ_p), \bQ_p/\bZ_p)
$$
is injective.  Hence, if $\pi_{-t}(IM) = \Hom_{\bZ_p}(\pi_t(M),
\bQ_p/\bZ_p)$ is a Noetherian (= finitely generated) $\bZ_p$-module, then
$\rho$ exhibits $\pi_t(M)$ as a submodule of an Artinian $\bZ_p$-module,
which must itself be Artinian.  This applies with $M = C^f_1 ko$, since
we know from Proposition~\ref{prop:MRCor93} that $IM \simeq \Sigma^6 ko$,
and $ko$ is a Noetherian $S_p$-module.  It also applies with $M = C^f_2
\tmf$, since $IM \simeq \Sigma^{23} \tmf$ and $\tmf$ is a Noetherian
$S_p$-module by Lemma~\ref{lem:tmffinite}\eqref{item1}.  Hence $\rho \:
M \to I(IM)$ is in fact an equivalence in these cases, so that
\begin{align*}
C^f_1 ko &\simeq I(I C^f_1 ko) \simeq I(\Sigma^6 ko)
	\simeq \Sigma^{-6} I(ko) \\
C^f_2 \tmf &\simeq I(I C^f_2 \tmf) \simeq I(\Sigma^{23} \tmf)
	\simeq \Sigma^{-23} I(\tmf) \,.
\end{align*}

We use Lemma~\ref{lem:CfnGammaJ} with $n=1$, $R = M = ko$ and $J =
(p, B)$ to see that
\begin{equation} \label{eq:Cf1koisGammapBko}
C^f_1 ko \simeq \Gamma_{(p, B)} ko \,.
\end{equation}
The finite spectra $S/I$ with $I = (p^{a_0}, v_1^{a_1})$, and their
Spanier--Whitehead duals, have type~$2$, so $\pi_* F(S/I, ko)$ is finite,
since $ko$ has fp-type~$1$.  Hence both $p$ and~$B$ act nilpotently on
this graded~$\bZ_p$-module.  This confirms the first condition.  For the
second condition, note that $R/J = ko/(p, B)$ and $F(S/p^{a_0}, R/J)$
have finite graded homotopy groups, hence $p^{a_0}$ acts nilpotently on
the former and $v_1^{a_1}$ acts nilpotently on the latter.  The radical
in~$\pi_*(ko)$ of $J = (p, B)$ equals~$\sqrt{J} = \fn_p$, and as reviewed
in Subsection~\ref{subsec:localcohom} this implies the equivalence
$$
\Gamma_{(p, B)} ko \simeq \Gamma_{\fn_p} ko \,.
$$

Similarly, we use Lemma~\ref{lem:CfnGammaJ} with $n=2$, $R = M = \tmf$
and $J = (p, B, M)$ to see that
\begin{equation} \label{eq:Cf2tmfisGammapBMtmf}
C^f_2 \tmf \simeq \Gamma_{(p, B, M)} \tmf \,.
\end{equation}
The finite spectra $S/I$ with $I = (p^{a_0}, v_1^{a_1}, v_2^{a_2})$, and
their Spanier--Whitehead duals, have type~$3$, so $\pi_* F(S/I, \tmf)$
is finite, since $\tmf$ has fp-type~$2$.  Hence $p$, $B$ and~$M$
act nilpotently on this graded~$\bZ_p$-module.  Next, note that $R/J =
\tmf/(p, B, M)$, $F(S/p^{a_0}, R/J)$ and $F(S/(p^{a_0}, v_1^{a_1}), R/J)$
have finite homotopy, by Lemma~\ref{lem:tmffinite}\eqref{item3}, so that
$p^{a_0}$ acts nilpotently on the first, $v_1^{a_1}$ acts nilpotently
on the second, and $v_2^{a_2}$ acts nilpotently on the third.
The radical in~$\pi_*(\tmf)$ of $J = (p, B, M)$ equals~$\sqrt{J} =
\fn_p$, by Lemma~\ref{lem:tmffinite}\eqref{item2}, which implies the
equivalence
\begin{equation*}
\Gamma_{(p, B, M)} \tmf \simeq \Gamma_{\fn_p} \tmf \,.  \qedhere
\end{equation*}
\end{proof}

\begin{remark}
As an alternative to Lemma~\ref{lem:CfnGammaJ}, we could
use Lemmas~\ref{lem:CfnCell} and~\ref{lem:CellGammaJ} to
establish~\eqref{eq:Cf1koisGammapBko}.  For $p=2$ the Adams complex $A =
S/(2, v_1^4)$ from~\cite{Ada66}*{\S12} has type~$2$ and satisfies $ko
\wedge S/(2, v_1^4) \simeq ko/(2, B)$.  Hence $C^f_1 ko \simeq \Cell_{ko
\wedge A} ko \simeq \Cell_{ko/(2, B)} ko \simeq \Gamma_{(2, B)} ko$.
For $p$ odd the complex $S/(p, v_1)$ has type~$2$ and $ko \wedge S/(p,
v_1) \simeq ko/(p, A^m)$ with $m = (p-1)/2$.  Hence $C^f_1 ko \simeq
\Gamma_{(p, A^m)} ko \simeq \Gamma_{(p, B)} ko$, since $(p, A^m)$ and $(p,
B)$ have the same radical.

Likewise, we could use the existence~\cite{BHHM08} of a $2$-local
finite complex $A = S/(2, v_1^4, v_2^{32})$ of type~$3$ satisfying
$\tmf \wedge A \simeq \tmf/(2, B, M)$ to deduce that $C^f_2 \tmf
\simeq \Cell_{\tmf \wedge A} \simeq \Cell_{\tmf/(2, B, M)} \tmf \simeq
\Gamma_{(2, B, M)} \tmf$.  Similar remarks apply at~$p=3$, using the
type~$3$ complex $A = S/(3, v_1, v_2^9)$ constructed in~\cite{BP04}.
In these cases the argument for~\eqref{eq:Cf2tmfisGammapBMtmf} using
Lemma~\ref{lem:CfnGammaJ} is dramatically simpler, as is to be expected,
since asking for $R \wedge A$ and $R/J$ to be equivalent is a much more
restrictive condition than asking that they build one another.
\end{remark}

\begin{theorem}
There are equivalences of $ko$-modules
$$
\Gamma_B ko = \Gamma_{\fn_0} ko
	\simeq \Sigma^{-5} I_{\bZ_p}(ko)
$$
(at all primes~$p$), and equivalences of $\tmf$-modules
$$
\Gamma_{(B,M)} \tmf = \Gamma_{\fn_0} \tmf
	\simeq \Sigma^{-22} I_{\bZ_p}(\tmf)
$$
(at $p=2$ and at $p=3$).
The underlying $S_p$-modules are Noetherian and bounded above.
\end{theorem}

\begin{proof}
Proposition~\ref{prop:MRCor93}, equivalence~\eqref{eq:Cf1koisGammapBko},
and Lemma~\ref{lem:BCdualofGammap} applied to the $ko$-module $M =
\Gamma_B ko$, give equivalences
$$
\Sigma^6 ko \simeq IC^f_1 ko \simeq I(\Gamma_{(p,B)} ko)
	\simeq \Sigma (I_{\bZ_p} \Gamma_B ko)^\wedge_p \,.
$$
Here $\Gamma_B ko$ is a Noetherian $S_p$-module, which implies
that $I_{\bZ_p} \Gamma_B ko$ is also Noetherian and (in particular)
$p$-complete.  Hence $\Sigma^5 ko \simeq I_{\bZ_p} \Gamma_B ko$, which
implies that
$$
\Gamma_B ko \simeq I_{\bZ_p}(I_{\bZ_p} \Gamma_B ko)
	\simeq I_{\bZ_p}(\Sigma^5 ko)
	\simeq \Sigma^{-5} I_{\bZ_p}(ko) \,.
$$
Furthermore, $\Gamma_B ko = \Gamma_{\fn_0} ko$, since the radical of~$(B)$
in $\pi_*(ko)$ equals $\fn_0$.

Similarly, Proposition~\ref{prop:MRCor93},
equivalence~\eqref{eq:Cf2tmfisGammapBMtmf}, and
Lemma~\ref{lem:BCdualofGammap} applied to the $\tmf$-module $M =
\Gamma_{(B, M)} \tmf$, give equivalences
$$
\Sigma^{23} \tmf \simeq IC^f_2 \tmf \simeq I(\Gamma_{(p,B,M)} \tmf)
        \simeq \Sigma (I_{\bZ_p} \Gamma_{(B, M)} \tmf)^\wedge_p \,.
$$
Part~\eqref{item4} of Lemma~\ref{lem:tmffinite} asserts that $\Gamma_{(B,
M)} \tmf$ is a Noetherian $S_p$-module, which implies that $I_{\bZ_p}
\Gamma_{(B, M)} \tmf$ is also Noetherian and $p$-complete.  Hence
$\Sigma^{22} \tmf \simeq I_{\bZ_p} \Gamma_{(B, M)} \tmf$, which implies
that
$$
\Gamma_{(B, M)} \tmf \simeq I_{\bZ_p}(I_{\bZ_p} \Gamma_{(B, M)} \tmf)
        \simeq I_{\bZ_p}(\Sigma^{22} \tmf)
        \simeq \Sigma^{-22} I_{\bZ_p}(\tmf) \,.
$$
Finally, $\Gamma_{(B, M)} \tmf = \Gamma_{\fn_0} \tmf$, since
the radical of~$(B, M)$ in $\pi_*(\tmf)$ equals~$\fn_0$ by
Lemma~\ref{lem:tmffinite}\eqref{item2}.
\end{proof}

\section{Gorenstein duality}

\subsection{Gorenstein maps of $S$-algebras}

The original version~\cite{DGI06}*{Def.~8.1} of the following definition
was slightly more restrictive, but by~\cite{DGI06}*{Prop.~8.4} there is
no difference when $k$ is proxy-small as an $R$-module.

\begin{definition}
Let $R \to k$ be a map of $S$-algebras.  We say that $R \to k$ is
\emph{Gorenstein of shift~$a$} if there is an equivalence of left
$k$-modules
$$
F_R(k, R) \simeq \Sigma^a k \,.
$$
\end{definition}

Our next aim is to prove Proposition~\ref{prop:GorshiftGordual}.
We write $\bF_p$ and $\bZ_p$ for the mod~$p$ and $p$-adic integral
Eilenberg--MacLane spectra, respectively, with their unique (commutative)
$S_p$-algebra structures.

Suppose that $R$ is an $S_p$-algebra and that $k = \bF_p$.  There is
then an equivalence
$$
k \simeq Ik = F_{S_p}(k, I) \cong F_R(k, IR)
$$
of left $k$-modules, where $Ik$ and $IR = F_{S_p}(R, I)$
are the Brown--Comenetz duals of~$k$ and~$R$, and where $k$ acts from the
right on the domains of the two mapping spectra.  Hence, if $R \to k$
is Gorenstein of shift~$k$, then there is a $k$-module equivalence
$$
F_R(k, R) \simeq \Sigma^a k \simeq F_R(k, \Sigma^a IR) \,.
$$
Recall the notation $\cE = F_R(k, k)$ from
Subsection~\ref{subsec:cellularisation}.  Restriction along $R
\to k$ defines an $S$-algebra map $k^{op} \to \cE$, and the left
$\cE$-action on~$k$ induces right $\cE$-actions on $F_R(k, R)$ and
$F_R(k, \Sigma^a IR)$.  If $R$ is connective with $\pi_0(R) = \bZ_p$,
then $\cE$ is coconnective with $\pi_0(\cE) \cong k^{op}$ a field.
According to~\cite{DGI06}*{Prop.~3.9} the $k$-module equivalence above
then extends to an $\cE$-module equivalence, so that
$$
F_R(k, R) \wedge_{\cE} k \simeq F_R(k, \Sigma^a IR) \wedge_{\cE} k \,.
$$
Moreover, if $k$ is proxy-small, so that $k$-cellularisation is
effectively constructible by~\cite{DGI06}*{Thm.~4.10}, then we can
rewrite this as an equivalence
$$
\Cell_k R \simeq \Cell_k(\Sigma^a IR) \,.
$$
Finally, if $\pi_*(IR)$ is $p$-power torsion, then $k$ builds~$IR$ as
the homotopy colimit of a refined Whitehead tower in $R$-modules, so
that $\Cell_k(\Sigma^a IR) \simeq \Sigma^a IR$.  Hence these hypotheses
ensure that $R \to k$ has Gorenstein duality of shift~$a$, in the
following sense.

\begin{definition}
A map $R \to k = \bF_p$ of $S_p$-algebras has \emph{Gorenstein duality
of shift~$a$} if there is an equivalence
$$
\Cell_k R \simeq \Sigma^a IR \,.
$$
Similarly, a map $R \to k = \bZ_p$ of $S_p$-algebras has \emph{Gorenstein
duality of shift~$a$} if there is an equivalence
$$
\Cell_k R \simeq \Sigma^a I_{\bZ_p} R \,.
$$
\end{definition}

\begin{proposition} \label{prop:GorshiftGordual}
Let $k = \bF_p$ or $\bZ_p$, let $R \to k$ be a map of connective
$S_p$-algebras with $\pi_0(R) = \bZ_p$, and let $\cE = F_R(k, k)$.
Suppose that
\begin{enumerate}
\item
$R \to k$ is Gorenstein of shift~$a$, and
\item
$k$ is proxy-small as an $R$-module.
\end{enumerate}
For $k = \bF_p$ also assume that
\begin{enumerate}
\setcounter{enumi}{2}
\item
$\Hom_{\bZ_p}(\pi_*(R), \bQ_p/\bZ_p)$ is $p$-power torsion.
\end{enumerate}
Then $R \to k$ has Gorenstein duality of shift~$a$.
\end{proposition}

\begin{proof}
The case $k = \bF_p$ was discussed above.
When $k = \bZ_p$, we have an equivalence
$$
k \simeq F_{S_p}(k, I_{\bZ_p}) \cong F_R(k, I_{\bZ_p} R)
$$
of $k$-modules, where $I_{\bZ_p} R = F_{S_p}(R, I_{\bZ_p})$ is the
Anderson dual of~$R$.  Hence, if $R \to k$ is Gorenstein of shift~$a$
then there is a $k$-module equivalence
$$
F_R(k, R) \simeq \Sigma^a k \simeq F_R(k, \Sigma^a I_{\bZ_p} R) \,.
$$
Since $R$ is connective with $\pi_0(R) = \bZ_p$, it follows that $\cE$
is coconnective with $\pi_0(\cE) \cong k^{op}$ and $\pi_{-1}(\cE)
= 0$.  We prove in Proposition~\ref{prop:uniqueEaction} below that this
implies that the $k$-module equivalence above extends to an $\cE$-module
equivalence, so that
$$
F_R(k, R) \wedge_{\cE} k \simeq F_R(k, \Sigma^a I_{\bZ_p} R)
	\wedge_{\cE} k \,.
$$
If $k$ is proxy-small, then we can rewrite this as
$$
\Cell_k R \simeq \Cell_k(\Sigma^a I_{\bZ_p} R) \,.
$$
Since $\pi_*(I_{\bZ_p} R)$ is a bounded above graded $\bZ_p$-module,
it follows that $I_{\bZ_p} R$ is built from~$k$, so that $\Cell_k R
\simeq \Sigma^a I_{\bZ_p} R$.
\end{proof}

\subsection{Uniqueness of $\cE$-module structures}

\begin{proposition} \label{prop:uniqueEaction}
Let $k = \bF_p$ or~$\bZ_p$, and let $k^{op} \to \cE$ be a map of
coconnective $S$-algebras inducing an isomorphism on~$\pi_0$.
For $k = \bZ_p$ also assume that $\pi_{-1}(\cE) = 0$.
Then any two right $\cE$-module structures on~$k$ are equivalent.
\end{proposition}

\begin{proof}
When $k = \bF_p$, this is a special case of~\cite{DGI06}*{Prop.~3.9}.
When $k = \bZ_p$, we refine the proof of that proposition.  Let $k_1$ and
$k_2$ be right $\cE$-modules whose restricted $k$-module structures are
given by the usual left action on~$k$.  Let $M = \cE$ as an $\cE$-module,
with the usual right action, and choose $\cE$-module maps $f_1 \:
M \to k_1$ and $f_2 \: M \to k_2$ inducing isomorphisms on~$\pi_0$.
We shall extend~$M$ to a cellular $\cE$-module~$N$ such that $f_1$
extends to an $\cE$-module equivalence $g_1 \: N \to k_1$, and such
that $f_2$ extends to an $\cE$-module map $g_2 \: N \to k_2$.  It will
then follow that $g_2$ is also an equivalence, and $k_1$ and $k_2$
are equivalent as $\cE$-modules.
$$
\xymatrix{
M \ar[r] \ar@(u,u)[rrr]^-{f_1} \ar@(d,l)[drr]_-{f_2}
	& M' \ar[r] & N \ar[r]^-{g_1} \ar[d]_-{g_2} & k_1 \\
& & k_2
}
$$
Note that $\pi_{-1}(M) = 0$.  As a first approximation to~$N$ we build
a cellular $\cE$-module~$M'$ by attaching $\cE$-cells of dimension $\le
0$ to~$M$, so that $\pi_{-1}(M') = 0$ and $\pi_t(M) \to \pi_t(M')$ is
trivial for each $t\le-2$.  More precisely, for each $t\le-2$ choose an
$\cE$-module map
$$
\bigvee_\alpha \Sigma^t \cE
	\overset{\phi_t}\longto \bigvee_\beta \Sigma^t \cE
$$
such that
$$
0 \to \bigoplus_\alpha \bZ_p
	\overset{\pi_t(\phi_t)}\longto \bigoplus_\beta \bZ_p
	\longto \pi_t(M) \to 0
$$
is a free $\bZ_p$-resolution of $\pi_t(M)$.  There is then a map $C\phi_t
\to M$ from the homotopy cofibre of $\phi_t$, inducing an isomorphism
on $\pi_t$.  The composite $C\phi_t \to M \to k_1$ is null-homotopic, since
$C\phi_t$ has cells in dimensions~$t$ and~$t+1 \le -1$ only, while $k_1$ is
connective.  Let $M'$ be the mapping cone of the sum over~$t$ of the maps
$C\phi_t \to M$, and let $f'_1 \: M' \to k_1$ extend~$f_1$.  Then $M'$
has the stated properties.

Iterating the process infinitely often, and letting $N$ be the
(homotopy) colimit of the sequence $M \subset M' \subset \dots$, we
calculate that $\pi_t(N) = 0$ for $t \ne 0$, while $g_1 \: N \to k_1$
is a $\pi_0$-isomorphism, and therefore an equivalence.

We obtained~$N$ from~$M$ by attaching cells of dimensions $\le 0$, so
the obstructions to extending $f_2 \: M \to k_2$ lie in the negative
homotopy groups of~$k_2$, which are trivial.  Hence an extension $g_2 \:
N \to k_2$ exists.  It must be a $\pi_0$-isomorphism, since $f_1$, $f_2$
and~$g_1$ have this property, and is therefore an equivalence, as claimed.
\end{proof}

\begin{remark}
The hypothesis on $\pi_{-1}(\cE)$ can in general not be omitted;
see~\cite{DGI06}*{Rem.~3.11}.
\end{remark}

\subsection{Gorenstein descent}

Suppose that we are given maps $R \to T \to k$ of $S$-algebras, and
that~$T$ is somehow easier to work with than~$R$.  A \emph{descent
theorem} for a property~$P$ gives hypotheses under which $P$ for~$T \to
k$ implies $P$ for~$R \to k$.  We first apply this idea in the case of
the Gorenstein property.

\begin{lemma}
Let $T \to k$ be a map of $S$-algebras, and suppose that the homomorphism
of coefficient rings $\pi_*(T) \to \pi_*(k)$ is (algebraically)
Gorenstein in the sense that $\Ext_{\pi_*(T)}^{*,*}(\pi_*(k),
\pi_*(T))$ is a free $\pi_*(k)$-module of rank~$1$, on a generator in
bidegree~$(s,t)$.  Then $T \to k$ is Gorenstein of shift $a = t-s$.
\end{lemma}

Grothendieck's definition given in~\cite{Har67}*{p.~63} is more
restrictive, but this condition suffices for our purposes.

\begin{proof}
The conditionally convergent $\Ext$ spectral sequence
$$
E_2^{s,t} = \Ext_{\pi_*(T)}^{s,t}(\pi_*(k), \pi_*(T))
	\Longrightarrow_s \pi_{t-s} F_T(k, T)
$$
of~\cite{EKMM97}*{Thm.~IV.4.1} is a $\pi_*(k)$-module spectral sequence
that collapses at the $E_2$-term, hence is strongly convergent.
It follows that $\pi_* F_T(k, T) \cong \Sigma^a \pi_*(k)$ as
$\pi_*(k)$-modules, so that $F_T(k, T) \simeq \Sigma^a k$ as $k$-modules.
\end{proof}

\begin{example} \label{ex:ZpFpGor}
$\bZ_p \to \bF_p$ is Gorenstein of algebraic shift $(s,t) = (1,0)$
and of topological shift~$a = -1$.
\end{example}

\begin{lemma}
Let $R \to T \to k$ be maps of $S$-algebras, and suppose that $R \to T$
is Gorenstein of shift~$b$.  Then $T \to k$ is Gorenstein of shift~$a$
if and only if $R \to k$ is Gorenstein of shift~$a+b$.
\end{lemma}

\begin{proof}
By hypothesis, $F_R(T, R) \simeq \Sigma^b T$ as left $T$-modules.
It follows that
$$
F_R(k, R) \cong F_T(k, F_R(T, R)) \simeq F_T(k, \Sigma^b T)
	\cong \Sigma^b F_T(k, T)
$$
as left~$k$-modules.  Hence $F_T(k, T) \simeq \Sigma^a k$ if
and only if $F_R(k, R) \simeq \Sigma^{a+b} k$.
\end{proof}

\begin{example} \label{ex:polyGor}
If $\pi_*(T) \cong \bZ_p[y_1, \dots, y_d]$ is polynomial on finitely
many generators, and $k = \bZ_p$, then the ring homomorphism $\pi_*(T)
\to \bZ_p$ is Gorenstein of shift $(s, t) = (d, - \sum_{i=1}^d |y_i|)$.
Hence the $S$-algebra map $T \to \bZ_p$ is Gorenstein of shift
$$
a = - d - \sum_{i=1}^d |y_i| = - \sum_{i=1}^d (|y_i|+1) \,.
$$
Moreover, $\pi_*(T) \to \bF_p$ is Gorenstein of shift $(d + 1,  -
\sum_{i=1}^d |y_i|)$ and $T \to \bF_p$ is Gorenstein of shift $- d -
1 - \sum_{i=1}^d |y_i|$.
\end{example}

\begin{proposition} \label{prop:kotmfGorShift}
The $S$-algebra maps
$$
ko \longto ku \longto \bZ_p \longto \bF_p
$$
are Gorenstein of shift $-2$, $-3$ and~$-1$, respectively.  Hence $ko
\to \bZ_p$ is Gorenstein of shift~$-5$ and $ko \to \bF_p$ is Gorenstein
of shift~$-6$.

At $p=2$ the $S$-algebra maps
$$
\tmf \longto \tmf_1(3) \longto \bZ_2 \longto \bF_2
$$
are Gorenstein of shift $-12$, $-10$ and~$-1$, respectively.
At $p=3$ the $S$-algebra maps
$$
\tmf \longto \tmf_0(2) \longto \bZ_3 \longto \bF_3
$$
are Gorenstein of shift $-8$, $-14$ and~$-1$, respectively.
At $p\ge5$ the $S$-algebra maps
$$
\tmf \longto \bZ_p \longto \bF_p
$$
are Gorenstein of shift $-22$ and~$-1$, respectively.  Hence $\tmf \to
\bZ_p$ is Gorenstein of shift~$-22$, and $\tmf \to \bF_p$ is Gorenstein
of shift~$-23$, uniformly at all primes.
\end{proposition}

\begin{proof}
The homotopy rings
\begin{align*}
\pi_*(ku) &\cong \bZ_p[u] \\
\pi_*(\tmf_1(3)) &\cong \bZ_2[a_1, a_3] \\
\pi_*(\tmf_0(2)) &\cong \bZ_3[a_2, a_4] \\
\pi_*(\tmf) &\cong \bZ_p[c_4, c_6] \,,
\end{align*}
where $p\ge5$ in the last case, are all polynomial, with $|u| = 2$, $|a_i|
= 2i$ and $|c_i| = 2i$.  See \cite{Bot59}, \cite{MR09}*{Prop.~3.2},
\cite{Beh06}*{\S1.3}, \cite{Del75}*{Prop.~6.1} or~\cite{BR21}*{\S9.3,
Thm.~13.4}.  This accounts for the Gorenstein shifts by~$-3$, $-10$,
$-14$ and~$-22$, as in Example~\ref{ex:polyGor}.

The shifts by~$-1$ are covered by Example~\ref{ex:ZpFpGor}.

By Wood's theorem~\cite{BG10}*{Lem.~4.1.2}, and its
parallels~\cite{Mat16}*{Thm.~4.12, Thm.~4.15} for topological modular
forms, there are equivalences
\begin{equation} \label{eq:wood}
\begin{aligned}
ko \wedge C\eta &\simeq ku \\
\tmf \wedge \Phi &\simeq \tmf_1(3) \\
\tmf \wedge \Psi &\simeq \tmf_0(2)
\end{aligned}
\end{equation}
of $ko$- or $\tmf$-modules, according to the case.  Here $C\eta = S
\cup_\eta e^2$ is a $2$-cell, $2$-dimensional Spanier--Whitehead self-dual
spectrum, $\Phi$ is an $8$-cell, $12$-dimensional Spanier--Whitehead
self-dual $2$-local spectrum~\cite{BR21}*{Lem.~1.42} with mod~$2$
cohomology $H^*(\Phi) \cong A(2)/\!/E(2) \cong \Phi A(1)$ realising
the double of~$A(1)$, and $\Psi = S \cup_\nu e^4 \cup_\nu e^8$ is
a $3$-cell $8$-dimensional Spanier--Whitehead self-dual $3$-local
spectrum~\cite{BR21}*{Def.~13.3} with mod~$3$ cohomology $H^*(\Psi)
\cong P(0) = \<P^1\>$.  The duality equivalences $D(C\eta) \simeq
\Sigma^{-2} C\eta$, $D\Phi \simeq \Sigma^{-12} \Phi$ and $D\Psi \simeq
\Sigma^{-8} \Psi$ account for the Gorenstein shifts by~$-2$, $-12$
and~$-8$, respectively.  For example, in the case of $\tmf$ at~$p=2$
we have equivalences
\begin{align*}
F_{\tmf}(\tmf_1(3), \tmf)
	&\simeq F_{\tmf}(\tmf \wedge \Phi, \tmf)
	\simeq F(\Phi, \tmf) \\
	&\simeq \tmf \wedge D\Phi
	\simeq \tmf \wedge \Sigma^{-12} \Phi
	\simeq \Sigma^{-12} \tmf_1(3)
\end{align*}
of $\tmf$-modules.
\end{proof}

\subsection{Small descent}

We can use descent to verify that $k$ is a (proxy-)small $R$-module
in the cases relevant for Sections~\ref{sec:ko} and~\ref{sec:tmf}.

\begin{lemma} \label{lem:smalldescent}
Let $R \to T \to k$ be maps of $S$-algebras, such that $T$ is small as
an $R$-module and $k$ is small as a $T$-module.  Then $k$ is small as
an $R$-module.
\end{lemma}

\begin{proof}
Since $T$ finitely builds $k$ as a $T$-module, this remains true
as $R$-modules.  Hence $R$ finitely builds~$k$.
\end{proof}

\begin{lemma} \label{lem:polysmall}
Let $T$ be a commutative $S$-algebra with $\pi_*(T) \cong \bZ_p[y_1,
\dots, y_d]$.  Then $\bZ_p$ and $\bF_p$ are small as $T$-modules.
\end{lemma}

\begin{proof}
$T/y_1 \wedge_T \dots \wedge_T T/y_d \simeq \bZ_p$ and
$\bZ_p/p \simeq \bF_p$ are finitely built from~$T$.
\end{proof}

\begin{remark}
More generally, if $\pi_*(k)$ is a perfect $\pi_*(T)$-module, meaning
that it admits a finite length resolution by finitely generated projective
$\pi_*(T)$-modules, then $k$ is small as a $T$-module.
\end{remark}

\begin{corollary} \label{cor:ZpFpsmallkotmf}
$\bZ_p$ and $\bF_p$ are small as $ko$-modules and as $\tmf$-modules,
at all primes~$p$.
\end{corollary}

\begin{proof}
Lemma~\ref{lem:polysmall} applies to the commutative $S_p$-algebras
$ko$ for $p\ge3$, $ku$ for all~$p$,
$\tmf_1(3)$ for $p=2$, $\tmf_0(2)$ for $p=3$, and~$\tmf$ for $p\ge5$.
Lemma~\ref{lem:smalldescent} then covers the cases of~$ko$ at~$p=2$,
and $\tmf$ at $p \in \{2,3\}$, in view of~\eqref{eq:wood}.
\end{proof}

\subsection{Descent of algebraic cellularisation}

\begin{definition}
Let $R$ be a commutative $S$-algebra and $k$ an $R$-module.  We say that
$R$ has \emph{algebraic $k$-cellularisation by~$J$} if $J = (x_1, \dots,
x_d) \subset \pi_*(R)$ is a finitely generated ideal with
$$
\Cell_k M \simeq \Gamma_J M
$$
for all $R$-modules~$M$.
\end{definition}

This condition only depends on the radical $\sqrt{J}$ of~$J$, and by
Lemmas~\ref{lem:mutuallybuild} and~\ref{lem:CellGammaJ} it is equivalent
to asking that the $R$-modules $k$ and~$R/J$ mutually build one another.

\begin{lemma} \label{lem:algcellpoly}
Let $T$ be a commutative $S$-algebra with $\pi_*(T) \cong \bZ_p[y_1,
\dots, y_d]$.  Then $T$ has algebraic $\bZ_p$-cellularisation by $(y_1,
\dots, y_d)$, and algebraic $\bF_p$-cellularisation by $(p, y_1, \dots,
y_d)$.
\end{lemma}

\begin{proof}
Letting $J' = (y_1, \dots, y_d)$ or $J' = (p, y_1, \dots, y_d)$ we have
$T/J' \simeq k = \bZ_p$ or $T/J' \simeq k = \bF_p$, according to the case.
Hence $\Cell_k M \simeq \Cell_{T/J'} M \simeq \Gamma_{J'} M$.
\end{proof}

\begin{lemma} \label{lem:algcelldesc}
Let $\phi \: R \to T$ be a map of commutative $S$-algebras, where $R$ is
connective with $\pi_0(R) = \bZ_p$.  Let $J = (x_1, \dots, x_d) \subset
\pi_*(R)$ be such that $\pi_*(R/J)$ is bounded above, and suppose that $T$
has algebraic $\bZ_p$-cellularisation by
$$
J' = (\phi(x_1), \dots, \phi(x_d)) \subset \pi_*(T) \,.
$$
Then $R$ has algebraic $\bZ_p$-cellularisation by~$J$.

Similarly, if $\pi_*(R/J)$ is $p$-power torsion and bounded above,
and $T$ has algebraic $\bF_p$-cellularisation by $J'$, then $R$ has
algebraic $\bF_p$-cellularisation by~$J$.
\end{lemma}

\begin{proof}
In the case $k = \bZ_p$, the $R$-module $\bZ_p$ builds $R/J$ since
$\pi_*(R/J)$ is bounded above.  Conversely, $R$ builds~$T$ so $R/J$
builds $T \wedge_R R/J = T/J'$.  By hypothesis, $T/J'$ builds $\bZ_p$
in $T$-modules, hence also in $R$-modules.  Thus $R/J$ builds $\bZ_p$
in $R$-modules.

Similarly, for $k = \bF_p$ the $R$-module $\bF_p$ builds $R/J$ since
$\pi_*(R/J)$ is $p$-power torsion and bounded above.  Conversely,
$R/J$ builds $T/J'$ as before.  By hypothesis, $T/J'$ builds $\bF_p$
in $T$-modules, hence also in $R$-modules.  Thus $R/J$ builds $\bF_p$
in $R$-modules.
\end{proof}

Recall that $B \in \pi_8(\tmf)$ (together with $B + \epsilon$) is detected
by the modular form~$c_4$, while we write $M$ for $M \in \pi_{192}(\tmf)$
detected by $\Delta^8$ when $p=2$, and for $H \in \pi_{72}(\tmf)$ detected
by $\Delta^3$ when $p=3$.  For uniformity of notation, let us also write
$M$ for the class in $\pi_{24}(\tmf)$ detected by~$\Delta$ when $p\ge5$.

\begin{proposition} \label{prop:algcellkotmf}
The commutative $S_p$-algebra $ko$ has algebraic $\bZ_p$-cellularisation
by $(B)$, and algebraic $\bF_p$-cellularisation by $(p, B)$, for all
primes~$p$.
\begin{align*}
\Cell_{\bZ_p} ko &\simeq \Gamma_{(B)} ko \\
\Cell_{\bF_p} ko &\simeq \Gamma_{(p, B)} ko
\end{align*}
Likewise, $\tmf$ has algebraic $\bZ_p$-cellularisation by $(B, M)$,
and algebraic $\bF_p$-cellu\-larisation by $(p, B, M)$, for all primes~$p$.
\begin{align*}
\Cell_{\bZ_p} \tmf &\simeq \Gamma_{(B, M)} \tmf \\
\Cell_{\bF_p} \tmf &\simeq \Gamma_{(p, B, M)} \tmf
\end{align*}
\end{proposition}

\begin{proof}
For $ko$, we apply Lemma~\ref{lem:algcelldesc} to the complexification
map $\phi \: ko \to ku$ with $J = (B)$, where $\phi(B) = u^4$.
Then $\pi_*(ko/B) \cong \bZ_p\{1, \eta, \eta^2, A\}/(2 \eta, 2 \eta^2)$
is finitely generated over~$\bZ_p$.  Moreover, $J' = (u^4)$ has radical
$(u) \subset \pi_*(ku)$.  According to Lemma~\ref{lem:algcellpoly},
$ku$ has algebraic $\bZ_p$-cellularisation by~$(u)$, hence it also has
algebraic $\bZ_p$-cellularisation by~$J'$.

Similarly, with $J = (p, B)$ we see that $\pi_*(ko/(p, B))$ is finite
and $J' = (p, u^4)$ has radical $(p, u) \subset \pi_*(ku)$, so $ku$
has algebraic $\bF_p$-cellularisation by~$(p, u)$ and by~$J'$.

For $\tmf$ at $p=2$ we apply Lemma~\ref{lem:algcelldesc} to the map $\phi
\: \tmf \to \tmf_1(3)$ with $J = (B, M)$, where
\begin{alignat*}{2}
\phi(B) &= c_4
&\qquad
c_4 &= a_1 (a_1^3 - 24 a_3) \\
\phi(M) &= \Delta^8
&\qquad
\Delta &= a_3^3 (a_1^3 - 27 a_3) \,,
\end{alignat*}
according to the formulas for $\Gamma_1(3)$-modular forms.
See~\cite{BR21}*{\S9.3} and the more detailed references therein.
It is clear from Theorem~\ref{thm:ZBmoduleN} that $\pi_*(\tmf/(B, M))
\cong \pi_*(N/B)$ is finitely generated over~$\bZ_2$.  Moreover, $J' =
(c_4, \Delta^8)$ has radical $(a_1, a_3) \subset \pi_*(\tmf_1(3))$,
so $\tmf_1(3)$ has algebraic $\bZ_2$-cellularisation by $(a_1, a_3)$
according to Lemma~\ref{lem:algcellpoly}, hence also by~$J'$.

Similarly, with $J = (2, B, M)$ we see that $\pi_*(\tmf/(2, B, M)) \cong
\pi_*(N/(2, B))$ is finite and $J' = (2, c_4, \Delta^8)$ has radical~$(2,
a_1, a_3)$, so $\tmf$ has algebraic $\bF_2$-cellularisation by $(2, a_1,
a_3)$ and by~$J'$.

For $\tmf$ at $p=3$ we apply Lemma~\ref{lem:algcelldesc} to the map
$\phi \: \tmf \to \tmf_0(2)$ with $J = (B, H)$, where
\begin{alignat*}{2}
\phi(B) &= c_4
&\qquad
c_4 &= 16 (a_2^2 - 3 a_4) \\
\phi(H) &= \Delta^3
&\qquad
\Delta &= 16 a_4^2 (a_2^2 - 4 a_4)
\end{alignat*}
according to the formulas for $\Gamma_0(2)$-modular forms.
See~\cite{BR21}*{\S13.1} and the more detailed references therein.
It is clear from Theorem~\ref{thm:ZBmoduleNp3} that $\pi_*(\tmf/(B, H))
\cong \pi_*(N/B)$ is finitely generated over~$\bZ_3$.  Moreover, $J' =
(c_4, \Delta^3)$ has radical $(a_2, a_4) \subset \pi_*(\tmf_0(2))$,
so $\tmf_0(2)$ has algebraic $\bZ_3$-cellularisation by $(a_2, a_4)$
according to Lemma~\ref{lem:algcellpoly}, hence also by~$J'$.

Similarly, with $J = (3, B, H)$ we see that $\pi_*(\tmf/(3, B, H)) \cong
\pi_*(N/(3, B))$ is finite and $J' = (3, c_4, \Delta^3)$ has radical~$(3,
a_2, a_4)$, so $\tmf$ has algebraic $\bF_3$-cellularisation by $(3, a_2,
a_4)$ and by~$J'$.

For $\tmf$ at $p\ge5$, the ideal $J' = (B, M) = (c_4, \Delta)$, with
$\Delta = (c_4^3 - c_6^2)/1728$, has radical~$(c_4, c_6)$.  Hence $\tmf$
has algebraic $\bZ_p$-cellularisation by $(c_4, c_6)$ and by~$J'$.

Similarly, the ideal $J' = (p, B, M) = (p, c_4, \Delta)$ has radical~$(p,
c_4, c_6)$, so $\tmf$ has algebraic $\bF_p$-cellularisation by~$(p, c_4,
c_6)$ and by~$J'$.
\end{proof}

\subsection{Local cohomology theorems by Gorenstein duality}

\begin{theorem} \label{thm:Gorduality}
There are equivalences of $ko$-modules
$$
\Gamma_B ko = \Gamma_{\fn_0} ko
	\simeq \Sigma^{-5} I_{\bZ_p}(ko)
$$
and equivalences of $\tmf$-modules
$$
\Gamma_{(B,M)} \tmf = \Gamma_{\fn_0} \tmf
	\simeq \Sigma^{-22} I_{\bZ_p}(\tmf)
$$
at all primes~$p$.
\end{theorem}

\begin{proof}
We apply Proposition~\ref{prop:GorshiftGordual} to $R \to k$ with $R =
ko$ or $R = \tmf$ and $k = \bZ_p$.  Then $R \to k$ is Gorenstein of
shift $a = -5$ or $a = -22$ by Proposition~\ref{prop:kotmfGorShift},
and $k$ is small, hence proxy-small, as an $R$-module by
Corollary~\ref{cor:ZpFpsmallkotmf}.  Hence
$\Cell_k R \simeq \Sigma^a I_{\bZ_p} R$
in each case.  Moreover,
$\Cell_k R \simeq \Gamma_J R$
for $J = (B) \subset \pi_*(ko)$ or $J = (B, M) \subset \pi_*(\tmf)$,
by Proposition~\ref{prop:algcellkotmf}.  Finally, $\Gamma_J R
\simeq \Gamma_{\fn_0} R$ since $J$ has radical~$\fn_0$ in each case,
cf.~Lemma~\ref{lem:tmffinite}\eqref{item2}.
\end{proof}

\begin{theorem}
There are equivalences of $ko$-modules
$$
\Gamma_{(p,B)} ko = \Gamma_{\fn_p} ko
	\simeq \Sigma^{-6} I(ko)
$$
and equivalences of $\tmf$-modules
$$
\Gamma_{(p,B,M)} \tmf = \Gamma_{\fn_p} \tmf
	\simeq \Sigma^{-23} I(\tmf)
$$
at all primes~$p$.
\end{theorem}

\begin{proof}
We apply Proposition~\ref{prop:GorshiftGordual} to $R \to k$ with $R =
ko$ or $R = \tmf$ and $k = \bF_p$.  Then $R \to k$ is Gorenstein of
shift $a = -6$ or $a = -23$ by Proposition~\ref{prop:kotmfGorShift},
and $k$ is small, hence proxy-small, as an $R$-module by
Corollary~\ref{cor:ZpFpsmallkotmf}.  Furthermore, $\pi_t(R)$ is a
finitely generated $\bZ_p$-module for each~$t$, as is clear from
Theorems~\ref{thm:pitmfextension} and~\ref{thm:ZBmoduleN} below, so
$\Hom_{\bZ_p}(\pi_*(R), \bQ_p/\bZ_p)$ is $p$-power torsion in each degree.
Hence
$\Cell_k R \simeq \Sigma^a I R$
in each case.  Moreover,
$\Cell_k R \simeq \Gamma_J R$
for $J = (p, B) \subset \pi_*(ko)$ or $J = (p, B, M) \subset \pi_*(\tmf)$,
by Proposition~\ref{prop:algcellkotmf}.  Finally, $\Gamma_J R
\simeq \Gamma_{\fn_p} R$ since $J$ has radical~$\fn_p$ in each case,
cf.~Lemma~\ref{lem:tmffinite}\eqref{item2}.
\end{proof}

\subsection{$ko$- and $\tmf$-module Steenrod algebras}

For completeness, we record the structure of $\pi_*(\cE)$ in our main
cases of interest, where $\cE = F_R(k, k)$, $R = ko$ or $\tmf$,
and $k = \bF_p$.

\begin{proposition}[\cite{Hil07}, \cite{DFHH14}, \cite{BR21}]
For $p=2$ there are algebra isomorphisms
\begin{gather*}
\pi_* F_{ko}(\bF_2, \bF_2) \cong A(1) \\
\pi_* F_{\tmf}(\bF_2, \bF_2) \cong A(2) \,.
\end{gather*}
For $p=3$ there is a square-zero quadratic extension
$$
\pi_* F_{\tmf}(\bF_3, \bF_3) = A_{\tmf} \longto A(1) \,,
$$
where $A_{\tmf}$ is generated by classes $\beta$ and $P^1$ in
cohomological degrees~$1$ and~$4$, subject to $\beta^2 = 0$, $\beta
(P^1)^2 \beta = (\beta P^1)^2 + (P^1 \beta)^2$ and $(P^1)^3 = 0$.
In each case, classes in homotopical degree~$-m$ correspond to classes
in cohomological degree~$m$.
\end{proposition}

\begin{proof}
Restriction along $S \to \tmf$ induces an $S$-algebra map
$$
\cE = F_{\tmf}(\bF_2, \bF_2) \longto F_S(\bF_2, \bF_2)
$$
and an algebra homomorphism $\pi_*(\cE) \to A$ to the mod~$2$ Steenrod
algebra.  Base change along $S \to \bF_2$
lets us rewrite the $S$-algebra map as
$$
\cE \cong F_{\bF_2 \wedge \tmf}(\bF_2 \wedge \bF_2, \bF_2)
	\longto F_{\bF_2}(\bF_2 \wedge \bF_2, \bF_2) \,.
$$
Since the dual Steenrod algebra $A_* = \pi_*(\bF_2 \wedge \bF_2)$ is
free as an $H_*(\tmf) = \pi_*(\bF_2 \wedge \tmf)$-module, the $\Ext$
spectral sequences for these two function spectra collapse, and let us
rewrite the algebra homomorphism as the monomorphism
$$
\pi_*(\cE) \cong \Hom_{H_*(\tmf)}(A_*, \bF_2)
	\longto \Hom_{\bF_2}(A_*, \bF_2) \cong A \,.
$$
By duality, this identifies $\pi_*(\cE)$ with the $H^*(\tmf) =
A/\!/A(2)$-comodule primitives in~$A$, which is precisely the
subalgebra~$A(2)$.

The proof for~$ko$ is the same, replacing $A(2)$ with~$A(1)$.

The result for $p=3$ is due to Henriques and Hill \cite{Hil07}*{Thm.~2.2},
\cite{DFHH14}*{\S13.3}, except for the comment that the extension is
square-zero, which appears in~\cite{BR21}*{\S13.1}.
\end{proof}

\section{Thera duality}

A third line of proof is discussed in~\cite{BR21}*{\S10.3, \S10.4, \S13.5},
yielding the following theorems.

\begin{theorem}[\cite{BR21}*{Thm.~10.6, Prop.~10.12}]
	\label{thm:TheradualityA}
There are equivalences of $2$-complete $\tmf$-modules
\begin{align*}
\Sigma^{23} \tmf &\simeq I(\Gamma_{(2, B, M)} \tmf) \\
\Sigma^{22} \tmf &\simeq I_{\bZ_2}(\Gamma_{(B, M)} \tmf) \,.
\end{align*}
\end{theorem}

\begin{theorem}[\cite{BR21}*{Thm.~13.20, Prop.~13.21}]
	\label{thm:TheradualityBC}
There are equivalences of $3$-complete $\tmf$-modules
\begin{align*}
\Sigma^{23} \tmf &\simeq I(\Gamma_{(3, B, H)} \tmf) \\
\Sigma^{22} \tmf &\simeq I_{\bZ_3}(\Gamma_{(B, H)} \tmf) \,.
\end{align*}
\end{theorem}

This approach combines descent with a strengthening of the Cohen--Macaulay
property, equivalent to the Gorenstein property.  One first observes that
$$
\Sigma^{11} \tmf_1(3) \simeq I(\Gamma_{(2, a_1, a_3)} \tmf_1(3)) \,,
$$
because the local cohomology of $\pi_*(\tmf_1(3)) = \bZ_2[a_1, a_3]$
at the maximal ideal $\fn_2 = (2, a_1, a_3)$ is concentrated in a single
cohomological degree, and, moreover, its $\bZ_2$-module Pontryagin dual
is a free $\pi_*(\tmf_1(3))$-module on one generator.  The conclusion
for $\tmf$ follows by faithful descent along $\tmf \to \tmf_1(3) \simeq
\tmf \wedge \Phi$, since $\Phi$ is Spanier--Whitehead self-dual.

\section{Topological $K$-theory} \label{sec:ko}

As a warm-up to Section~\ref{sec:tmf}, we spell out the structure
of the local cohomology spectral sequences
$$
E_2^{s,t} = H^s_{(B)}(\pi_*(ko))_t
        \Longrightarrow_s \pi_{t-s}(\Gamma_{B} ko)
                \cong \pi_{t-s}(\Sigma^{-5} I_{\bZ_2}(ko))
$$
and
$$
E_2^{s,t} = H^s_{(2, B)}(\pi_*(ko))_t
        \Longrightarrow_s \pi_{t-s}(\Gamma_{(2, B)} ko)
                \cong \pi_{t-s}(\Sigma^{-6} I(ko)) \,.
$$
Multiplication by~$B$ acts injectively on the depth~$1$ graded
commutative ring
$$
\pi_*(ko) = \bZ_2[\eta, A, B]/(2 \eta, \eta^3, \eta A, A^2 - 4 B)
$$
and we let $N_*$ denote a basic block for this action.

\begin{definition}
In this section only, let $N_* \subset \pi_*(ko)$ be the $\bZ_2$-submodule
of classes in degrees $0 \le * < 8$, and let $N = ko/B$.
\end{definition}

\begin{lemma}
The composite
$N_* \otimes \bZ[B] \to \pi_*(ko) \otimes \bZ[B]
	\overset{\cdot}\to \pi_*(ko)$
is an isomorphism.  As a $\bZ_2$-module, $N_* = \bZ_2\{1, \eta, \eta^2,
A\}/(2 \eta, 2 \eta^2)$ is a split extension by the $2$-torsion submodule
$\Gamma_2 N_* = \bZ/2\{\eta, \eta^2\}$ of the $2$-torsion free quotient
$N_* / \Gamma_2 N_* = \bZ_2\{1, A\}$.
\end{lemma}

\begin{lemma}
$H^0_{(B)}(\pi_*(ko)) = 0$ and $H^1_{(B)}(\pi_*(ko)) \cong N_* \otimes
\bZ[B]/B^\infty$.
\end{lemma}

\begin{proof}
These are the cohomology groups of the complex
$$
0 \to \pi_*(ko) \overset{\gamma}\longto \pi_*(ko)[1/B] \to 0 \,,
$$
which we may rewrite as
$0 \to N_* \otimes \bZ[B]
	\overset{\gamma}\longto N_* \otimes \bZ[B^{\pm1}] \to 0$.
\end{proof}

\begin{figure}
\includegraphics[scale=0.52]{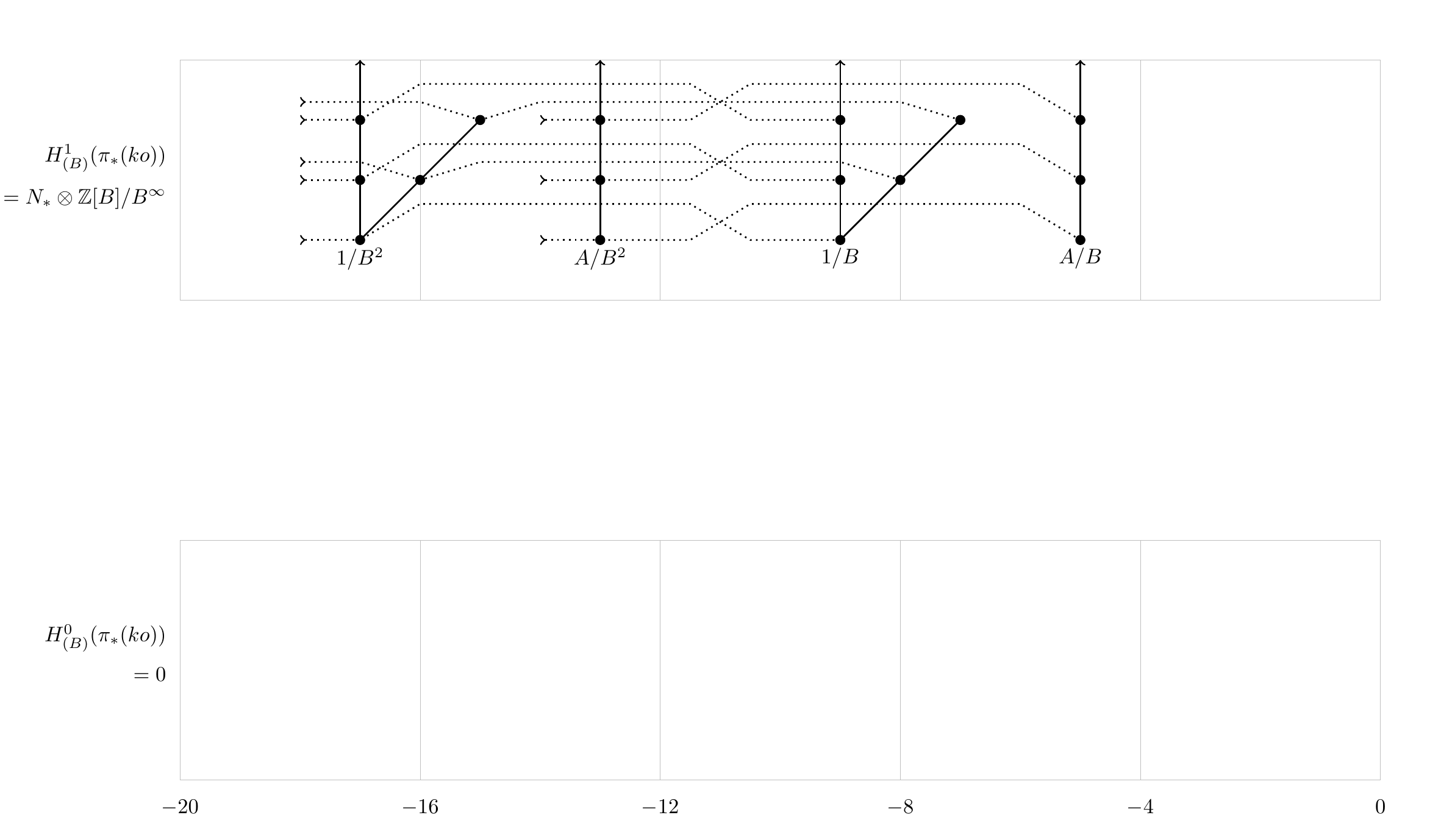}
\caption{$E_2^{s,t} = H^s_{(B)}(\pi_*(ko))_t
        \Longrightarrow_s \pi_{t-s}(\Gamma_B ko)$
	\label{fig:GammaBko}}
\end{figure}

\begin{proposition}
The local cohomology spectral sequence
$$
E_2^{s,t} = H^s_{(B)}(\pi_*(ko))_t
	\Longrightarrow_s \pi_{t-s}(\Gamma_B ko)
	\cong \pi_{t-s}(\Sigma^{-5} I_{\bZ_2} ko)
$$
has $E_2$-term concentrated on the $s=1$ line, with
$E_2^{1,*} = N_* \otimes \bZ[B]/B^\infty$.  There is
no room for differentials or hidden extensions, so
$E_2 = E_\infty$.  Hence there are isomorphisms
$$
\Sigma^{-1} \bZ_2\{1, A\} \otimes \bZ[B]/B^\infty \cong \Sigma^{-5}
	\Hom_{\bZ_2}(\bZ_2\{1, A\} \otimes \bZ[B], \bZ_2)
$$
and
$$
\Sigma^{-1} \bZ/2\{\eta, \eta^2\} \otimes \bZ[B]/B^\infty \cong \Sigma^{-6}
	\Ext_{\bZ_2}(\bZ/2\{\eta, \eta^2\} \otimes \bZ[B], \bZ_2) \,.
$$
\end{proposition}

\begin{proof}
See Figure~\ref{fig:GammaBko} and recall the short exact
sequence~\eqref{eq:ExtIZpHom}.
\end{proof}

\begin{lemma}
$H^s_{(2, B)}(\pi_*(ko)) \cong H^{s-1}_{(2)}(N_*) \otimes \bZ[B]/B^\infty$
where
\begin{align*}
H^0_{(2)}(N_*) &= \bZ/2\{\eta, \eta^2\} \\
H^1_{(2)}(N_*) &= \bQ_2/\bZ_2\{1, A\} \,.
\end{align*}
\end{lemma}

\begin{proof}
See Lemma~\ref{lem:HBMtmf} for the proof of the first isomorphism.
The $H^*_{(2)}(N_*)$ are the cohomology groups of the complex
\begin{equation*}
0 \to N_* \overset{\gamma}\longto N_*[1/2] \to 0 \,. \qedhere
\end{equation*}
\end{proof}

\begin{figure}
\includegraphics[scale=0.52]{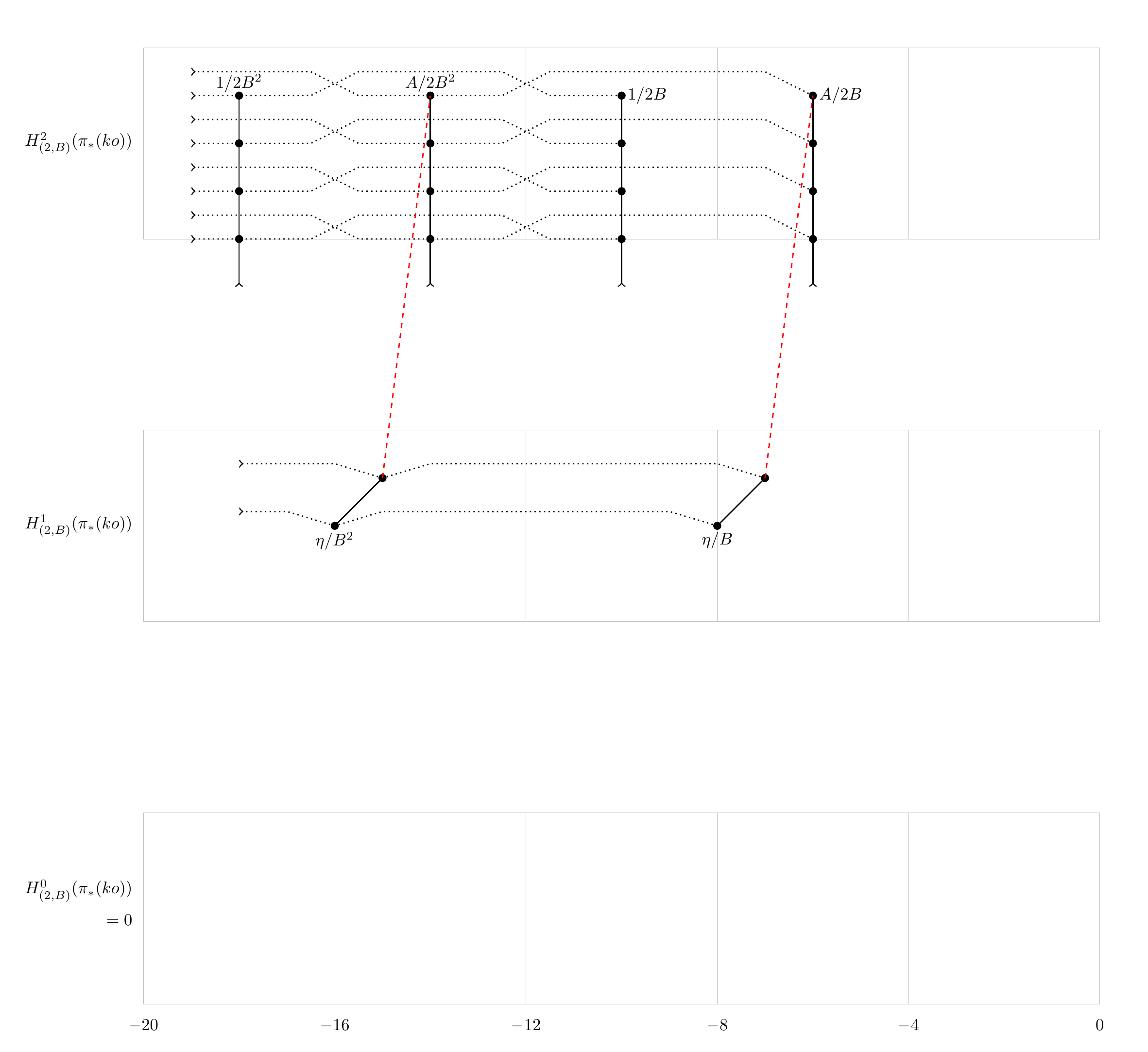}
\caption{$E_2^{s,t} = H^s_{(2, B)}(\pi_*(ko))_t
        \Longrightarrow_s \pi_{t-s}(\Gamma_{(2, B)} ko)$
	\label{fig:Gamma2Bko}}
\end{figure}

\begin{proposition}
The local cohomology spectral sequence
$$
E_2^{s,t} = H^s_{(2, B)}(\pi_*(ko))_t
	\Longrightarrow_s \pi_{t-s}(\Gamma_{(2, B)} ko)
	\cong \pi_{t-s}(\Sigma^{-6} I(ko))
$$
has $E_2$-term
$$
E_2^{s,t} \cong H^{s-1}_{(2)}(N_*) \otimes \bZ[B]/B^\infty \,.
$$
There is no room for differentials, so $E_2 = E_\infty$.
Hence there are isomorphisms
$$
\Sigma^{-2} \bQ_2/\bZ_2\{1, A\} \otimes \bZ[B]/B^\infty \cong \Sigma^{-6}
	\Hom_{\bZ_2}(\bZ_2\{1, A\} \otimes \bZ[B], \bQ_2/\bZ_2)
$$
and
$$
\Sigma^{-1} \bZ/2\{\eta, \eta^2\} \otimes \bZ[B]/B^\infty \cong \Sigma^{-6}
	\Hom_{\bZ_2}(\bZ/2\{\eta, \eta^2\} \otimes \bZ[B], \bQ_2/\bZ_2) \,.
$$
Moreover, there are hidden $\eta$-extensions as shown by sloping dashed
red lines.
\end{proposition}

\begin{proof}
See Figure~\ref{fig:Gamma2Bko}.
\end{proof}

\section{Topological modular forms} \label{sec:tmf}

We can now work out the structure of the local cohomology spectral sequences
\begin{align*}
E_2^{s,t} &= H^s_{(B, M)}(\pi_*(\tmf))_t \\
	&\Longrightarrow_s \pi_{t-s}(\Gamma_{(B, M)} \tmf)
		\cong \pi_{t-s}(\Sigma^{-22} I_{\bZ_p}(\tmf))
\end{align*}
and
\begin{align*}
E_2^{s,t} &= H^s_{(p, B, M)}(\pi_*(\tmf))_t \\
	&\Longrightarrow_s \pi_{t-s}(\Gamma_{(p, B, M)} \tmf)
		\cong \pi_{t-s}(\Sigma^{-23} I(\tmf))
\end{align*}
for $p=2$ and for $p=3$.  Recall the algebra generators for
$\pi_*(\tmf)$ listed in Table~\ref{tab:alggenpitmfp2} for $p=2$ and in
Table~\ref{tab:alggenpitmfp3} for $p=3$.  In each case multiplication
by~$M$ acts injectively on the depth~$1$ graded commutative ring
$\pi_*(\tmf)$, and we let $N_*$ denote a basic block for this action.
(The notation $BB$ is used for a similar object in~\cite{GM17}.)
To begin, we review the $\bZ_p[B, M]$-module structure
on $\pi_*(\tmf)$ and the $\bZ_p[B]$-module structure on~$N_*$, in the
notation from~\cite{BR21}*{Ch.~9}.

\subsection{$(B, M)$-local cohomology of $\tmf$}
\label{subsec:BMloccohtmf}

Let $p=2$ in this subsection and the next.
See Figure~\ref{fig:pitmf} for the mod~$2$ Adams $E_\infty$-term for
$\tmf$ in the range $0 \le t-s \le 192$, with all hidden $2$-, $\eta$-
and $\nu$-extensions shown.  There are no hidden $B$- or $M$-extensions
in this spectral sequence.

\begin{definition} \label{def:Nstar}
Let $N_* \subset \pi_*(\tmf)$ be the $\bZ_2[B]$-submodule generated by
all classes in degrees $0 \le * < 192$, and let $N = \tmf/M$.
\end{definition}

\begin{theorem}[\cite{BR21}*{Thm.~9.27}]
\label{thm:pitmfextension}
The composite homomorphisms
\begin{gather*}
N_* \otimes \bZ[M] \longto \pi_*(\tmf) \otimes \pi_*(\tmf)
        \overset{\cdot}\longto \pi_*(\tmf) \\
N_* \subset \pi_*(\tmf) \longto \pi_*(N)
\end{gather*}
are isomorphisms.  Hence $\pi_*(\tmf)$ is a (split) extension of $\bZ_2[B,
M]$-modules
$$
0 \to \Gamma_B N_* \otimes \bZ[M] \longto \pi_*(\tmf)
	\longto \frac{N_*}{\Gamma_B N_*} \otimes \bZ[M] \to 0 \,.
$$
\end{theorem}

\begin{definition} \label{def:nukdk}
Let $\nu_3 = \eta_1^3$ and $\nu_7 = 0$, and set $d_k = 8/\gcd(k,8)$,
so that $d_0 = 1$, $d_4 = 2$, $d_2 = d_6 = 4$, $d_1 = d_3 = d_5 = d_7 =
8$ and~$d_{7-k} \nu_k = 0$ for $0 \le k \le 7$.
\end{definition}

\begin{theorem}[\cite{BR21}*{Thm.~9.26}] \label{thm:ZBmoduleN}
As a $\bZ_2[B]$-module, $N_*$ is a split extension
$$
0 \to \Gamma_B N_* \longto N_* \longto \frac{N_*}{\Gamma_B N_*} \to 0 \,.
$$
The $B$-power torsion submodule $\Gamma_B N_*$ is given in
Table~\ref{BtorsionN}.  It is concentrated in degrees $3 \le * \le 164$,
and is finite in each degree.  The action of~$B$ is as indicated in the
table, together with $B \cdot \epsilon_1 = 2 \bar\kappa^2$, $B \cdot \eta
\nu_2 = 2 \bar\kappa^3$ and $B \cdot \epsilon_5 \kappa = 4 \nu \nu_6$.

The $B$-torsion free quotient of $N_*$ is the direct sum
$$
\frac{N_*}{\Gamma_B N_*} = \bigoplus_{k=0}^7 ko[k]
$$
of the following eight $\bZ_2[B]$-modules,
with $ko[k]$ concentrated in degrees~$* \ge 24k$:
\begin{align*}
ko[0] &= \bZ_2[B]\{1, C\}
        \oplus \bZ/2[B]\{\eta, \eta^2\} \\
ko[1] &= \bZ_2\{D_1\} \oplus \bZ_2[B]\{B_1, C_1\}
        \oplus \bZ/2[B]\{\eta_1, \eta \eta_1\} \\
ko[2] &= \bZ_2\{D_2\} \oplus \bZ_2[B]\{B_2, C_2\}
        \oplus \bZ/2[B]\{\eta B_2, \eta_1^2\} \\
ko[3] &= \bZ_2\{D_3\} \oplus \bZ_2[B]\{B_3, C_3\}
        \oplus \bZ/2[B]\{\eta B_3, \eta^2 B_3\} \\
ko[4] &= \bZ_2\{D_4\} \oplus \bZ_2[B]\{B_4, C_4\}
        \oplus \bZ/2[B]\{\eta_4, \eta \eta_4\} \\
ko[5] &= \bZ_2\{D_5\} \oplus \bZ_2[B]\{B_5, C_5\}
        \oplus \bZ/2[B]\{\eta B_5, \eta_1 \eta_4\} \\
ko[6] &= \bZ_2\{D_6\} \oplus \bZ_2[B]\{B_6, C_6\}
        \oplus \bZ/2[B]\{\eta B_6, \eta^2 B_6\} \\
ko[7] &= \bZ_2\{D_7\} \oplus \bZ_2[B]\{B_7, C_7\}
        \oplus \bZ/2[B]\{\eta B_7, \eta^2 B_7\} \,.
\end{align*}
The $\bZ_2[B]$-module structures are specified by $B \cdot D_k = d_k B_k$
for each $1 \le k \le 7$.
\end{theorem}

\begin{remark}
The submodule $N_* \subset \pi_*(\tmf)$ is preserved by the action of
$\eta$, $\nu$, $\epsilon$, $\kappa$ and~$\bar\kappa$.  To check this,
note that the $B^2$-torsion classes $\kappa C_7$, $\bar\kappa B_7$ and
$\bar\kappa C_7$ are zero.  It follows that the isomorphism $N_* \otimes
\bZ[M] \cong \pi_*(\tmf)$ also respects the action by these elements.
\end{remark}

\begin{table}
\caption{$B$-power torsion in $N_*$ at $p=2$
	\label{BtorsionN}}
\begin{tabular}{>{$}r<{$} | >{$}c<{$} | >{$}c<{$}}
n & \Gamma_B N_n & \text{generator} \\
\hline
3 & \bZ/8 & \nu \\
6 & \bZ/2 & \nu^2 \\
8 & \bZ/2 & \epsilon \\
9 & \bZ/2 & \eta \epsilon \\
14 & \bZ/2 & \kappa \\
15 & \bZ/2 & \eta \kappa \\
17 & \bZ/2 & \nu \kappa \\
20 & \bZ/8 & \bar\kappa \\
21 & \bZ/2 & \eta \bar\kappa \\
22 & \bZ/2 & \eta^2 \bar\kappa = B \cdot \kappa \\
\hline
27 & \bZ/4 & \nu_1 \\
28 & \bZ/2 & \eta \nu_1 = B \cdot \bar\kappa \\
32 & \bZ/2 & \epsilon_1 \\
33 & \bZ/2 & \eta \epsilon_1 \\
34 & \bZ/2 & \kappa \bar\kappa \\
35 & \bZ/2 & \eta \kappa \bar\kappa = B \cdot \nu_1 \\
39 & \bZ/2 & \eta_1 \kappa \\
40 & \bZ/4 & \bar\kappa^2 \\
41 & \bZ/2 & \eta \bar\kappa^2 \\
42 & \bZ/2 & \eta^2 \bar\kappa^2 = B \cdot \kappa \bar\kappa \\
45 & \bZ/2 & \eta_1 \bar\kappa \\
46 & \bZ/2 & \eta \eta_1 \bar\kappa \\
\hline
51 & \bZ/8 & \nu_2 \\
52 & \bZ/2 & \eta \nu_2 \\
53 & \bZ/2 & \eta^2 \nu_2 = B \cdot \eta_1 \bar\kappa \\
54 & \bZ/4 & \nu \nu_2 \\
57 & \bZ/2 & \nu^2 \nu_2 \\
59 & \bZ/2 & B \nu_2 \\
60 & \bZ/4 & \bar\kappa^3 \\
65 & (\bZ/2)^2 & \eta_1 \bar\kappa^2 \\
- & - & \nu_2 \kappa \\
66 & \bZ/2 & \eta \nu_2 \kappa \\
68 & \bZ/2 & \nu \nu_2 \kappa \\
70 & \bZ/2 & \eta_1^2 \bar\kappa \\
\hline
75 & \bZ/2 & \eta_1^3 \\
80 & \bZ/2 & \bar\kappa^4 \\
\end{tabular}
\qquad
\qquad
\qquad
\begin{tabular}{>{$}r<{$} | >{$}c<{$} | >{$}c<{$}}
n & \Gamma_B N_n & \text{generator} \\
\hline
85 & \bZ/2 & \eta_1 \bar\kappa^3 \\
90 & \bZ/2 & \eta_1^2 \bar\kappa^2 \\
\hline
99 & \bZ/8 & \nu_4 \\
100 & \bZ/2 & \eta \nu_4 \\
102 & \bZ/2 & \nu \nu_4 \\
104 & \bZ/2 & \epsilon_4 \\
105 & (\bZ/2)^2 & \eta \epsilon_4 \\
- & - & \eta_1 \bar\kappa^4 \\
110 & \bZ/4 & \kappa_4 \\
111 & \bZ/2 & \eta \kappa_4 \\
113 & \bZ/2 & \nu \kappa_4 \\
116 & \bZ/4 & \bar\kappa D_4 \\
117 & \bZ/2 & \eta_4 \bar\kappa \\
118 & \bZ/2 & \eta \eta_4 \bar\kappa = B \cdot \kappa_4 \\
\hline
123 & \bZ/4 & \nu_5 \\
124 & \bZ/2 & \eta \nu_5 \\
125 & \bZ/2 & \eta^2 \nu_5 = B \cdot \eta_4 \bar\kappa \\
128 & \bZ/2 & \epsilon_5 \\
129 & \bZ/2 & \eta \epsilon_5 \\
130 & \bZ/4 & \kappa_4 \bar\kappa \\
131 & \bZ/2 & \eta \kappa_4 \bar\kappa = B \cdot \nu_5 \\
135 & \bZ/2 & \eta_1 \kappa_4 \\
136 & \bZ/2 & \eta \eta_1 \kappa_4 = B \cdot \epsilon_5 \\
137 & \bZ/2 & \nu_5 \kappa \\
138 & \bZ/2 & \eta \nu_5 \kappa = B \cdot \kappa_4 \bar\kappa \\
142 & \bZ/2 & \epsilon_5 \kappa \\
\hline
147 & \bZ/8 & \nu_6 \\
148 & \bZ/2 & \eta \nu_6 \\
149 & \bZ/2 & \eta^2 \nu_6 \\
150 & \bZ/8 & \nu \nu_6 \\
153 & \bZ/2 & \nu^2 \nu_6 \\
155 & \bZ/2 & B \nu_6 \\
156 & \bZ/2 & B \eta \nu_6 \\
161 & \bZ/2 & \nu_6 \kappa \\
162 & \bZ/2 & \eta \nu_6 \kappa \\
164 & \bZ/2 & \nu \nu_6 \kappa \\
\end{tabular}
\end{table}

\begin{lemma} \label{lem:HBMtmf}
$$
H^s_{(B,M)}(\pi_*(\tmf))
	\cong H^{s-1}_{(B)}(N_*) \otimes \bZ[M]/M^\infty \,.
$$
\end{lemma}

\begin{proof}
The spectral sequence
$$
E_2^{i,j} = H^i_{(B)}(H^j_{(M)}(N_* \otimes \bZ[M]))
	\Longrightarrow_i H^{i+j}_{(B, M)}(N_* \otimes \bZ[M])
$$
collapses at the $j=1$ line, where $H^1_{(M)}(\bZ[M]) =
\bZ[M]/M^\infty = \bZ[M^{-1}]\{1/M\}$.
\end{proof}

\begin{lemma} \label{lem:HBN}
$$
H^0_{(B)}(N_*) = \Gamma_B N_*
$$
and
$$
H^1_{(B)}(N_*) = N_*/B^\infty = \bigoplus_{k=0}^7 ko[k]/B^\infty
$$
is the direct sum of the following eight $\bZ_2[B]$-modules,
with $ko[k]/B^\infty$ concentrated in degrees $* \le 24k+4$:
\begin{align*}
ko[0]/B^\infty &= \bZ_2[B]/B^\infty\{1, C\}
	\oplus \bZ/2[B]/B^\infty\{\eta, \eta^2\} \\
ko[1]/B^\infty &= \bZ_2[B]/B^\infty\{B_1, C_1\}/(8 B_1/B)
	\oplus \bZ/2[B]/B^\infty\{\eta_1, \eta \eta_1\} \\
ko[2]/B^\infty &= \bZ_2[B]/B^\infty\{B_2, C_2\}/(4 B_2/B)
	\oplus \bZ/2[B]/B^\infty\{\eta B_2, \eta_1^2\} \\
ko[3]/B^\infty &= \bZ_2[B]/B^\infty\{B_3, C_3\}/(8 B_3/B)
	\oplus \bZ/2[B]/B^\infty\{\eta B_3, \eta^2 B_3\} \\
ko[4]/B^\infty &= \bZ_2[B]/B^\infty\{B_4, C_4\}/(2 B_4/B)
	\oplus \bZ/2[B]/B^\infty\{\eta_4, \eta \eta_4\} \\
ko[5]/B^\infty &= \bZ_2[B]/B^\infty\{B_5, C_5\}/(8 B_5/B)
	\oplus \bZ/2[B]/B^\infty\{\eta B_5, \eta_1 \eta_4\} \\
ko[6]/B^\infty &= \bZ_2[B]/B^\infty\{B_6, C_6\}/(4 B_6/B)
	\oplus \bZ/2[B]/B^\infty\{\eta B_6, \eta^2 B_6\} \\
ko[7]/B^\infty &= \bZ_2[B]/B^\infty\{B_7, C_7\}/(8 B_7/B)
	\oplus \bZ/2[B]/B^\infty\{\eta B_7, \eta^2 B_7\} \,.
\end{align*}
Here $\bZ_2[B]/B^\infty = \bZ_2[B^{-1}]\{1/B\}$
and $\bZ/2[B]/B^\infty = \bZ/2[B^{-1}]\{1/B\}$.
\end{lemma}

\begin{proof}
The relations $B \cdot \eta_k = \eta B_k$ from~\cite{BR21}*{Def.~7.22(7)}
ensure that
$$
ko[k][1/B] = \bZ_2[B^{\pm1}] \{B_k, C_k\}
	\oplus \bZ/2[B^{\pm1}] \{\eta B_k, \eta^2 B_k\}
$$
for each $0 \le k \le 7$, from which the formulas for $ko[k]/B^\infty$
follow.  Note that $B \cdot D_k = d_k B_k$ in $ko[k]$ implies the relation
$d_k \cdot B_k/B = 0$ in~$ko[k]/B^\infty$.
\end{proof}

\begin{theorem} \label{thm:GammaBMtmfspseq}
At $p=2$, the local cohomology spectral sequence
$$
E_2^{s,t} = H^s_{(B,M)}(\pi_*(\tmf))_t
	\Longrightarrow_s \pi_{t-s}(\Gamma_{(B, M)} \tmf)
		\cong \pi_{t-s}(\Sigma^{-22} I_{\bZ_2}(\tmf))
$$
has $E_2$-term
$$
H^s_{(B,M)}(\pi_*(\tmf))_*
        \cong H^{s-1}_{(B)}(N_*) \otimes \bZ[M]/M^\infty
$$
where $H^*_{(B)}(N_*)$ is displayed in Figures~\ref{fig:GammaBN-ab}
and~\ref{fig:GammaBN-cd}.  There is no room for differentials, so $E_2
= E_\infty$.  There are hidden additive extensions
$$
d_{7-k} \cdot \nu_k \doteq C_k/B
$$
(multiplied by all negative powers of~$M$)
for $0 \le k \le 6$, indicated by vertical dashed red lines in the
figures.  Moreover, there are hidden $\eta$- and $\nu$-extensions
as shown by sloping dashed and dotted red lines in these figures.
\end{theorem}

\begin{proof}
See~\cite{BR21}*{\S9.2} for the $\eta$- and $\nu$-multiplications
in $\Gamma_B N_*$ that are not evident from the notation.  We note
in particular the relation $\nu^2 \nu_4 = \eta \epsilon_4 +
\eta_1 \bar\kappa^4$ in degree~$105$.  The dotted black lines show
$B$-multiplications.  The homotopy cofibre (and fibre) sequences
\begin{gather*}
\Sigma^{192} \Gamma_{(B, M)} \tmf
	\overset{M}\longto \Gamma_{(B, M)} \tmf
	\longto \Gamma_B N \\
I_{\bZ_2} N
	\longto I_{\bZ_2}(\tmf)
	\overset{M}\longto I_{\bZ_2}(\Sigma^{192} \tmf)
\end{gather*}
and the equivalence $\Gamma_{(B, M)} \tmf \simeq \Sigma^{-22}
I_{\bZ_2}(\tmf)$ imply an equivalence
$$
\Gamma_B N \simeq \Sigma^{171} I_{\bZ_2} N
$$
of $\tmf$-modules.
For each $0 \le k \le 6$ the group $\pi_{24k+3}(\Gamma_B N)
\cong \pi_{-24(7-k)}(I_{\bZ_2} N)$ sits in a short exact sequence
\begin{multline*}
0 \to \Ext_{\bZ_2}(\pi_{24(7-k)-1}(N), \bZ_2)
	\longto \pi_{-24(7-k)}(I_{\bZ_2} N) \\
	\longto \Hom_{\bZ_2}(\pi_{24(7-k)}(N), \bZ_2) \to 0 \,,
\end{multline*}
cf.~\eqref{eq:ExtIZpHom}.  Here $\pi_{24(7-k)-1}(N) = 0$ and
$$
\pi_{24(7-k)}(N) \cong \bZ_2\{B^{3(7-k)}, \dots, B^3 D_{6-k}, D_{7-k}\}
	\cong \bZ_2^{8-k} \,,
$$
so $\pi_{24k+3}(\Gamma_B N) \cong \bZ_2^{8-k}$ is $2$-torsion free.
In each case this implies that $\nu_k$, which generates a cyclic group
$\<\nu_k\>$ of order~$d_{7-k}$ in $\Gamma_B N_*$, lifts to a class of
infinite order in~$\pi_{24k+3}(\Gamma_B N)$.  Since $\nu_k$ is ($B$-
or) $B^2$-torsion in $\Gamma_B N_*$, its lift must also be ($B$- or)
$B^2$-torsion, and the only possibility is that $d_{7-k}$ times the lift
of $\nu_k$ is a $2$-adic unit times the image of $C_k/B \in N_*/B^\infty$.
Hence there is a hidden $2$-extension from $\frac12 d_{7-k} \nu_k$ in
Adams bidegree $(t-s,s) = (24k+3, 0)$ to $C_k/B$ in bidegree $(t-s,s) =
(24k+3,1)$, in the local cohomology spectral sequence
$$
E_2^{s,t} = H^s_{(B)}(N_*)_t
	\Longrightarrow_s \pi_{t-s}(\Gamma_B N)
	\cong \pi_{t-s}(\Sigma^{171} I_{\bZ_2} N) \,.
$$
This translates to a hidden $2$-extension from $\frac12 d_{7-k} \nu_k/M$ in
bidegree $(t-s,s) = (24k-190, 1)$ to $C_k/BM$ in bidegree $(t-s,s) =
(24k-190,2)$ in the local cohomology spectral sequence for $\Gamma_{(B,
M)} \tmf$, together with its multiples by all negative powers of~$M$.

There is no room for further hidden $2$-extensions, by elementary
$\eta$-, $\nu$- and $B$-linearity considerations.  The hidden $\eta$-
and $\nu$-extensions are present in $\pi_*(I_{\bZ_2} N)$, hence also
in $\pi_*(\Gamma_B N)$ and in $\pi_*(\Gamma_{(B, M)} \tmf)$, with the
appropriate degree shifts.
\end{proof}

\subsection{$(2, B, M)$-local cohomology of $\tmf$}
\label{subsec:2BMloccohtmf}

\begin{lemma}
$$
H^s_{(2,B,M)}(\pi_*(\tmf))
	\cong H^{s-1}_{(2,B)}(N_*) \otimes \bZ[M]/M^\infty \,.
$$
\end{lemma}

\begin{proof}
Replace $(B)$ by $(2, B)$ in the proof of Lemma~\ref{lem:HBMtmf}.
\end{proof}

\begin{proposition}
All $B$-power torsion in $N_*$ is $2$-power torsion, so
\begin{align*}
H^0_{(2, B)}(N_*) &= \Gamma_B N_* \\
H^1_{(2, B)}(N_*) &= \Gamma_2(N_*/B^\infty) \\
H^2_{(2, B)}(N_*) &= N_*/(2^\infty, B^\infty)
\end{align*}
with a short exact sequence
$$
0 \to (\Gamma_2 N_*)/B^\infty \longto \Gamma_2(N_*/B^\infty)
	\longto \Gamma_B(N_*/2^\infty) \to 0 \,.
$$
Here
\begin{multline*}
(\Gamma_2 N_*)/B^\infty = \bZ/2[B]/B^\infty
	\{\eta, \eta^2, \eta_1, \eta \eta_1,
	\eta B_2, \eta_1^2, \eta B_3, \eta^2 B_3, \\
	\eta_4, \eta \eta_4, \eta B_5, \eta_1 \eta_4,
	\eta B_6, \eta^2 B_6, \eta B_7, \eta^2 B_7\} \,,
\end{multline*}
and
$$
\Gamma_B(N_*/2^\infty) = \bigoplus_{k=1}^7 \bZ/d_k \{B_k/B\} \,,
$$
while
$$
N_*/(2^\infty, B^\infty) = \bigoplus_{k=0}^7
	\bZ_2[B]/(2^\infty, B^\infty) \{B_k/B, C_k\} \,,
$$
where $\bZ_2[B]/(2^\infty, B^\infty) = \bQ_2/\bZ_2[B^{-1}]\{1/B\}$.
\end{proposition}

\begin{proof}
This follows from the composite functor spectral sequence of
Subsection~\ref{subsec:composite} with $R_* = \pi_*(\tmf)/M$, first applied
with $x = B$ and $y = 2$, and thereafter with $x = 2$ and $y = B$.
The formulas for $(\Gamma_2 N_*)/B^\infty$ and $N_*/(2^\infty,
B^\infty)$ follow from the expressions for $N_*$ and $N_*/B^\infty$
in Theorem~\ref{thm:ZBmoduleN} and Lemma~\ref{lem:HBN}.
Only the summands $\bZ_2[D_k] \oplus \bZ_2[B]\{B_k\} \subset ko[k]$
of $N_*$ contribute to $\Gamma_B(N_*/2^\infty)$, where $B \cdot
D_k = d_k B_k$.  The $B$-power torsion in $ko[k]/2^\infty$ equals
$\bZ/d_k\{D_k/d_k\} \subset \bQ_2/\bZ_2\{D_k\}$, which we can rewrite
as $\bZ/d_k\{B_k/B\}$.
\end{proof}

\begin{theorem}
The local cohomology spectral sequence
$$
E_2^{s,t} = H^s_{(2, B, M)}(\pi_*(\tmf))_t
	\Longrightarrow_s \pi_{t-s}(\Gamma_{(2, B, M)} \tmf)
	\cong \pi_{t-s}(\Sigma^{-23} I(\tmf))
$$
has $E_2$-term
$$
H^s_{(2, B, M)}(\pi_*(\tmf))_*
	\cong H^{s-1}_{(2, B)}(N_*) \otimes \bZ[M]/M^\infty
$$
where $H^*_{(2, B)}(N_*)$ is displayed in Figures~\ref{fig:Gamma2BN-ab}
through~\ref{fig:Gamma2BN-gh}.  There are $d_2$-differentials
$$
d_2(\nu_k) \doteq C_k/d_{7-k}B
$$
(multiplied by all negative powers of~$M$) for $0 \le k \le 6$, indicated
by the green zigzag arrows increasing the filtration by~$2$.
There are no hidden additive extensions, but several hidden $\eta$- and
$\nu$-extensions, as shown by sloping dashed and dotted red lines in
these figures.
\end{theorem}

\begin{proof}
The homotopy (co-)fibre sequences
\begin{gather*}
\Sigma^{192} \Gamma_{(2, B, M)} \tmf
        \overset{M}\longto \Gamma_{(2, B, M)} \tmf
        \longto \Gamma_{(2, B)} N \\
I N \longto I(\tmf)
        \overset{M}\longto I(\Sigma^{192} \tmf)
\end{gather*}
and the equivalence $\Gamma_{(2, B, M)} \tmf \simeq \Sigma^{-23}
I(\tmf)$ imply an equivalence
$$
\Gamma_{(2, B)} N \simeq \Sigma^{170} I N
$$
of $\tmf$-modules.
For each $0 \le k \le 6$ the group
$$
\pi_{24k+3}(\Gamma_{(2,B)} N)
	\cong \pi_{-24(7-k)+1}(I N)
	\cong \Hom_{\bZ_2}(\pi_{24(7-k)-1}(N), \bQ_2/\bZ_2)
$$
is trivial, since $\pi_{24(7-k)-1}(N) = 0$.  Hence the group $\<\nu_k\>
= \bZ/d_{7-k} \{\nu_k\}$ in degree~$24k+3$ of $\Gamma_{(2, B)} N_*
= \Gamma_B N_*$ cannot survive to $E_\infty$ in the local cohomology
spectral sequence
$$
E_2^{s,t} = H^s_{(2, B)}(N_*)_t
	\Longrightarrow_s \pi_{t-s}(\Gamma_{(2, B)} N)
		\cong \pi_{t-s}(\Sigma^{170} I N) \,.
$$
This means that $d_2$ must act injectively on~$\<\nu_k\>$.  Since $\nu_k$
is ($B$- or) $B^2$-torsion, the only possible target in bidegree~$(t-s,s)
= (24k+2, 2)$ is $\bQ_2/\bZ_2 \{C_k/B\}$, and therefore $d_2$ maps
$\<\nu_k\>$ isomorphically to the subgroup of this target that is
generated by $C_k/d_{7-k} B$.

This translates to a $d_2$-differential in the local cohomology spectral
sequence for $\Gamma_{(2, B, M)} \tmf$ from $\nu_k/M$ in bidegree $(t-s,s)
= (24k-190, 1)$ to $C_k/d_{7-k} BM$ in bidegree $(t-s,s) = (24k-191,3)$
together with its multiples by all negative powers of~$M$.  The $2$-,
$\eta$- and $\nu$-extensions in $\pi_*(N)$ and $\pi_*(I N)$ are also
present in $\pi_*(\Gamma_{(2, B)} N)$ and in $\pi_*(\Gamma_{(2, B,
M)} \tmf)$, with the appropriate degree shifts, and those that increase
the local cohomology filtration degree are displayed with red lines.
\end{proof}

\begin{remark}
Let $\Theta N_* \subset \Gamma_B N_*$ be the part of the $B$-power
torsion in $N_*$ that is not in degrees~$* \equiv 3 \mod 24$, omitting
the subgroups $\<\nu_k\>$ for $0 \le k \le 6$ from Table~\ref{BtorsionN}.
This equals the kernel of the $d_2$-differential in the $(2, B)$-local
cohomology spectral sequence, which is also the image of the edge
homomorphism $\pi_*(\Gamma_{(2, B)} N) \to \Gamma_B N_*$.  Furthermore,
let $\Theta \pi_*(\tmf)$ be the part of $\Gamma_B \pi_*(\tmf)$ that is
not in degrees~$* \equiv 3 \mod 24$, which equals the image of the edge
homomorphism $\pi_*(\Gamma_{(2, B)} \tmf) \to \Gamma_B \pi_*(\tmf)$.

The image of the $2$-complete $\tmf$-Hurewicz homomorphism $\pi_*(S) \to
\pi_*(\tmf)$ is the direct sum of $\bZ$ in degree~$0$, the $8$-periodic
groups $\bZ/2\{\eta B^k\}$ and $\bZ/2\{\eta^2 B^k\}$ for $k\ge0$,
the group $\bZ/8\{\nu\}$ in degree~$3$, and the $192$-periodic groups
$\Theta \pi_*(\tmf) \cong \Theta N_* \otimes \bZ[M]$.
This was conjectured by Mahowald, was proved for degrees~$n \le 101$
and $n=125$ in~\cite{BR21}*{Thm.~11.89}, and has now been proved in
all degrees by Behrens, Mahowald and Quigley~\cite{BMQ}.
The three first summands of the $\tmf$-Hurewicz image are also
detected by the Adams $d$- and $e$-invariants.  To see that the fourth
summand is contained in the image from $\pi_*(\Gamma_{(2, B)} \tmf)$,
one can use the commutative diagram
$$
\xymatrix{
C^f_1 S \ar[r] \ar[d] & S \ar[d] \\
C^f_1 \tmf \ar[r] & \tmf
}
$$
and the equivalence $C^f_1 \tmf \simeq \Gamma_{(2, B)} \tmf$ from
Lemmas~\ref{lem:CfnCell} and~\ref{lem:CellGammaJ}.
\end{remark}

\subsection{$(B, H)$-local cohomology of $\tmf$}
\label{subsec:BHloccohtmf}

Let $p=3$ in this subsection and the next.
See Figure~\ref{fig:pitmfp3} for the mod~$3$ ($\tmf$-module) Adams
$E_\infty$-term for $\tmf$ in the range $0 \le t-s \le 72$, with all
hidden $\nu$-extensions shown.  There are no hidden $B$- or $H$-extensions
in this spectral sequence.

\begin{definition} \label{def:Nstarp3}
Let $N_* \subset \pi_*(\tmf)$ be the $\bZ_3[B]$-submodule generated by
all classes in degrees $0 \le * < 72$, and let $N = \tmf/H$.
\end{definition}

\begin{theorem}[\cite{BR21}*{Lem.~13.16}]
\label{thm:pitmfextensionp3}
The composite homomorphisms
\begin{gather*}
N_* \otimes \bZ[H] \longto \pi_*(\tmf) \otimes \pi_*(\tmf)
        \overset{\cdot}\longto \pi_*(\tmf) \\
N_* \subset \pi_*(\tmf) \longto \pi_*(N)
\end{gather*}
are isomorphisms.  Hence $\pi_*(\tmf)$ is a (split) extension of $\bZ_3[B,
H]$-modules
$$
0 \to \Gamma_B N_* \otimes \bZ[H] \longto \pi_*(\tmf)
        \longto \frac{N_*}{\Gamma_B N_*} \otimes \bZ[H] \to 0 \,.
$$
\end{theorem}

\begin{theorem}[\cite{BR21}*{Thm.~13.18}] \label{thm:ZBmoduleNp3}
As a $\bZ_3[B]$-module, $N_*$ is a split extension
$$
0 \to \Gamma_B N_* \longto N_* \longto \frac{N_*}{\Gamma_B N_*} \to 0 \,.
$$
The $B$-power torsion submodule $\Gamma_B N_*$ is given in
Table~\ref{BtorsionNp3}.  It is concentrated in degrees $3 \le * \le 40$,
and is annihilated by $(3, B)$.

The $B$-torsion free quotient of $N_*$ is the direct sum
$$
\frac{N_*}{\Gamma_B N_*} = \bigoplus_{k=0}^2 ko[k]
$$
of the following three $\bZ_3[B]$-modules,
with $ko[k]$ concentrated in degrees~$* \ge 24k$:
\begin{align*}
ko[0] &= \bZ_3[B]\{1, C\} \\
ko[1] &= \bZ_3\{D_1\} \oplus \bZ_3[B]\{B_1, C_1\} \\
ko[2] &= \bZ_3\{D_2\} \oplus \bZ_3[B]\{B_2, C_2\} \,.
\end{align*}
The $\bZ_3[B]$-module structures are specified by $B \cdot D_1 = 3 B_1$
and $B \cdot  D_2 = 3 B_2$.
\end{theorem}

\begin{table}
\caption{$B$-power torsion in $N_*$ at $p=3$
	\label{BtorsionNp3}}
\begin{tabular}{>{$}r<{$} | >{$}c<{$} | >{$}c<{$}}
n & \Gamma_B N_n & \text{generator} \\
\hline
3 & \bZ/3 & \nu \\
10 & \bZ/3 & \beta \\
13 & \bZ/3 & \nu \beta \\
20 & \bZ/3 & \beta^2 \\
27 & \bZ/3 & \nu_1 \\
30 & \bZ/3 & \beta^3 \\
37 & \bZ/3 & \nu_1 \beta \\
40 & \bZ/3 & \beta^4 \\
\end{tabular}
\end{table}

\begin{lemma} \label{lem:HBHtmf}
$$
H^s_{(B,H)}(\pi_*(\tmf))
	\cong H^{s-1}_{(B)}(N_*) \otimes \bZ[H]/H^\infty \,.
$$
\end{lemma}

\begin{proof}
Replace $M$ by $H$ in the proof of Lemma~\ref{lem:HBMtmf}.
\end{proof}

\begin{lemma} \label{lem:HBNp3}
$$
H^0_{(B)}(N_*) = \Gamma_B N_*
$$
and
$$
H^1_{(B)}(N_*) = N_*/B^\infty = \bigoplus_{k=0}^2 ko[k]/B^\infty
$$
is the direct sum of the following three $\bZ_3[B]$-modules,
with $ko[k]/B^\infty$ concentrated in degrees $* \le 24k+4$:
\begin{align*}
ko[0]/B^\infty &= \bZ_3[B]/B^\infty\{1, C\} \\
ko[1]/B^\infty &= \bZ_3[B]/B^\infty\{B_1, C_1\}/(3 B_1/B) \\
ko[2]/B^\infty &= \bZ_3[B]/B^\infty\{B_2, C_2\}/(3 B_2/B) \,.
\end{align*}
\end{lemma}

\begin{proof}
For $k \in \{1,2\}$, the relation $B \cdot D_k = 3 B_k$ in $ko[k]$
implies the relation $3 \cdot B_k/B = 0$ in~$ko[k]/B^\infty$.
\end{proof}

\begin{theorem}
At $p=3$, the local cohomology spectral sequence
$$
E_2^{s,t} = H^s_{(B,H)}(\pi_*(\tmf))_t
        \Longrightarrow_s \pi_{t-s}(\Gamma_{(B, H)} \tmf)
                \cong \pi_{t-s}(\Sigma^{-22} I_{\bZ_3}(\tmf))
$$
has $E_2$-term
$$
H^s_{(B,H)}(\pi_*(\tmf))_*
        \cong H^{s-1}_{(B)}(N_*) \otimes \bZ[H]/H^\infty
$$
where $H^*_{(B)}(N_*)$ is displayed in Figure~\ref{fig:GammaBNp3}.
There is no room for differentials, so $E_2 = E_\infty$.  There are
hidden additive extensions
$$
3 \cdot \nu \doteq C/B
\quad\text{and}\quad
3 \cdot \nu_1 \doteq C_1/B
$$
(multiplied by all negative powers of~$H$), indicated by vertical dashed
red lines in the figure.  Moreover, there is a hidden $\nu$-extension
from $\beta^2$ to $B_1/B$, shown by a sloping dotted red line.
\end{theorem}

\begin{proof}
We refer to~\cite{BR21}*{Prop.~13.14} for the relation $\nu
\nu_1 \doteq \beta^3$.  The equivalence $\Gamma_{(B, H)} \tmf \simeq
\Sigma^{-22} I_{\bZ_3}(\tmf)$ implies an equivalence $\Gamma_B N \simeq
\Sigma^{51} I_{\bZ_3} N$ of $\tmf$-modules.  For $k \in \{0,1\}$ the
group $\pi_{24k+3}(\Gamma_B N) \cong \pi_{-24(2-k)}(I_{\bZ_3} N)$ sits
in an extension
$$
0 \to \Ext_{\bZ_3}(\pi_{24(2-k)-1}(N), \bZ_3)
        \to \pi_{-24(2-k)}(I_{\bZ_3} N)
	\to \Hom_{\bZ_3}(\pi_{24(2-k)}(N), \bZ_3) \to 0 \,.
$$
Here $\pi_{24(2-k)-1}(N) = 0$ and
$$
\pi_{24(2-k)}(N) \cong \bZ_3\{B^{3(2-k)}, \dots, D_{2-k}\}
	\cong \bZ_3^{3-k} \,,
$$
so $\pi_{24k+3}(\Gamma_B N) \cong \bZ_3^{3-k}$.  Since $\nu_k$ is
$B$-torsion, there must be a $3$-extension in $\pi_{24k+3}(\Gamma_B N)$
from $\nu_k$ to a $3$-adic unit times $C_k/B$.  The $\nu$-extension
from $\nu_1$ to $\beta^3$ in $\pi_*(N)$ appears in dual form in
$\pi_*(I_{\bZ_3} N)$, $\pi_*(\Gamma_B N)$ and~$\pi_*(\Gamma_{(B, H)}
\tmf)$, and appears as a hidden $\nu$-extension from $\beta^2$ to $B_1/B$
in the second of these.
\end{proof}

\subsection{$(3, B, H)$-local cohomology of $\tmf$}
\label{subsec:3BHloccohtmf}

\begin{lemma}
$$
H^s_{(3,B,H)}(\pi_*(\tmf))
	\cong H^{s-1}_{(3,B)}(N_*) \otimes \bZ[H]/H^\infty \,.
$$
\end{lemma}

\begin{proof}
Replace $(B)$ by $(3, B)$ in the proof of Lemma~\ref{lem:HBHtmf}.
\end{proof}

\begin{proposition}
\begin{align*}
H^0_{(3, B)}(N_*) &= \Gamma_B N_* = \bZ/3\{\nu, \beta, \nu\beta,
	\beta^2, \nu_1, \beta^3, \nu_1\beta, \beta^4\} \\
H^1_{(3, B)}(N_*) &= \Gamma_3(N_*/B^\infty) = \Gamma_B(N_*/3^\infty)
	= \bZ/3 \{B_1/B, B_2/B\} \\
H^2_{(3, B)}(N_*) &= N_*/(3^\infty, B^\infty) = \bigoplus_{k=0}^2
	\bZ_3[B]/(3^\infty, B^\infty) \{B_k/B, C_k\} \,.
\end{align*}
\end{proposition}

\begin{proof}
This follows from the composite functor spectral sequence of
Subsection~\ref{subsec:composite} with $R_* = \pi_*(\tmf)/H$, first
applied with $x = 3$ and $y = B$, and thereafter with $x = B$ and $y = 3$.
The groups $(\Gamma_B N_*)/3^\infty$ and $(\Gamma_3 N_*)/B^\infty$ vanish.
The $3$-power torsion in $ko[k]/B^\infty$ is trivial for $k=0$, and
equals $\bZ/3\{B_k/B\}$ for $k \in \{1,2\}$.
\end{proof}

\begin{theorem}
The local cohomology spectral sequence
$$
E_2^{s,t} = H^s_{(3, B, H)}(\pi_*(\tmf))_t
	\Longrightarrow_s \pi_{t-s}(\Gamma_{(3, B, H)} \tmf)
	\cong \pi_{t-s}(\Sigma^{-23} I(\tmf))
$$
has $E_2$-term
$$
H^s_{(3, B, H)}(\pi_*(\tmf))_*
	\cong H^{s-1}_{(3, B)}(N_*) \otimes \bZ[H]/H^\infty
$$
where $H^*_{(3, B)}(N_*)$ is displayed in Figure~\ref{fig:Gamma3BN}.
There are $d_2$-differentials
$$
d_2(\nu_k) \doteq C_k/3 B
$$
(multiplied by all negative powers of~$M$) for $k \in \{0,1\}$, indicated
by the green zigzag arrows increasing the filtration by~$2$.
There are no hidden additive extensions, but hidden $\nu$-extensions
from $\beta^2$ to $B_1/B$ and from $B_2/B$ to $C_2/3B$, as shown by
sloping dashed red lines in this figure.
\end{theorem}

\begin{proof}
The equivalence $\Gamma_{(3, B, M)} \tmf \simeq \Sigma^{-23}
I(\tmf)$ implies an equivalence
$$
\Gamma_{(3, B)} N \simeq \Sigma^{50} I N
$$
of $\tmf$-modules.
For each $k \in \{0,1\}$ the group
$$
\pi_{24k+3}(\Gamma_{(3,B)} N)
	\cong \pi_{-24(2-k)+1}(I N)
	\cong \Hom_{\bZ_3}(\pi_{24(2-k)-1}(N), \bQ_3/\bZ_3)
$$
is trivial, since $\pi_{24(2-k)-1}(N) = 0$.  Hence the group $\<\nu_k\>
\cong \bZ/3$ in degree~$24k+3$ of $\Gamma_{(3, B)} N_* = \Gamma_B N_*$
cannot survive to $E_\infty$ in the local cohomology spectral sequence
$$
E_2^{s,t} = H^s_{(3, B)}(N_*)_t
	\Longrightarrow_s \pi_{t-s}(\Gamma_{(3, B)} N)
		\cong \pi_{t-s}(\Sigma^{50} I N) \,.
$$
Since $\nu_k$ is $B$-torsion, it follows that $d_2$ maps $\<\nu_k\>$
isomorphically to the subgroup of $\bQ_3/\bZ_3\{C_k/B\}$ that is generated
by $C_k/3 B$.  This translates to a $d_2$-differential in the local
cohomology spectral sequence for $\Gamma_{(3, B, H)} \tmf$ from $\nu_k/H$
in bidegree $(t-s,s) = (24k-70, 1)$ to $C_k/3 B H$ in bidegree $(t-s,s) =
(24k-71,3)$ together with its multiples by all negative powers of~$H$.
The $\nu$-extensions in $\pi_*(N)$ and $\pi_*(I N)$ are also present
in $\pi_*(\Gamma_{(3, B)} N)$ and in $\pi_*(\Gamma_{(3, B, H)} \tmf)$,
with the appropriate degree shifts, and those that increase the local
cohomology filtration degree are displayed in red.
\end{proof}

\begin{remark}
Let $\Theta N_* \subset \Gamma_B N_*$ be the part of the $B$-power
torsion in $N_*$ that is not in degrees~$* \equiv 3 \mod 24$, omitting
$\bZ/3\{\nu\}$ and $\bZ/3\{\nu_1\}$ from Table~\ref{BtorsionNp3}.
Likewise, let $\Theta \pi_*(\tmf)$ be the part of $\Gamma_B \pi_*(\tmf)$
that is not in degrees~$* \equiv 3 \mod 24$, which equals the image
of the $(3, B)$-local cohomology spectral sequence edge homomorphism
$\pi_*(\Gamma_{(3, B)} \tmf) \to \Gamma_B \pi_*(\tmf)$.

Mahowald conjectured that the image of the $3$-complete $\tmf$-Hurewicz
homomorphism $\pi_*(S) \to \pi_*(\tmf)$ is the direct sum of $\bZ$ in
degree~$0$, the group $\bZ/3\{\nu\}$ in degree~$3$, and the $72$-periodic
groups $\Theta \pi_*(\tmf) \cong \Theta N_* \otimes \bZ[H]$.  This was
proved for degrees~$n < 154$ in~\cite{BR21}*{Prop.~13.29}.
\end{remark}

\subsection{Charts}
\label{subsec:charts}

Figure~\ref{fig:A-duality} shows $N_* \cong \pi_*(N)$, $\pi_*(\Gamma_B
N)$ and $\pi_*(\Sigma^{171} I_{\bZ_2} N)$ in the range $-9 \le * \le
72$, visible as three horizontal wedges.  The vertical direction has no
intrinsic meaning.  Circled numbers represent finite cyclic groups of
that order, squares represent infinite cyclic groups, and each ellipse
containing `$22$' represents a Klein Vierergruppe.  Horizontal dashed lines
show multiplication by~$B$, which extends indefinitely to the right in the
upper wedge, and indefinitely to the left in the middle and lower wedges.
Thick vertical lines indicate additive extensions, by which a square and a
circle combine to an infinite cyclic group.  The passage from the upper to
the middle wedge is given by taking the homotopy fibre of the localisation
map $\gamma \: N \to N[1/B]$, leaving the $B$-power torsion (shown in
red) in place and replacing copies of $\bZ[B]$ or $\bZ/2[B]$ (shown in
blue) by desuspended copies of $\bZ[B]/B^\infty$ or $\bZ/2[B]/B^\infty$
(shown in green), respectively.  The passage from the upper to the lower
wedge takes the torsion-free part of $\pi_*(N)$ to its linear dual in
degree~$171-*$, and takes the torsion in $\pi_*(N)$ to its Pontryagin
dual in degree~$170-*$.  The local cohomology theorem asserts that the
middle and lower wedges are isomorphic.  Note in particular how this is
achieved in degrees $* \equiv -1, 3 \mod 24$.

Figure~\ref{fig:pitmf} shows the $E_\infty$-term of the mod~$2$
Adams spectral sequence
$$
E_2^{s,t} = \Ext_{A(2)}^{s,t}(\bF_2, \bF_2)
	\Longrightarrow_s \pi_{t-s}(\tmf)
$$
for $0 \le t-s \le 192$, together with all hidden $2$-, $\eta$-
and $\nu$-extensions in this range.  There is also a more
subtle multiplicative relation in degree~$105$, cf.~the proof of
Theorem~\ref{thm:GammaBMtmfspseq}.  The vertical coordinate gives the
Adams filtration~$s$.  The $B$-power torsion classes are shown in red, and
selected product factorisations in terms of the algebra indecomposables
in $\pi_*(\tmf)$ are shown.  The $B$-periodic classes are shown in
black, and usually only the $\bZ[B]$-module generators are labelled.
The $\bZ[B]$-submodule generated by the classes in degrees $0 \le * < 192$
defines the basic block $N_*$, which repeats $M$-periodically, so that
$\pi_*(\tmf) \cong N_* \otimes \bZ[M]$.  Note how the additive structure
of $N_*$ also appears in the upper wedge of Figure~\ref{fig:A-duality}.

Figures~\ref{fig:GammaBN-ab} and~\ref{fig:GammaBN-cd} show the collapsing
local cohomology spectral sequence for $\Gamma_B N$, in the range
$-20 \le t-s \le 172$, broken into four sections.  In each section the
lower row shows $H^0_{(B)}(N_*) = \Gamma_B N_*$, while the upper row
shows $H^1_{(B)}(N_*) = N_*/B^\infty$ shifted one unit to the left.
Multiplication by~$2$, $\eta$, $\nu$ and~$B$ is shown by lines increasing
the topological degree by $0$, $1$, $3$ and~$8$, respectively.  The dotted
arrows coming from the left indicate classes that are infinitely divisible
by~$B$.  Multiplications that connect the lower and upper rows increase
the local cohomology filtration, hence are hidden, and are shown in red.
The additive extensions in degrees~$* \equiv 3 \mod 24$ are also carried
over to the central wedge of Figure~\ref{fig:A-duality}.

It may be easiest to study these charts by starting in high degrees
and descending from there.  The top terms in $N_*$ that are not
$B$-divisible are $\bZ_2\{C_7\}$, $\bZ/2\{\eta^2 B_7\}$, $\bZ/2\{\eta
B_7\}$ and~$\bZ/8\{B_7\}$ in degrees~$180$ and $178$ to~$176$,
while the topmost $B$-power torsion in $N_*$ is $\bZ/2\{\nu
\nu_6 \kappa\}$ in degree~$164$.  These contribute $\bZ_2\{C_7/B\}$,
$\bZ/2\{\eta^2 B_7/B\}$, $\bZ/2\{\eta B_7/B\}$ and~$\bZ/8\{B_7/B\}$
to $N_*/B^\infty$ in internal degrees~$172$ and $170$ to~$168$, shifted
to topological degrees~$171$ and $169$ to~$167$ in $\pi_*(\Gamma_B N)$,
together with $\bZ/2\{\nu \nu_6 \kappa\}$ in degree~$164$ of $\Gamma_B
N_*$ and $\pi_*(\Gamma_B N)$.  In the Anderson dual, the bottom term
$\bZ_2\{1\}$ of $\pi_*(N)$ contributes a copy of $\bZ_2$ in degree~$171$
of $\pi_*(\Sigma^{171} I_{\bZ_2} N)$, while the terms $\bZ/2\{\eta\}$,
$\bZ/2\{\eta^2\}$, $\bZ/8\{\nu\}$ and~$\bZ/2\{\nu^2\}$ contribute the
groups $\bZ/2$, $\bZ/2$, $\bZ/8$ and~$\bZ/2$ in degrees~$169$ to~$167$
and~$164$.  The duality theorem matches these groups isomorphically.

Figures~\ref{fig:Gamma2BN-ab} through~\ref{fig:Gamma2BN-gh} show the
local cohomology spectral sequence for $\Gamma_{(2, B)} N$, in the range
$-20 \le t-s \le 172$.  In each figure the lower row shows $H^0_{(2,
B)}(N_*) = \Gamma_B N_*$, the middle row shows $H^1_{(2, B)}(N_*) =
\Gamma_2(N_*/B^\infty)$ shifted one unit to the left, and the upper row
shows $H^2_{(2, B)}(N_*) = N_*/(2^\infty, B^\infty)$ shifted two units to
the left.  There are nonzero $d_2$-differentials from topological degrees
$* \equiv 3 \mod 24$, leaving $E_3 = E_\infty$.  Multiplications by~$2$,
$\eta$, $\nu$ and~$B$, infinitely $B$-divisible towers, and hidden
extensions, are shown as in the previous figures.  Note how the abutment
$\pi_*(\Gamma_{(2, B)} N)$ is Pontryagin dual to $\pi_{170-*}(N)$.

The charts for $p=3$ follow the same conventions as for $p=2$.

\begin{landscape}

\begin{figure}
\vskip 0.60cm
\includegraphics[scale=0.305]{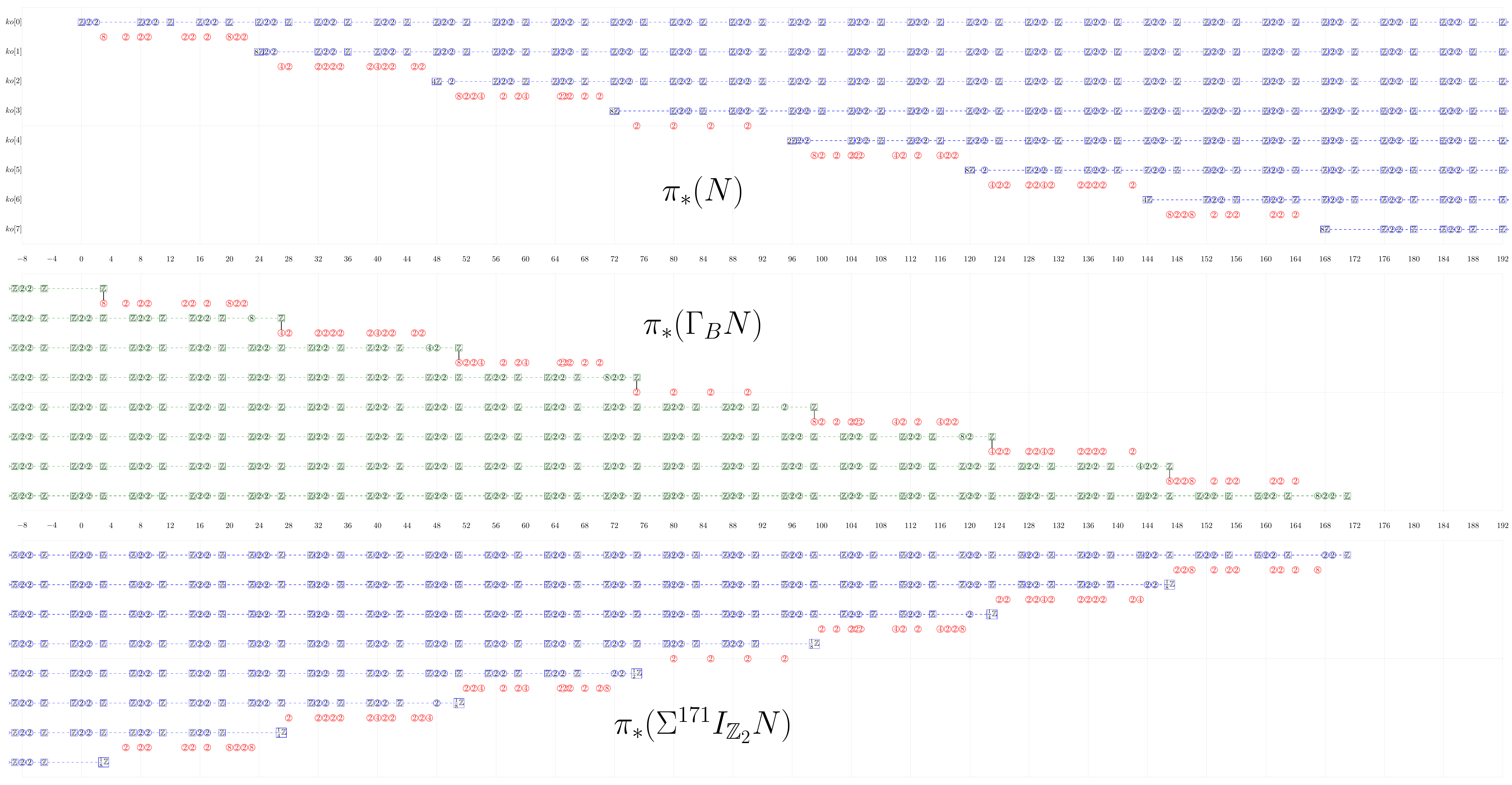}
\caption{Homotopy of the basic block $N$ of $\tmf$ at $p=2$, of its $B$-local
	cohomology, and of its shifted Anderson dual
	\label{fig:A-duality}}
\end{figure}

\begin{figure}
\includegraphics[scale=0.50]{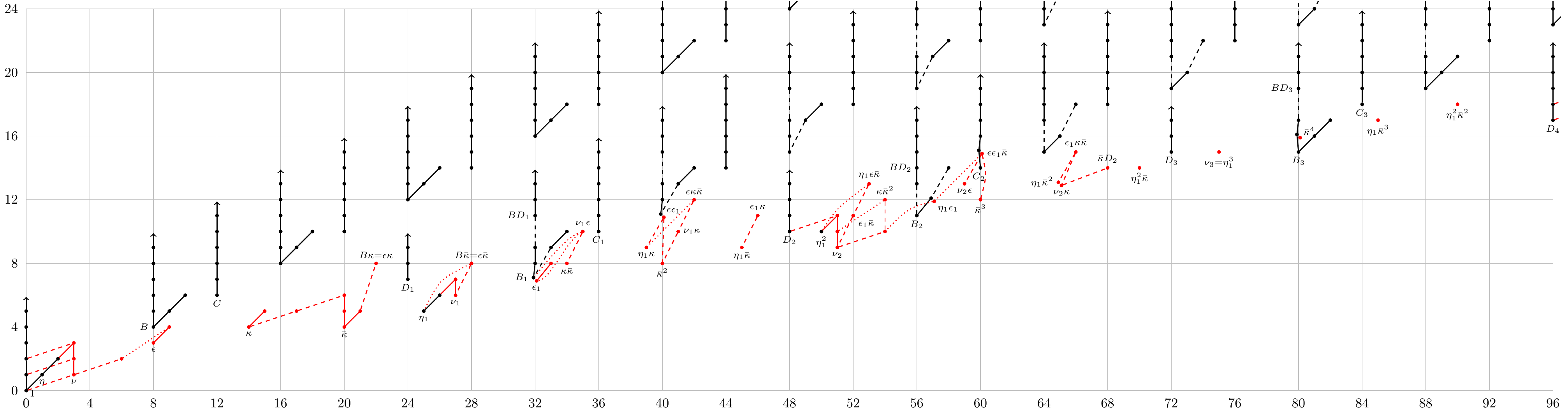}
\vskip 0.87cm
\includegraphics[scale=0.50]{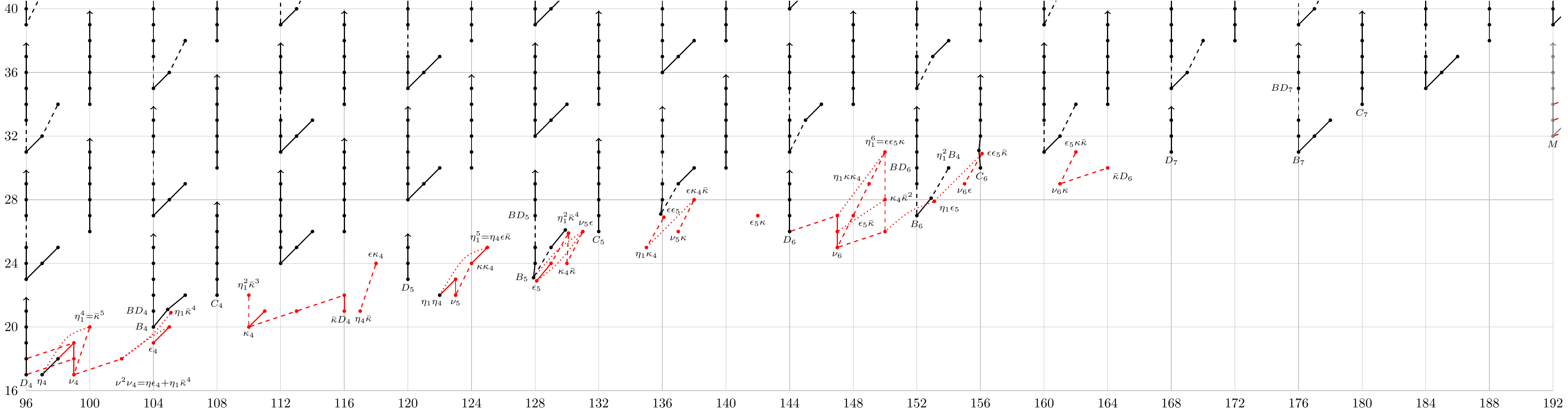}
\caption{$\pi_*(\tmf)$ at $p=2$ for $0 \le * \le 192$
	\label{fig:pitmf}}
\end{figure}

\begin{figure}
\includegraphics[scale=0.38]{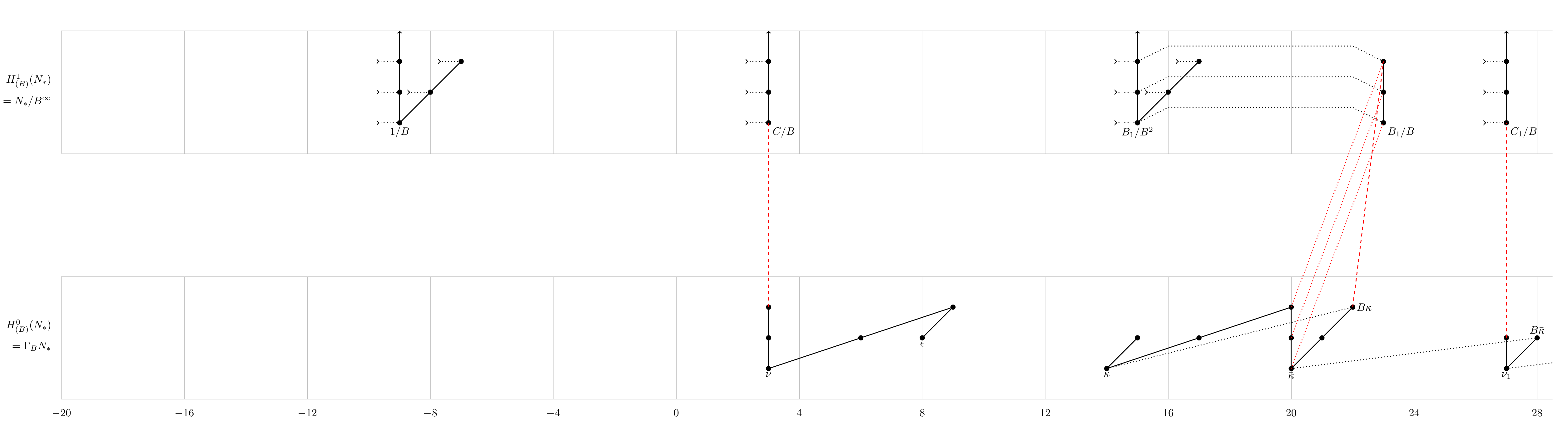}
\vskip 0.9cm
\includegraphics[scale=0.38]{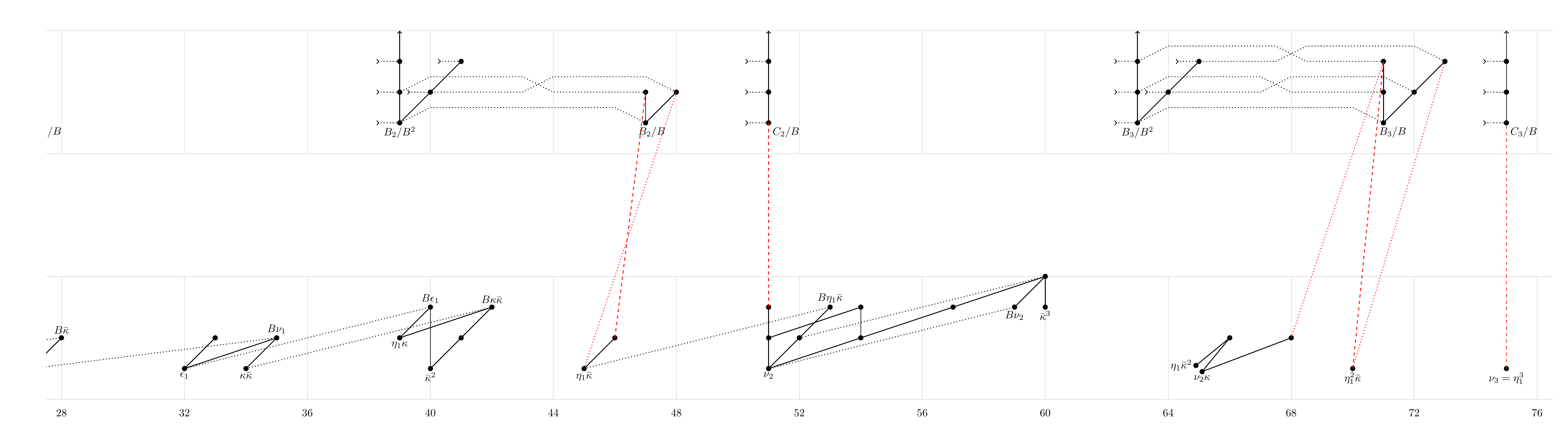}
\caption{$E_2^{s,t} = H^s_{(B)}(N_*)_t
	\Longrightarrow_s \pi_{t-s}(\Gamma_B N)$ at $p=2$
	\label{fig:GammaBN-ab}}
\end{figure}

\begin{figure}
\includegraphics[scale=0.38]{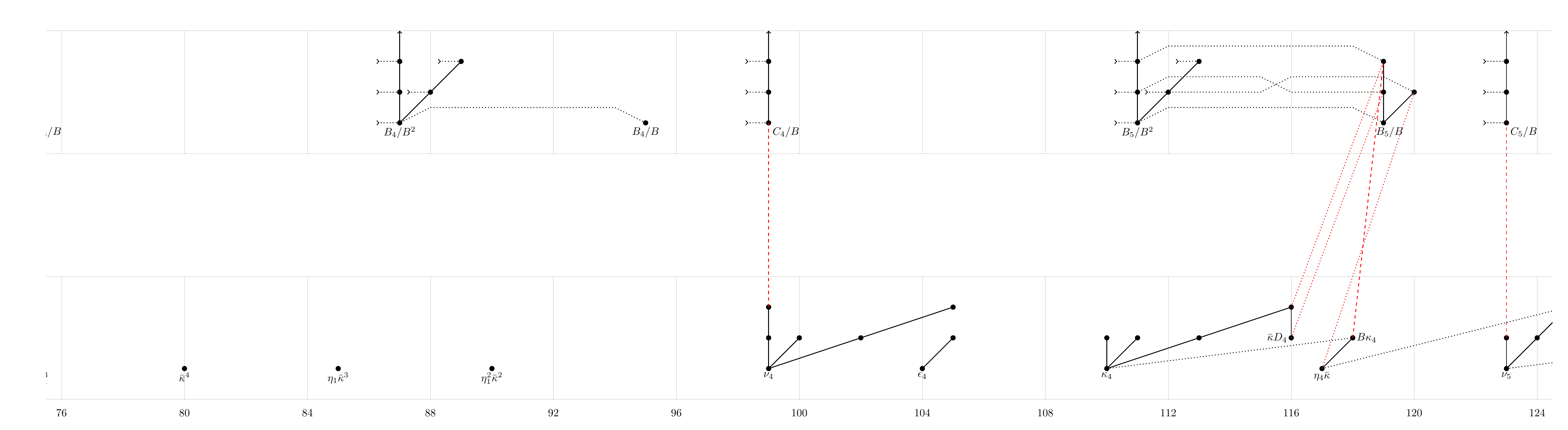}
\vskip 0.9cm
\includegraphics[scale=0.38]{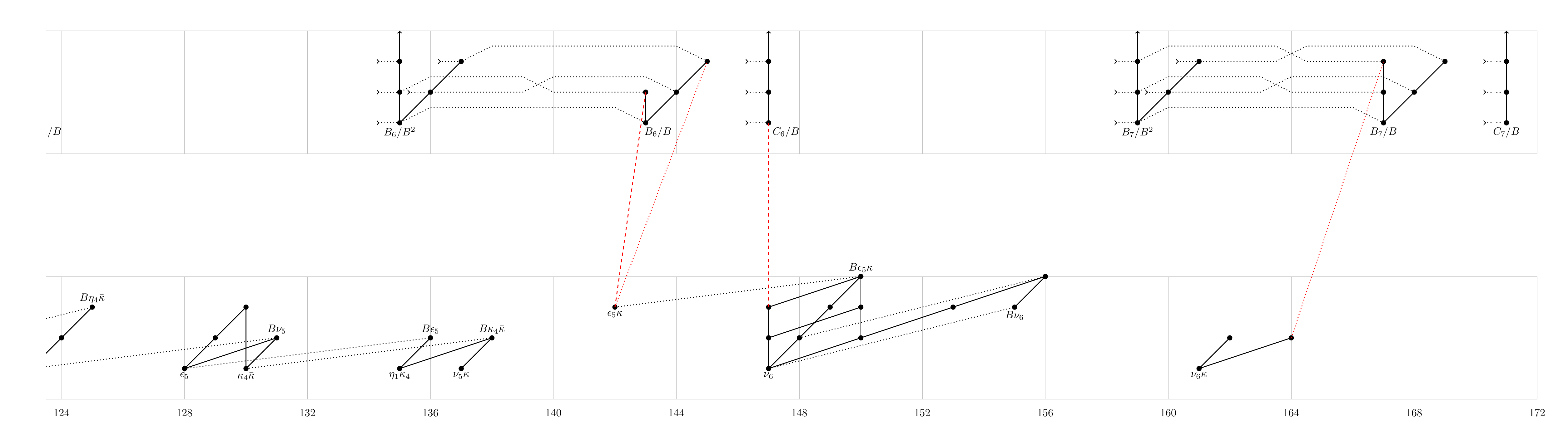}
\caption{$E_2^{s,t} = H^s_{(B)}(N_*)_t
	\Longrightarrow_s \pi_{t-s}(\Gamma_B N)$ at $p=2$
	\label{fig:GammaBN-cd}}
\end{figure}

\begin{figure}
\vskip 1.7cm
\includegraphics[scale=0.40]{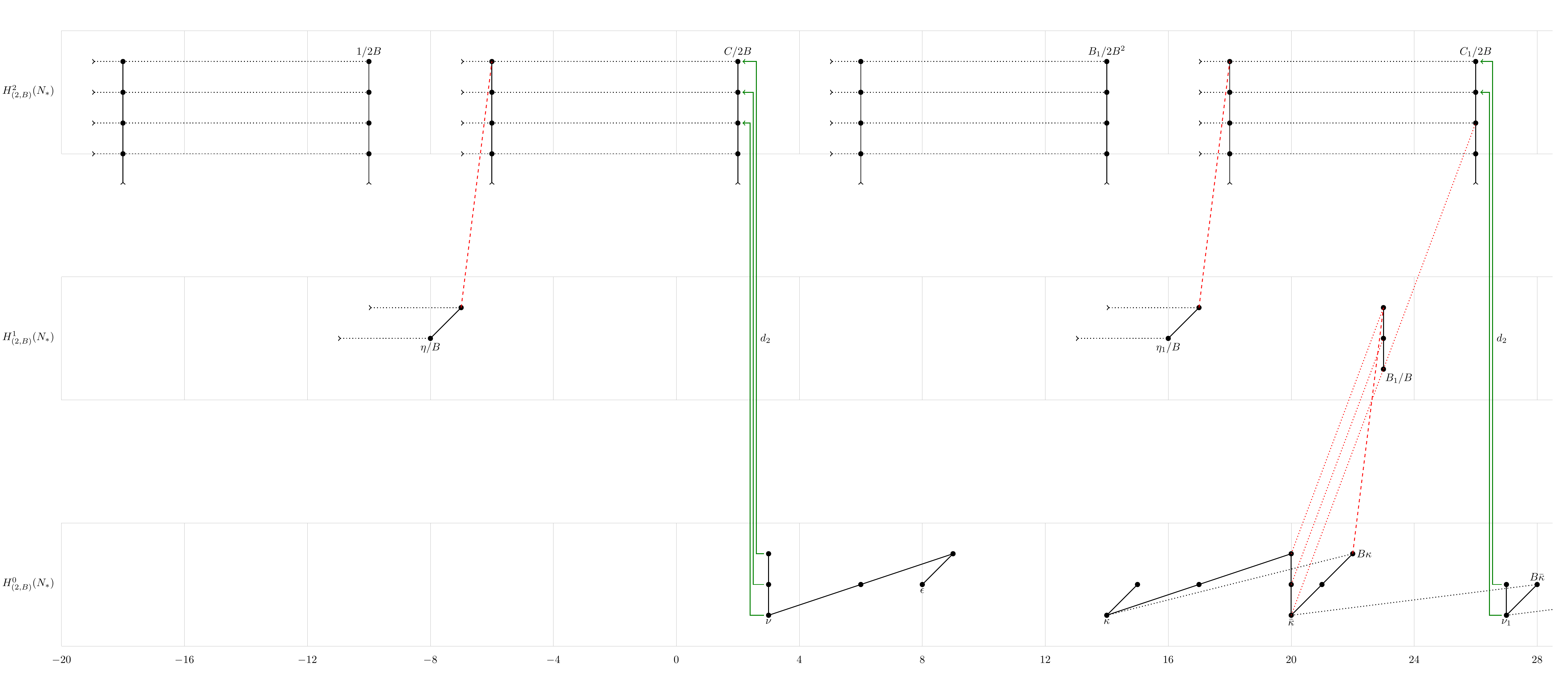}
\caption{$E_2^{s,t} = H^s_{(2, B)}(N_*)_t
	\Longrightarrow_s \pi_{t-s}(\Gamma_{(2, B)} N)$
	\label{fig:Gamma2BN-ab}}
\end{figure}

\begin{figure}
\vskip 1.7cm
\includegraphics[scale=0.40]{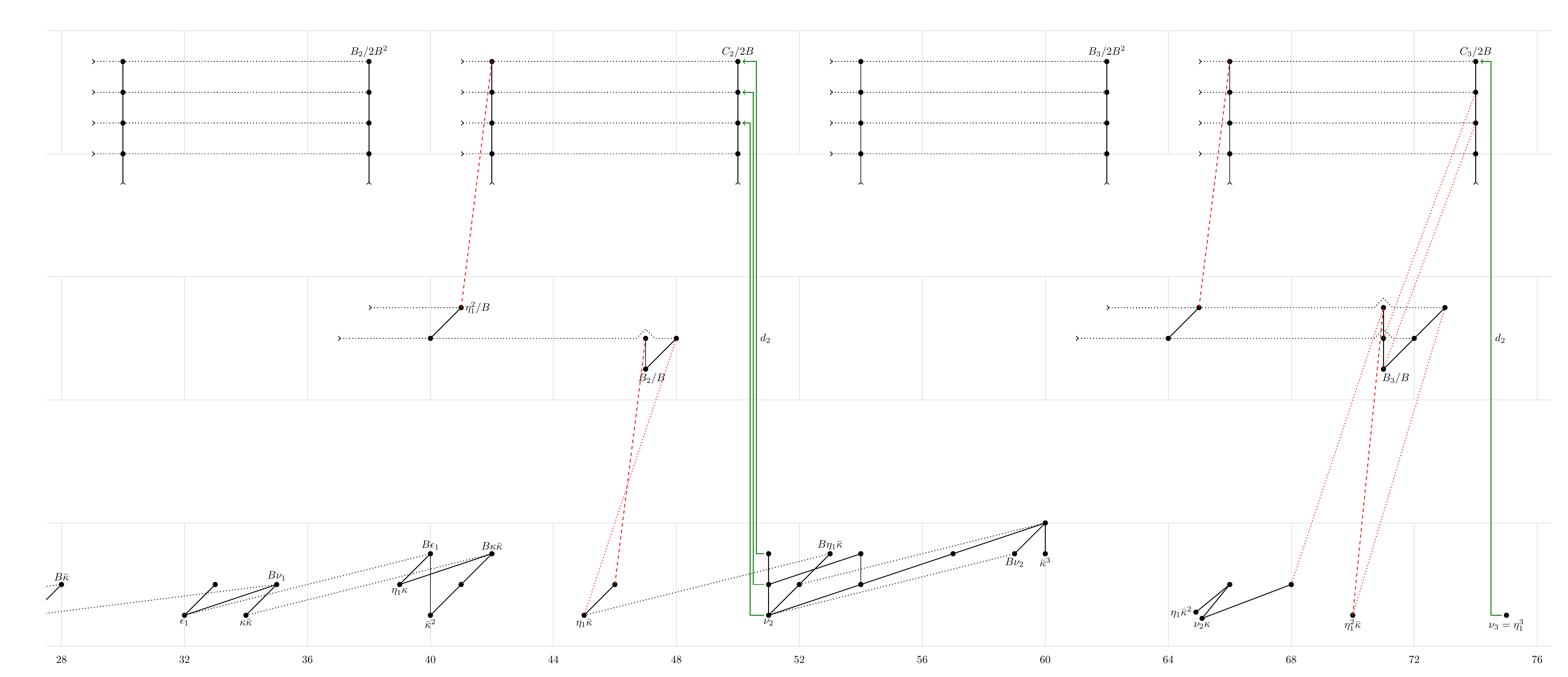}
\caption{$E_2^{s,t} = H^s_{(2, B)}(N_*)_t
	\Longrightarrow_s \pi_{t-s}(\Gamma_{(2, B)} N)$
	\label{fig:Gamma2BN-cd}}
\end{figure}

\begin{figure}
\vskip 1.7cm
\includegraphics[scale=0.40]{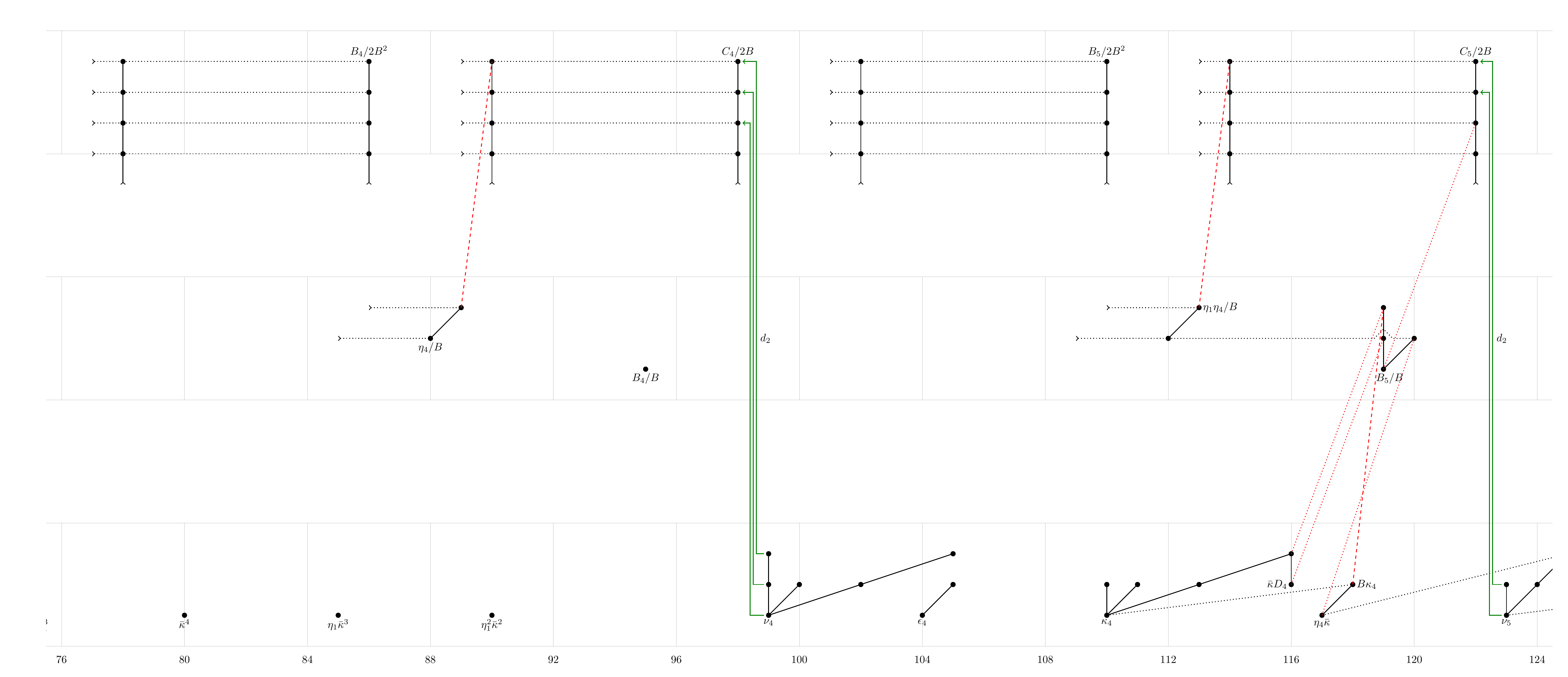}
\caption{$E_2^{s,t} = H^s_{(2, B)}(N_*)_t
	\Longrightarrow_s \pi_{t-s}(\Gamma_{(2, B)} N)$
	\label{fig:Gamma2BN-ef}}
\end{figure}

\begin{figure}
\vskip 1.7cm
\includegraphics[scale=0.40]{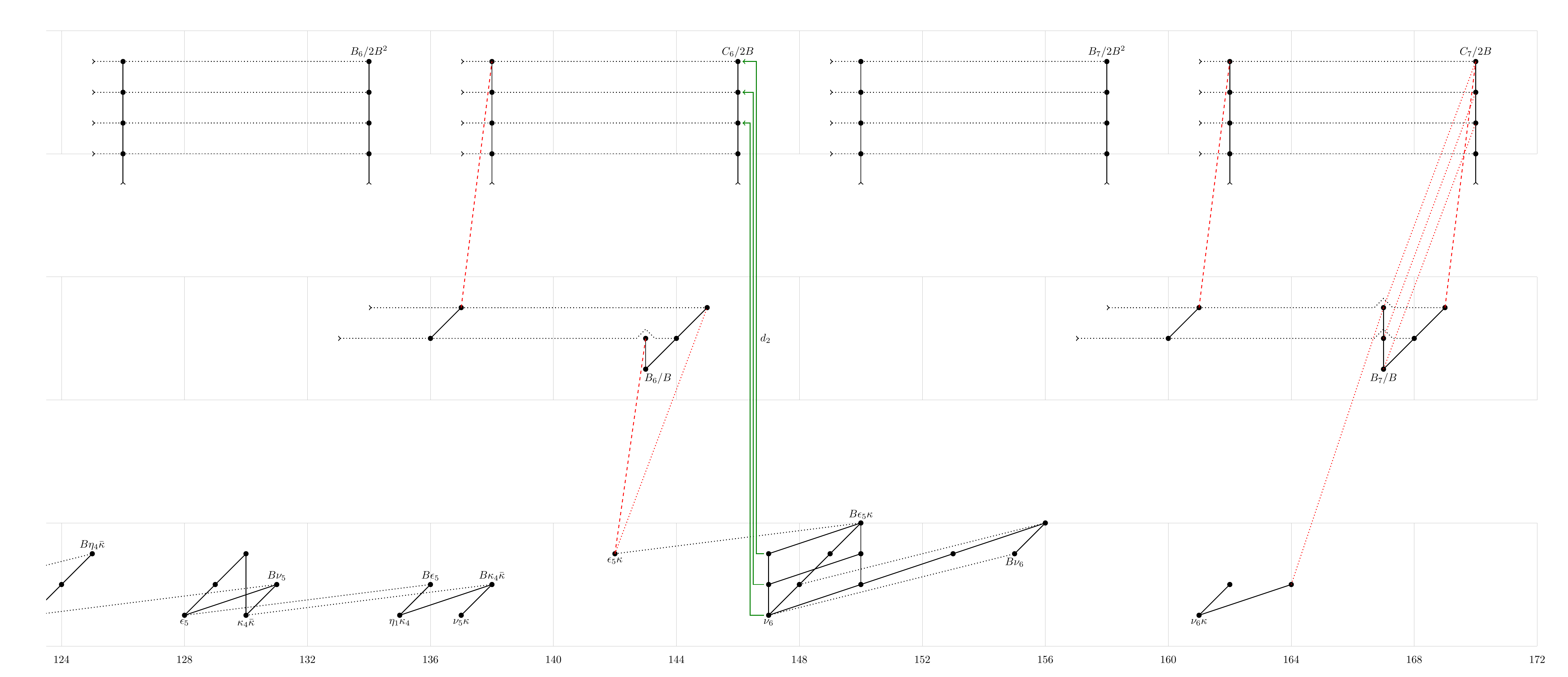}
\caption{$E_2^{s,t} = H^s_{(2, B)}(N_*)_t
	\Longrightarrow_s \pi_{t-s}(\Gamma_{(2, B)} N)$
	\label{fig:Gamma2BN-gh}}
\end{figure}

\begin{figure}
\includegraphics[scale=0.655]{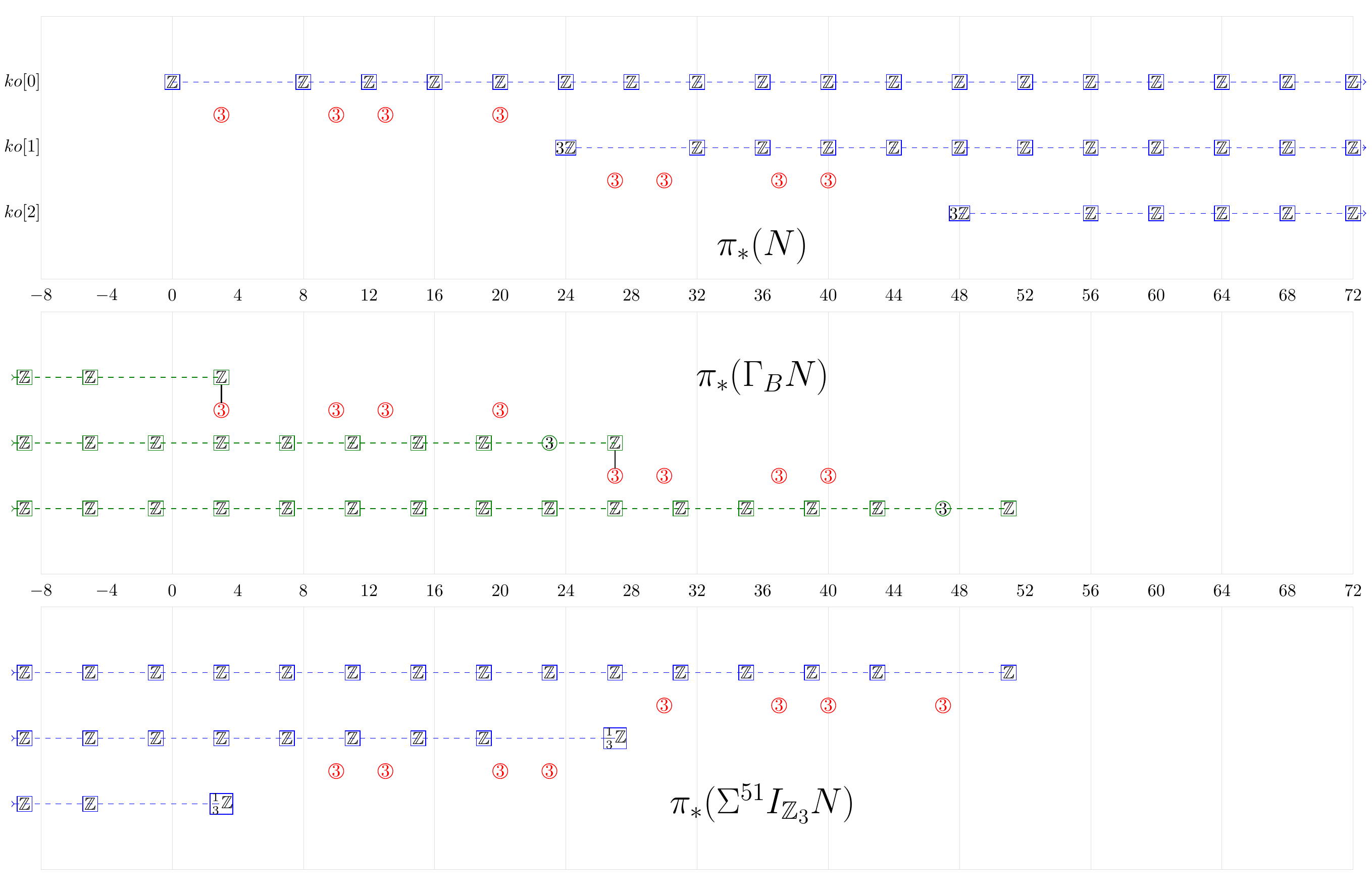}
\caption{Homotopy of the basic block $N$ of $\tmf$ at $p=3$, of its $B$-local
	cohomology, and of its shifted Anderson dual
	\label{fig:Ap3-duality}}
\end{figure}

\begin{figure}
\includegraphics[scale=0.28]{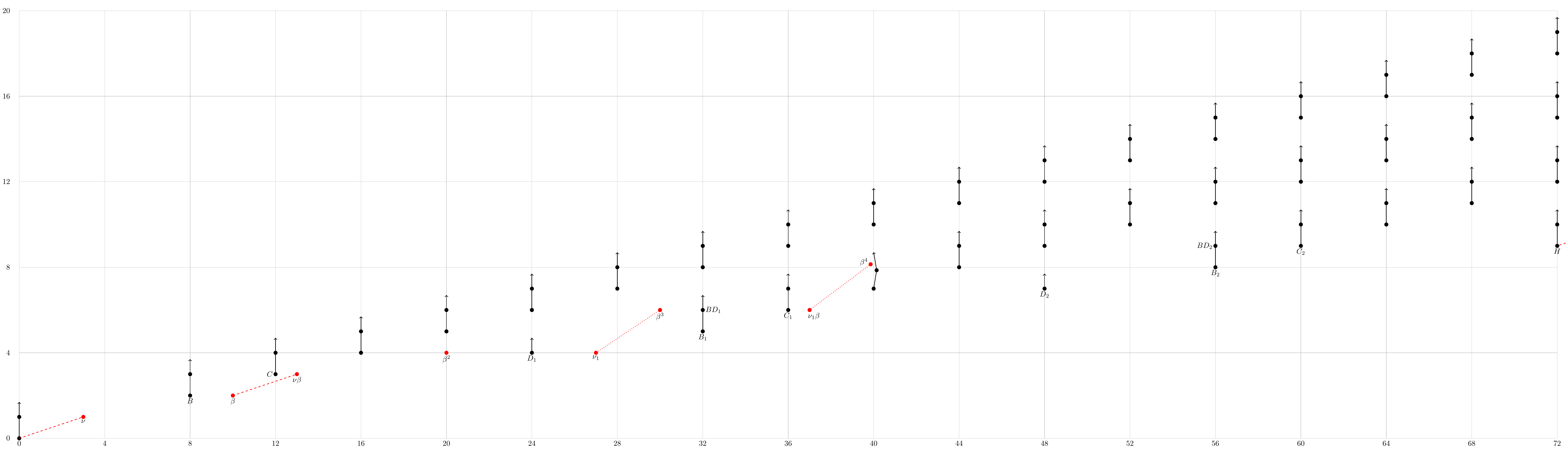}
\caption{$\pi_*(\tmf)$ at $p=3$ for $0 \le * \le 72$
	\label{fig:pitmfp3}}
\end{figure}
\begin{figure}
\vskip 0.3cm
\includegraphics[scale=0.30]{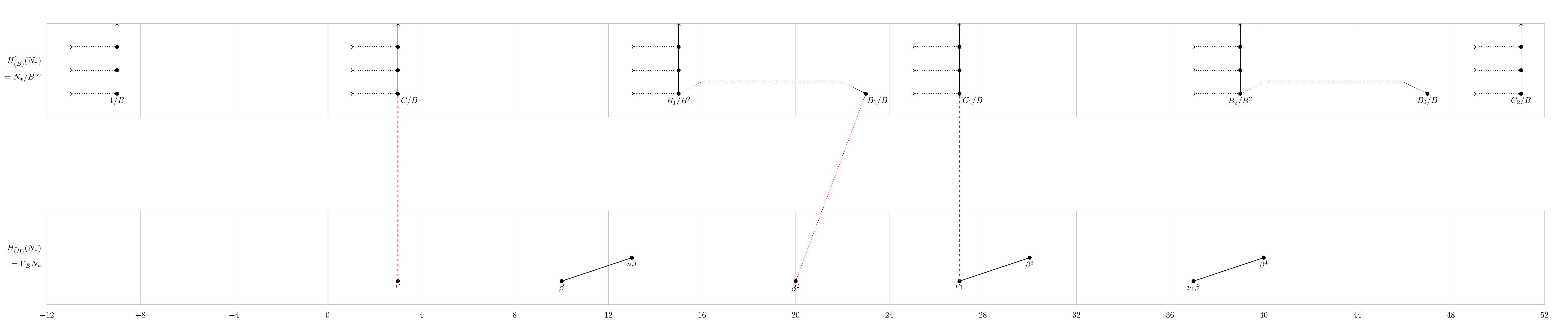}
\caption{$E_2^{s,t} = H^s_{(B)}(N_*)_t
	\Longrightarrow_s \pi_{t-s}(\Gamma_B N)$ at $p=3$
	\label{fig:GammaBNp3}}
\end{figure}

\begin{figure}
\vskip 3.0cm
\includegraphics[scale=0.27]{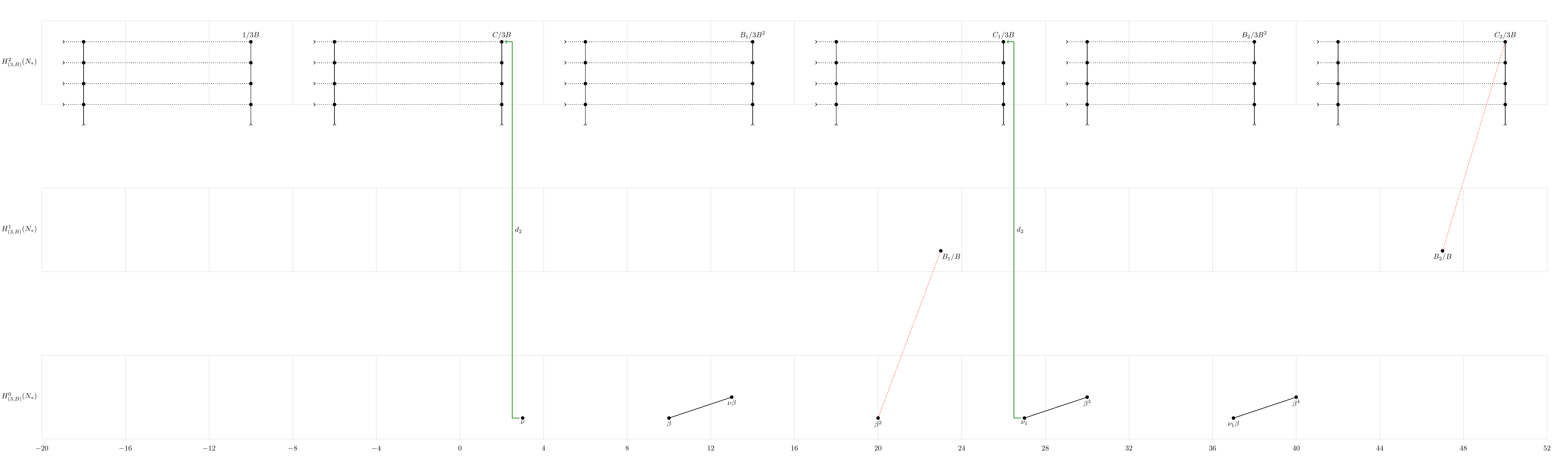}
\caption{$E_2^{s,t} = H^s_{(3, B)}(N_*)_t
	\Longrightarrow_s \pi_{t-s}(\Gamma_{(3, B)} N)$
	\label{fig:Gamma3BN}}
\end{figure}

\end{landscape}

\end{document}